\documentclass[11pt,reqno]{amsart}

\usepackage{amsmath,amsthm,amssymb,amsfonts}
\usepackage{bm}
\usepackage{natbib}
\usepackage{bbm}
\usepackage{graphicx}
\usepackage{multirow}
\usepackage{here}
\usepackage{fullpage}
\usepackage{url}
\usepackage{color}
\usepackage{xcolor}
\usepackage{mathrsfs}
\usepackage{enumitem}
\usepackage{comment}

\makeatletter

\@addtoreset{equation}{section}
\makeatletter

\newcommand{\vertiii}[1]{{\left\vert\kern-0.25ex\left\vert\kern-0.25ex\left\vert #1 \right\vert\kern-0.25ex\right\vert\kern-0.25ex\right\vert}}

\newtheorem{theorem}{Theorem}[section]
\newtheorem{corollary}{Corollary}[section]
\newtheorem{lemma}{Lemma}[section]
\newtheorem{proposition}{Proposition}[section]
\newtheorem{assumption}{Assumption}[section]
\theoremstyle{definition}

\DeclareMathOperator{\Var}{Var}

\DeclareMathOperator{\E}{E}
\DeclareMathOperator{\Prob}{P}

\newcommand{\R}{\mathbb{R}}

\renewcommand{\tilde}{\widetilde}
\renewcommand{\hat}{\widehat}

\allowdisplaybreaks

\begin{document}

\title[Series regression for spatial data]{Series ridge regression for spatial data on $\mathbb{R}^d$}
\thanks{D. Kurisu is partially supported by JSPS KAKENHI Grant Number 23K12456. We thank Peter Robinson, Taisuke Otsu, and Imma Valentina Curato for their helpful comments and discussion.
} 

\author[D. Kurisu]{Daisuke Kurisu}
\author[Y. Matsuda]{Yasumasa Matsuda}

\date{First version: November 9, 2023. This version: \today}

\address[D. Kurisu]{Center for Spatial Information Science, The University of Tokyo\\
5-1-5, Kashiwanoha, Kashiwa-shi, Chiba 277-8568, Japan.
}
\email{daisukekurisu@csis.u-tokyo.ac.jp}
\address[Y. Matsuda]{Graduate School of Economics and Management, Tohoku University\\
Sendai 980-8576, Japan.
}
\email{yasumasa.matsuda.a4@tohoku.ac.jp}

\begin{abstract}
This paper develops a general asymptotic theory of series estimators for spatial data collected at irregularly spaced locations within a sampling region $R_n \subset \mathbb{R}^d$. We employ a stochastic sampling design that can flexibly generate irregularly spaced sampling sites, encompassing both pure increasing and mixed increasing domain frameworks. Specifically, we focus on a spatial trend regression model and a nonparametric regression model with spatially dependent covariates. For these models, we investigate $L^2$-penalized series estimation of the trend and regression functions. We establish uniform and $L^2$ convergence rates and multivariate central limit theorems for general series estimators as main results. Additionally, we show that spline and wavelet series estimators achieve optimal uniform and $L^2$ convergence rates and propose methods for constructing confidence intervals for these estimators. Finally, we demonstrate that our dependence structure conditions on the underlying spatial processes cover a broad class of random fields, including L\'evy-driven continuous autoregressive and moving average random fields.
\end{abstract}


\maketitle

\section{Introduction}

Spatial data analysis plays an important role in many research fields, such as climate studies, ecology, hydrology, and seismology. There are many textbooks and monographs devoted to modeling and inference of spatial data, see, e.g. \cite{St99}, \cite{GaGu10}, \cite{BaCaGe14}, and \cite{Cr15}, among others. 

The goal of this paper is to develop a general asymptotic theory for series estimators with a ridge penalty (i.e., an $L^2$-penalty) under spatial dependence. Specifically, we consider two nonparametric regression models. The first model is a spatial trend regression. In this model, we contemplate a situation where a spatial process $\bm{Y}=\{Y(\bm{s}):\bm{s} \in R_n\}$ is discretely observed at irregularly spaced sampling sites $\{\bm{s}_{n,i}\}_{i=1}^n$ over the sampling region $R_n \subset \mathbb{R}^d$ and investigate nonparametric estimation of the spatial trend function using observations $\{Y(\bm{s}_{n,i}),\bm{s}_{n,i}\}_{i=1}^n$. The second model is a spatial regression model with heteroscedastic error terms. Specifically, we consider a scenario where a spatial process $\bm{Y}$ and a $p$-variate spatial process $\bm{X}=\{\bm{X}(\bm{s})=(X_1(\bm{s}),\dots, X_p(\bm{s})):\bm{s} \in R_n\}$ are discretely observed at irregularly spaced sampling sites. Using the observations $\{Y(\bm{s}_{n,i}),\bm{X}(\bm{s}_{n,i}),\bm{s}_{n,i}\}_{i=1}^n$, we consider nonparametric estimation of the spatial regression function of $\bm{Y}$, where the regression function depends on both the location and the covariate process $\bm{X}$. We will provide detailed definitions for the first model in Section \ref{Sec:reg-trend} and for the second model in Section \ref{Sec:reg-covariate} of the supplementary material, respectively. 

For these two models, we establish uniform and $L^2$ rates and multivariate central limit theorems (CLTs) for general series ridge estimators. Additionally, we establish that spline and wavelet series ridge estimators achieve optimal uniform and $L^2$ rates in \cite{St82} over H\"older spaces. These results can be viewed as an extension of previous works on non-penalized series estimation (e.g. \cite{Ne97} and \cite{de02} for i.i.d. data, and \cite{ChCh15} for time series, to name a few) to series ridge regression for spatial data. We also provide consistent estimators of the asymptotic variances that appear in the multivariate CLTs. 

For nonparametric regression of spatial data, there are a large number of papers addressing the theoretical analysis of kernel estimators, namely, \textit{local} nonparametric estimators. As a part of recent contributions to the literature, we refer to \cite{MaSt10}, \cite{ChZhYa22}, and \cite{KuMa23} for spatial trend regression, and \cite{HaLuTr04}, \cite{LuCh04}, \cite{HaLuYu09}, \cite{Ro11}, \cite{Je12}, \cite{MaSeOu17}, and \cite{Ku19b, Ku22a} for regression with spatially dependent covariates. On the other hand, in spatial data analysis, it is common to have a large number of sampling sites, ranging from thousands to tens of thousands. In such cases, when researchers are interested in the shape of the \textit{global} spatial trend, utilizing local nonparametric methods like kernel smoothing can lead to significant computational costs due to the repetition of local trend estimation. 

In practice, there is often a significant interest in global spatial trend functions or spatial regression functions in fields such as economics, sociology, epidemiology, and meteorology. See for example \cite{Sh07} for climate change and \cite{BaWaAn21} for species distribution, and see also Chapter 4 of \cite{Cr15}. Despite the evident demand for global nonparametric regression on spatial data, there appears to be a lack of theoretical results on nonparametric regression for spatial data using general series estimators, except for the work by \cite{LeRo16}. The authors provide $L^2$ and uniform rates for non-penalized series estimators of a spatial regression model, but their uniform rates are slower than the optimal rate in \cite{St82}. Further, our spatial regression model allows regression functions that can change smoothly over the sampling region, and hence our model can be seen as a nonlinear extension of geographically weighted regression models introduced in \cite{BrFoCh96} (see, e.g. \cite{LuChHaSt14} and \cite{GoLuChBrHa15} for a review). 
Most existing papers investigate the theoretical properties of non-penalized series estimators. In contrast, for spatial data with substantial sample sizes, penalized series estimators are often employed to numerically stabilize the estimation of regression functions. Therefore, this paper aims to examine the properties of series estimators incorporating a ridge penalty.

Additionally, series estimation of spatial trend functions plays a crucial role in the analysis of spatio-temporal data. In the field of functional data analysis, spatio-temporal data is often modeled as a surface time series when it can be treated as random surfaces observed at each time point (see, e.g. \cite{MaGe20} for a survey). However, in many spatio-temporal data, it is common that only discrete observations at irregularly spaced locations for each time point are available and hence researchers normally perform (penalized) series estimation to construct the random surfaces (cf. \cite{GrKoRe17}). Our results can provide building blocks for the theoretical analysis of such a procedure.

The details of our theoretical contributions are as follows. 
First, to cope with the irregular spatial spacing, we shall adopt the stochastic sampling design of \cite{La03a}, which allows the sampling sites to have a nonuniform density across the sampling region and enables the number of sampling sites $n$ to grow at a different rate compared with the volume of the sampling region $A_n$. In many scientific fields, such as ecology, geology, meteorology, and seismology, spatial samples are often collected over irregularly spaced points from continuous random fields because of physical constraints. The stochastic sampling scheme accommodates both the pure increasing domain case ($\lim_{n \to \infty}A_n/n = \kappa \in (0,\infty)$) and the mixed increasing domain case ($\lim_{n \to \infty}A_n/n = 0$). From a theoretical viewpoint, this scheme covers all possible asymptotic regimes since it is well-known that the sample mean is not consistent under the infill asymptotics (cf. \cite{La96}). See \cite{La03b}, \cite{LaZh06}, \cite{BaLaNo15}, \cite{MaYa18}, and \cite{KuKaSh24} for discussion on the stochastic spatial sampling design.

Second, for both spatial trend regression and spatial regression models, we employ a blocking argument to prove the uniform rates of series ridge estimators. However, unlike series estimation for temporally dependent data, the absence of a \textit{direction} in the observation of spatial data necessitates a different blocking approach. Additionally, due to stochastic sampling sites, the number of data points in each block becomes random, requiring consideration of its impact. To address this difficulty, we extend the blocking technique developed in \cite{Yu94} on $\beta$-mixing time series to our framework. 

Third, for the spatial trend regression model, we establish that the uniform rates and the convergence rate of a multivariate CLT for series ridge estimators depend not directly on the sample size (i.e., the number of sampling sites) but on the expansion rate of the volume of the sampling region by extending the arguments in \cite{LaZh06} and \cite{KuMa23}. This is because, to capture the spatial dependence of the error terms, the sampling region needs to expand rather than rely on an increase in sample size. Additionally, we have shown that the asymptotic variance of series ridge estimators depends on spatial long-run variance, and the multivariate CLTs hold for $\alpha$-mixing random fields. We also propose an estimator of the spatial long-run variance inspired by a heteroskedasticity-autocorrelation (HAC) robust estimator, which is commonly used in time series analysis for estimating the long-run variance, and establish its consistency. 

Fourth, we explore in detail examples of spatial processes that satisfy our dependence conditions. Specifically, we show that a broad class of L\'evy-driven moving average (MA) random fields, which include continuous autoregressive moving average (CARMA) random fields (cf. \cite{BrMa17}), satisfies our assumptions. The CARMA random fields are known as a rich class of models for spatial data that can represent non-Gaussian random fields by introducing non-Gaussian L\'evy random measures (cf. \cite{BrMa17} and \cite{Ku22a}). Verifying our regularity conditions to L\'evy-driven MA fields is indeed non-trivial and relies on several probabilistic techniques from L\'evy process theory and theory of infinitely divisible random measures (cf. \cite{Be96}, \cite{Sa99}, and \cite{RaRo89}).

The rest of the paper is organized as follows. 
In Section \ref{Sec:reg-trend}, we introduce the spatial trend regression model, a stochastic sampling design for irregularly spaced sampling sites, and the dependence structure of spatial processes. Additionally, we provide the uniform and $L^2$ convergence rates, a multivariate CLT, and an estimator of the asymptotic variance for series ridge estimators. In Section \ref{Appendix:example}, we provide examples of spatial processes that satisfy our dependence conditions. In Section \ref{Sec:real-data}, we analyze the human population in Tokyo to illustrate the usefulness of our methods. Section \ref{Sec:conclusion} concludes and discusses possible extensions. Proofs for some results in Section \ref{Sec:reg-trend} are included in Appendix. The supplementary material includes the introduction of the spatial regression model with the uniform and $L^2$ convergence rates, a multivariate CLT, and an estimator of the asymptotic variance for series ridge estimators (Section \ref{Sec:reg-covariate}), proofs for Section \ref{Sec:reg-trend} (Section \ref{Appendix:sec2}), proofs for Section \ref{Appendix:example} (Section \ref{Appendix:ex-proof}), proofs for Section \ref{Sec:reg-covariate} (Section \ref{Appendix:sec3}), and auxiliary lemmas (Section \ref{Appendix:lemmas}).

\subsection{Notation}
For any vector $\bm{x} = (x_{1},\dots, x_{q})' \in \R^{q}$, let $|\bm{x}| =\sum_{j=1}^{q}|x_{j}|$, $\|\bm{x}\| =\sqrt{\sum_{j=1}^{q}x_j^2}$, and $\|\bm{x}\|_\infty = \max_{1 \leq j \leq q}|x_j|$ denote the $\ell^{1}$, $\ell^{2}$, and $\ell^\infty$-norms of $\bm{x}$, respectively. For any set $T_1,T_2 \subset \mathbb{R}^d$, let $d(T_1,T_2) = \inf\{|\bm{x} - \bm{y}| : \bm{x} \in T_1, \bm{y} \in T_2\}$. For any $p \times p$ matrix $A$, let $A^-$ denote the Moore-Penrose generalized inverse, let $\text{tr}(A)$ denote the trace of $A$, $\|A\|_F = \sqrt{\text{tr}(A'A)}$ denote the Frobenius norm of $A$, and $\lambda_{\text{min}}(A)$ and $\lambda_{\text{max}}(A)$ denote the minimum and the maximum eigenvalue of $A$. For any $p \times q$ matrix $B$, let $\|B\| = \sqrt{\lambda_{\text{max}}(B'B)}$ denote the spectral norm of $B$. For any set $A \subset \R^{d}$ and any vector $\bm{a}=(a_1,\dots, a_d)' \in (0,\infty)^d$, let $|A|$ denote the Lebesgue measure of $A$, let $[\![A]\!]$ denote the number of elements in $A$, and let $\bm{a}A = \{(a_1x_1,\dots a_dx_d): \bm{x} = (x_1,\dots,x_d) \in A\}$. 
For any positive sequences $a_{n}, b_{n}$, we write $a_{n} \lesssim b_{n}$ if there is a constant $C >0$ independent of $n$ such that $a_{n} \leq Cb_{n}$ for all $n$,  $a_{n} \sim b_{n}$ if $a_{n} \lesssim b_{n}$ and $b_{n} \lesssim a_{n}$. For any non-empty set $D$ and any real-valued functions $f$ and $h$ on $D$, let $\|f\|_\infty = \sup_{x \in D}|f(x)|$, $\|f\|_{L^2}=(\int_Df^2(x)dx)^{1/2}$, and $\|f\|_{L^2(h)} = (\int_D f^2(x)h(x)dx)^{1/2}$. Let $(F, \vertiii{\cdot})$ be a subset of a normed space. For $\delta>0$, a finite set $G \subset F$ is called a $\delta$-covering of $F$ with respect to the norm $\vertiii{\cdot}$ if for any $f \in F$ there exists $g \in G$ such that $\vertiii{f-g} < \delta$. The minimum cardinality of a $\delta$-covering of $F$ with respect to $\vertiii{\cdot}$ is called the covering number of $F$ with respect to $\vertiii{\cdot}$ and denoted by $N(F, \vertiii{\cdot},\delta)$.

\section{Series trend estimators}\label{Sec:reg-trend}

In this section, we establish asymptotic properties of series ridge estimators of a spatial trend regression model. To this end, we first discuss mathematical settings of our model (Section \ref{Subsec:model}), sampling design (Section \ref{Subsec:sampling}), and spatial dependence structure (Section \ref{Subsec:dependence}). Then we provide uniform and $L^2$ convergence rates and a multivariate CLT for general series estimators (Sections \ref{Subsec:STR-unif-rate}, \ref{Subsec:STR-L2}, and \ref{Subsec:STR-AN}), and establish that spline and wavelet estimators attain the optimal uniform and $L^2$ convergence rates  (Sections \ref{Subsec:STR-opt-unif-rate} and \ref{Subsec:STR-L2}). 

\subsection{Spatial trend regression model}\label{Subsec:model}

Let $R_0=[-1/2,1/2]^d$. Consider the nonparametric regression model
\begin{align}\label{eq:Spatial-Trend-Reg}
    Y(\bm{s}_{n,i}) &= m_0\left({\bm{s}_{n,i} \over A_n}\right) + \eta\left({\bm{s}_{n,i} \over A_n}\right)e(\bm{s}_{n,i}) + \sigma_{\varepsilon}\left({\bm{s}_{n,i} \over A_n}\right)\varepsilon_{i}, \nonumber \\ 
    &:= m_0\left({\bm{s}_{n,i} \over A_n}\right) + e_{n,i} + \varepsilon_{n,i},\ \bm{s}_{n,i} = (s_{ni,1},\dots, s_{ni,d})' \in R_n,\ i=1,\dots,n,
\end{align}
where $R_n=\prod_{j=1}^d[-A_{n,j}/2,A_{n,j}/2]$, $A_n = \prod_{j=1}^d A_{n,j}$, ${\bm{s}_{n,i} \over A_n}=\left({s_{ni,1} \over A_{n,1}},\dots,{s_{ni,d} \over A_{n,d}}\right)'$ with $A_{n,j} \to \infty$ as $n \to \infty$, $m_0:R_0 \to \mathbb{R}$ is the spatial trend function, $\bm{e} = \{e(\bm{x}): \bm{x} \in \mathbb{R}^d\}$ is a stationary random field defined on $\mathbb{R}^d$ with $\E[e(\bm{x})] = 0$ and $\E[e^2(\bm{x})]=1$ for any $\bm{x} \in \mathbb{R}^d$, $\eta:R_0 \to (0,\infty)$ is the variance function of spatially dependent random variables $\{e_{n,i}\}$, $\{\varepsilon_{i}\}$ is a sequence of i.i.d. random variables such that $\E[\varepsilon_{i}] = 0$ and $\E[\varepsilon_{i}^2] = 1$, and $\sigma_{\varepsilon}: R_0 \to (0,\infty)$ is the variance function of random variables $\{\varepsilon_{n,i}\}$. Note that $A_{n,j}$ can grow at different rates with respect to the sample size $n$ for each $j$. The model (\ref{eq:Spatial-Trend-Reg}) is also considered in \cite{KuMa23}, who investigate the local polynomial estimation of the spatial trend function. See Remark 2.1 of \cite{KuMa23} for a discussion of the model. 

\subsection{Sampling design}\label{Subsec:sampling}

We assume that the sampling sites $\bm{s}_{n,1},\hdots, \bm{s}_{n,n}$ are obtained from the realizations of random vectors $\bm{S}_{n,1},\hdots, \bm{S}_{n,n}$. To simplify the notation, we will write $\bm{s}_{n,i}$ and $\bm{S}_{n,i}$ as $\bm{s}_i=(s_{i,1},\dots, s_{i,d})'$ and $\bm{S}_i=(S_{i,1},\dots,S_{i,d})'$, respectively. We assume the following conditions on our stochastic sampling design.

\begin{assumption}\label{Ass:sample}
Let $g$ be a probability density function with support $R_0=[-1/2,1/2]^d$. 
\begin{itemize}
\item[(i)] $A_n/n \to \kappa \in [0,\infty)$ as $n \to \infty$. 
\item[(ii)] $\{\bm{S}_{i}=(S_{i,1},\dots,S_{i,d})'\}_{i = 1}^{n}$ is a sequence of i.i.d. random vectors with density $A_n^{-1}g(\cdot/A_n)$. 
\item[(iii)] $\{\bm{S}_i\}_{i=1}^{n}$, $\bm{e} = \{e(\bm{x}): \bm{x} \in \mathbb{R}^{d}\}$, and $\{\varepsilon_i\}_{i=1}^{n}$ are mutually independent. 
\end{itemize}
\end{assumption}

Condition (i) implies that our sampling design allows both the pure increasing domain case ($\lim_{n \to \infty}A_{n}/n = \kappa \in (0,\infty)$) and the mixed increasing domain case ($\lim_{n \to \infty}A_n/n = 0$). Our scheme covers all possible asymptotic regimes that would validate asymptotic inference for spatial data. Although the infill asymptotics, which assume that the volume of $R_n$ is bounded but the number of sampling sites increases, are excluded from our regime, our sampling design is general enough as it is known that the infill asymptotics does not work even for the estimation of sample means (cf. \cite{La96}). Condition (ii) implies that the sampling density can be nonuniformly distributed over the sampling region $R_n$. 

As an alternative way to cope with irregularly spaced observations, one may work with the domain expanding and infill (DEI) asymptotics (cf. \cite{LuTj14}). In the DEI asymptotics, it is assumed from the outset that the sampling sites are countably infinite on $\mathbb{R}^d$, whereas in the pure/mixed increasing domain asymptotics, a finite number of points within a finite observation region $R_n$ are assumed to be obtained as sampling sites, which may be more realistic in some applications than the DEI asymptotics. We leave the thoretical analysis under the DEI asymptotics as future studies. 

\subsection{Dependence structure}\label{Subsec:dependence}

Now we discuss the dependence structure of the random field $\bm{e}$. We assume that random field $\bm{e}$ satisfies some mixing conditions. First, we define the $\alpha$- and $\beta$-mixing coefficients for the random field $\bm{e}$. Let $\mathcal{F}_{\bm{e}}(T) = \sigma(\{e(\bm{x}): \bm{x} \in T\})$ be the $\sigma$-field generated by the variables $\{e(\bm{x}):\bm{x} \in T\}$, $T \subset \mathbb{R}^d$. For any two subsets $T_1$ and $T_2$ of $\mathbb{R}^d$, let 
$\bar{\alpha}(T_1,T_2) = \sup\{|\Prob(A \cap B) - \Prob(A)\Prob(B)|: A \in \mathcal{F}_{\bm{e}}(T_1), B \in \mathcal{F}_{\bm{e}}(T_2)\}$, $\bar{\beta}(T_1,T_2) = \sup {1 \over 2}\sum_{j=1}^{J}\sum_{k=1}^{K}|\Prob(A_{j}\cap B_{k}) - \Prob(A_{j})\Prob(B_{k})|$ where the supremum for $\bar{\beta}(T_1,T_2)$ is taken over all pairs of (finite) partitions $\{A_{1},\hdots,A_{J}\}$ and $\{B_{1},\hdots, B_{K}\}$ of $\mathbb{R}^{d}$ such that $A_{j} \in \mathcal{F}_{\bm{e}}(T_1)$ and $B_{k} \in \mathcal{F}_{\bm{e}}(T_2)$.
The $\alpha$- and $\beta$-mixing coefficients of the random field $\bm{e}$ are defined as $\alpha(a;b) = \sup\{\bar{\alpha}(T_1,T_2): d(T_1,T_2) \geq a, T_1,T_2 \in \mathcal{R}(b)\}$, $\beta(a;b) = \sup\{\bar{\beta}(T_1,T_2): d(T_1,T_2) \geq a, T_1,T_2 \in \mathcal{R}(b)\}$ where $a,b>0$ and $\mathcal{R}(b)$ is the collection of all the finite disjoint unions of cubes in $\mathbb{R}^d$ with a total volume not exceeding $b$. This restriction is important for $d \geq 2$, as discussed in \cite{Br89} and \cite{La03a}. 

To establish uniform convergence rates of series ridge estimators, we assume that the random field $\bm{e}$ is $\beta$-mixing. Additionally, to establish CLTs of series estimators for spatial trend, we assume that it is $\alpha$-mixing. We refer to \cite{Br93} and \cite{Do94} for more details on mixing coefficients for random fields.

\begin{assumption}\label{Ass:model}
For $j=1,\dots,d$, let $\{A_{n1,j}\}_{n \geq 1}$ and $\{A_{n2,j}\}_{n \geq 1}$ be sequences of positive numbers such that $\min\{A_{n2,j}, {A_{n1,j} \over A_{n2,j}}\} \to \infty$ as $n \to \infty$. Define $A_n^{(1)} = \prod_{j=1}^{d}A_{n1,j}$ and $\underbar{A}_{n2} = \min_{1 \leq j \leq d}A_{n2,j}$. Let $q>2$ be some integer. 
\begin{itemize}
\item[(i)] The random field $\bm{e}=\{e(\bm{x}):\bm{x} \in \mathbb{R}^d\}$ is strictly stationary such that $\E[e(\bm{0})]=0$, $\E[e(\bm{0})^2]=1$, and $\E[|e(\bm{0})|^q]<\infty$. Moreover, $\int_{\mathbb{R}^d}|\sigma_{\bm{e}}(\bm{x})|d\bm{x}<\infty$ where $\sigma_{\bm{e}}(\bm{x}) = \E[e(\bm{0})e(\bm{x})]$.  
\item[(ii)] The random field $\bm{e}$ is $\beta$-mixing with mixing coefficients $\beta(a;b)$ such that as $n \to \infty$, 
\begin{align}\label{lim:beta-coef}
n^{d\eta_1}A_n(A_n^{(1)})^{-1}\beta(\underbar{A}_{n2};A_n) \to 0
\end{align}
for some $\eta_1>0$. 
\item[(iii)] $\{\varepsilon_i\}_{i=1}^n$ is a sequence of i.i.d. random variables such that $\E[\varepsilon_1]=0$, $\E[\varepsilon_1^2]=1$, and $\E[|\varepsilon_1|^q]<\infty$.
\item[(iv)] The functions $\eta$ and $\sigma_\varepsilon$ are continuous on $R_0$. 
\end{itemize}
\end{assumption}

Condition (\ref{lim:beta-coef}) is concerned with a large-block-small-block argument for $\beta$-mixing sequences. In order to derive uniform convergence rates of series estimators, more careful arguments on the effects of non-equidistant sampling sites are necessary than those for proving asymptotic normality. Specifically, we extend the blocking technique in \cite{Yu94} (Corollary 2.7) for $\beta$-mixing time series to random fields observed at regularly spaced sampling sites. In Section \ref{Appendix:example}, we will show that a wide class of random fields satisfies our $\beta$-mixing conditions.

\cite{Ku22a} imposes similar assumptions on the sampling design (Conditions (S1)-(S3) in Assumption 2.1). However, (S1) is a condition required for evaluating the bias of the kernel estimator considered in his paper and is not necessary when considering series estimators. For further details on this point, see comments after Corollary 2.1 below. Conditions (S2) and (S3) specify conditions on $A_n$, $A_{1,n}$ and $A_{2,n}$ in his notation, which correspond to $A_{n,j}$, $A_{n1,j}$ and $A_{n2,j}$ in our notation, concerning the sample size $n$. These conditions serve as sufficient conditions for the assumptions on 
$A_{n,j}$, $A_{n1,j}$ and $A_{n2,j}$ in our paper. In particular, the assumptions required for the theoretical results presented in Section 2 are typically satisfied by setting $A_{n,j} \sim n^{\mathfrak{a}_0/d}$, $A_{n1,j}/A_{n,j} + A_{n2,j}/A_{n1,j} \sim n^{-\mathfrak{a}_1}$ for some $\mathfrak{a}_0 \in(0,1]$ and $\mathfrak{a}_1>0$. Note that the condition $A_{n,j} \sim n^{\mathfrak{a}_0/d}$ allows both the pure and mixed increasing domain asymptotics. See also Proposition \ref{prp:LevyMA-ex} below on this point.

\subsection{Uniform convergence rates}\label{Subsec:STR-unif-rate}

Before introducing our estimator of $m_0$, we briefly review the concepts of linear sieves and sieve basis functions. Let $\Theta$ be an infinite-dimensional parameter space endowed with a (pseudo) metric $d$ and let $\Theta_n$ be a sequence of approximating spaces, which are less complex but dense in $\Theta$. The approximating spaces are called sieves by \cite{Gr81}. Popular sieves are typically compact, nondecreasing ($\Theta_n \subset \Theta_{n+1} \subset \dots \subset \Theta$) and are such that for any $\theta \in \Theta$ there exists an element $\pi_n\theta$ in $\Theta_n$ satisfying $d(\theta,\pi_n\theta) \to 0$ as $n \to \infty$, where the notation $\pi_n$ can be regarded as a projection mapping from $\Theta$ to $\Theta_n$. A sieve is called a (finite-dimensional) linear sieve if it is a linear span of finitely many known (sieve) basis functions. Linear sieves, including splines and wavelets as special cases, form a large class of sieves useful to approximate a class of smooth real-valued functions on $R_0$. See Section \ref{Subsec:STR-opt-unif-rate} or \cite{Ch07} for references and explicit construction of sieve basis functions, including spline and wavelet basis functions.

We estimate the spatial trend function $m_0$ by the following series ridge estimator. 
\begin{align*}
\hat{m}(\bm{z}) &= \hat{m}_n(\bm{z}):= \psi_J(\bm{z})'\left({\Psi'_{J,n}\Psi_{J,n} \over n} + \varsigma_{J,n} I_J\right)^{-1}{\Psi'_{J,n} \bm{Y} \over n},\ \bm{z}= (z_1,\dots,z_d)' \in R_0,
\end{align*}
where $I_J$ is the $J \times J$ identity matrix, $\varsigma_{J,n}$ is a sequence of positive constants with $\varsigma_{J,n} \to 0$ as $n, J \to \infty$, $\bm{Y} = (Y(\bm{S}_1),\dots,Y(\bm{S}_n))'$, $\psi_{J,1},\dots, \psi_{J,J}$ are a collection of $J$ sieve basis functions, $\psi_J(\bm{z}) = (\psi_{J,1}(\bm{z}),\dots,\psi_{J,J}(\bm{z}))'$, and $\Psi_{J,n} = (\psi_J(\bm{S}_1/A_n),\dots,\psi_J(\bm{S}_n/A_n))'$. 

In Section 4, we estimate the trend surface of the nighttime human population in the Tokyo central 23 districts as an application of the theoretical results provided in Section \ref{Subsec:STR-opt-unif-rate}. Additionally, as a method for uncertainty quantification, we present the standard error surface and 95\% confidence surfaces, using the multivariate central limit theorem for general series estimators provided in Section \ref{Subsec:STR-AN}.

Define $\zeta_J = \sup_{\bm{z} \in R_0}\|\psi_J(\bm{z})\|$, 
\begin{align*}
\lambda_J &= \lambda_{\text{min}}\left(\E\left[\psi_J(\bm{S}_j/A_n)\psi_J(\bm{S}_j/A_n)'\right]\right)^{-1/2}\\ 
&=\lambda_{\text{min}}\left(\E\left[\psi_J(\bm{S}_1/A_n)\psi_J(\bm{S}_1/A_n)'\right]\right)^{-1/2},\ j=1,\dots, n. 
\end{align*}
We assume the following conditions on the sieve basis functions. 

\begin{assumption}\label{Ass:sieve1}
Let $\nabla \psi_J(\bm{z}) = \left(\partial \psi_{J,j}(\bm{z})/\partial z_k\right)_{1 \leq j \leq J, 1 \leq k \leq d} \in \mathbb{R}^{J \times d}$. 
\begin{itemize}
\item[(i)] There exists $\omega \geq 0$ such that $\sup_{\bm{z} \in R_0}\|\nabla \psi_J(\bm{z})\| \lesssim J^\omega$.
\item[(ii)] There exists $\varpi \geq 0$ such that $\zeta_J \lesssim J^{\varpi}$.
\item[(iii)] $\lambda_{\rm{min}}\left(\E\left[\psi_J(\bm{S}_1/A_n)\psi_J(\bm{S}_1/A_n)'\right]\right)>0$ for each $J$ and $n$. 
\end{itemize}
\end{assumption}

Assumption \ref{Ass:sieve1} imposes a normalization on the functions approximated by a sieve space. Similar conditions are also considered in Assumption 4 of \cite{ChCh15}. Condition (iii) restricts the magnitude of the series terms. This condition is useful for controlling the convergence in probability of the sample second moment matrix of the approximating functions to its expectation. Note that Assumption \ref{Ass:sieve1} is a mild condition on the sieve basis functions. We can see that this assumption is satisfied by the widely used linear sieve bases. For example, we can see $\lambda_J \lesssim 1$ and $\zeta_J \lesssim \sqrt{J}$ for tensor-products of univariate polynomial spline, trigonometric polynomial or wavelet bases. We refer to \cite{Hu98} and \cite{Ch07} for more details. The same comments apply to Assumption \ref{Ass:sieve3} in Section \ref{Sec:reg-covariate}. 

Let $\tilde{\psi}_J(\bm{z})$ denote the orthonormalized vector of basis functions, that is, 
\[
\tilde{\psi}_J(\bm{z}) = \E\left[\psi_J(\bm{S}_1/A_n)\psi_J(\bm{S}_1/A_n)'\right]^{-1/2}\psi_J(\bm{z}),
\]
and $\tilde{\Psi}_{J,n} = (\tilde{\psi}_J(\bm{S}_1/A_n),\dots,\tilde{\psi}_J(\bm{S}_n/A_n))'$.

Define $L^2(R_0)$ as the space of functions for which $\|f\|_{L^2}<\infty$, and let $\Psi_J$ denote the closed linear span of $\{\psi_{J,1},\dots,\psi_{J,J}\}$ in $L^2(R_0)$.
Let $P_{J,n}$ be the empirical projection operator onto $\Psi_J$, namely,
\begin{align*}
P_{J,n}m(\bm{z}) &= \psi_J(\bm{z})'\left({\Psi'_{J,n}\Psi_{J,n} \over n}\right)^- {1 \over n}\sum_{i=1}^n \psi_J\left(\bm{S}_i \over A_n\right)m\left(\bm{S}_i \over A_n\right) = \tilde{\psi}_J(\bm{z})(\tilde{\Psi}'_{J,n}\tilde{\Psi}_{J,n})^- \tilde{\Psi}_{J,n}M
\end{align*}
where $M = M_n= (m(\bm{S}_1/A_n),\dots,m(\bm{S}_n/A_n))'$. The operator $P_{J,n}$ is well defined operator: if $L_n^2(R_0)$ denotes the space of functions with norm $\|\cdot\|_{L^2,n}$ where $\|f\|_{L^2,n}^2 = {1 \over n}\sum_{i=1}^{n}f(\bm{S}_i/A_n)^2$, then $P_{J,n}: L_n^2(R_0) \to L_n^2(R_0)$ is an orthogonal projection onto $\Psi_J$ whenever $\Psi'_{J,n}\Psi_{J,n}$ is invertible.

Let $\tilde{m}$ denote the projection of $m_0$ onto $\Psi_J$ under the empirical measure, namely, 
\begin{align*}
\tilde{m}(\bm{z}) &= P_{J,n}m_0(\bm{z})
= \tilde{\psi}_J(\bm{z})'(\tilde{\Psi}'_{J,n}\tilde{\Psi}_{J,n})^- \tilde{\Psi}'_{J,n} M_0
\end{align*}
where $M_0 = M_{0,n} = (m_0(\bm{S}_1/A_n),\dots,m_0(\bm{S}_n/A_n))'$. 

To establish a sharp uniform convergence rate of the variance
term for an arbitrary linear sieve space, we assume the following conditions.
\begin{assumption}\label{Ass:ze-lam}
Define $\overline{A}_{n1} = \max_{1 \leq j \leq d}A_{n1,j}$. Let $q$ be the integer in Assumption \ref{Ass:model} which is greater than $2$ and let $\iota>0$ be a constant. As $n, J \to \infty$, 
\begin{itemize}
\item[(i)] $(\zeta_J\lambda_J)^{{q \over q-2}} \lesssim \sqrt{{A_n \over \log n}}$,
\item[(ii)] $\zeta_J\lambda_J \lesssim A_n^{-{1 \over 2}}n^{1-{1 \over q}}(\log n)^{{1 \over 2} + (q-1)\iota}$,
\item[(iii)] $\zeta_J\lambda_J(\overline{A}_{n1})^d n^{{1 \over q}}(\log n)^{\iota + {1 \over 2}} \lesssim A_n^{{1 \over 2}}$.
\end{itemize}
\end{assumption}

Define $\bar{m}(\bm{z}) = \psi_J(\bm{z})'(\Psi'_{J,n}\Psi_{J,n} + \varsigma_{J,n} nI_J)^{-1} \Psi'_{J,n} M_0$. Note that $\hat{m}(\bm{z}) - \tilde{m}(\bm{z}) = \{\hat{m}(\bm{z}) - \bar{m}(\bm{z})\} + \{\bar{m}(\bm{z}) - \tilde{m}(\bm{z})\} =: \check{m}(\bm{z}) + \dot{m}(\bm{z})$. Therefore, to establish uniform convergence rates of general series estimators (see Proposition \ref{prp:variance} below), it is sufficient to show (a) $\sup_{\bm{z} \in R_0}|\check{m}(\bm{z})| = O_p(\zeta_J\lambda_J\sqrt{\log n/A_n})$ and (b) $\sup_{\bm{z} \in R_0}|\dot{m}(\bm{z})| = O_p(\varsigma_{J,n}\zeta_J^2)$. Assumption \ref{Ass:ze-lam} (i) implies $\zeta_J \lambda_J\sqrt{\log n/A_n} \to 0$ as $n,J \to \infty$. Assumptions \ref{Ass:ze-lam} (ii) and (iii) are mainly utilized in the proof of (a). Let $\mathcal{S}_n$ be a set with finitely many points in $R_0$ that satisfies certain additional conditions. For any $\bm{z} \in R_0$, we denote by $\bm{z}_n(\bm{z})$ the point in $\mathcal{S}_n$ that is the closest to $\bm{z}$ in terms of the Euclidean distance. Under these settings, we have $\sup_{\bm{z} \in R_0}|\check{m}(\bm{z})| \leq \sup_{\bm{z} \in R_0}|\check{m}(\bm{z}) - \check{m}(\bm{z}_n(\bm{z}))|+\sup_{\bm{z}_n \in \mathcal{S}_n}|\check{m}(\bm{z}_n)|$. We use Assumption \ref{Ass:sieve1} (i) to evaluate both $\sup_{\bm{z} \in R_0}|\check{m}(\bm{z}) - \check{m}(\bm{z}_n(\bm{z}))|$ and (b). To assess the convergence rate of $\sup_{\bm{z}_n \in S_n}|\check{m}(\bm{z}_n)|$, it is necessary to consider the spatial correlation of the observations. For this purpose, we divide the sampling region $R_n$ into subregions and apply a large-block-small-block argument. However, under the stochastic sampling design, the number of observations within each subregion is random, which poses challenges in the theoretical analysis. This requires a different proof approach compared to existing analyses for i.i.d. data, equally spaced time series, or spatial data on a regular grid. Our main technical contribution lies in providing a rigorous theoretical analysis by addressing this issue using a method of constructing independent blocks for $\beta$-mixing random fields that consider the randomness of the observation points. Furthermore, to show (b), it is necessary to obtain an upper bound of $\|(\tilde{\Psi}'_{J,n}\tilde{\Psi}_{J,n}/n + \varsigma_{J,n} \check{\Psi}_J^{-1})^{-1} - (\tilde{\Psi}'_{J,n}\tilde{\Psi}_{J,n}/n)^{-1}\|$. For this purpose, we establish a new inequality (Lemma \ref{lem:matrix-inv-norm} in the supplementary material) concerning the spectral norm of the difference between two invertible matrices. To evaluate the corresponding term in the spatial regression model introduced in Appendix \ref{Sec:reg-covariate}, we additionally need to extend \cite{Tr12}'s matrix exponential inequality for sums of independent random matrices to the case of spatially dependent data (Lemma \ref{lem:B} in the supplementary material). To evaluate the term corresponding to (b) in the spatial regression model, one could use general perturbation theory for linear operators (e.g., \cite{Ka95}) as an alternative to Lemma \ref{lem:matrix-inv-norm}. Additionally, Lemma \ref{lem:B} can be shown by establishing results corresponding to Theorem 4.2 and Corollary 4.2 of \cite{ChCh15} for spatially dependent data. See the proof of Proposition \ref{prp:variance} for a more detailed discussion. 

In Section \ref{Appendix:example}, we will see that Assumptions \ref{Ass:model} (i), (ii), and \ref{Ass:ze-lam} are satisfied by polynomial spline, trigonometric polynomial or wavelet bases when the random field $\bm{e}$ is in a class of L\'evy-driven moving average random fields, which can represent a wide class of spatial processes on $\mathbb{R}^d$, including both Gaussian and non-Gaussian random fields. We refer to \cite{BrMa17} and \cite{Ku22a} for examples and detailed properties on L\'evy-driven moving average random fields. 

\begin{proposition}\label{prp:variance}
Suppose that Assumptions \ref{Ass:sample}, \ref{Ass:model}, \ref{Ass:sieve1}, and \ref{Ass:ze-lam} hold. Additionally, assume that $\|m_0\|_\infty<\infty$, $\varsigma_{J,n}  \lesssim \zeta_J\lambda_J^{-1}\sqrt{\log n/n}$, and $\zeta_J^2 \lambda_J^{-1}(\zeta_J + \lambda_J^{-1})\sqrt{\log n/n} \to 0$ as $n,J \to \infty$.
Then
\begin{align}\label{eq:var-unif-rate}
\|\hat{m} - \tilde{m}\|_\infty &= O_p\left(\zeta_J\lambda_J\sqrt{{\log n \over A_n}} + \varsigma_{J,n} \zeta_J^2 \right)\ \text{as $n,J \to \infty$}.
\end{align}

\end{proposition}

In Proposition \ref{prp:variance}, the first term of the convergence rate corresponds to the variance term of $\hat{m}$, while the second term corresponds to the effect of the ridge penalty. Specifically, the first term implies that the convergence rate of the variance term of $\hat{m}$ differs from the case of i.i.d. data, depending not on the sample size $n$ but on the volume $A_n$ of the sampling region $R_n$. This is due to the spatial correlation of the error terms, meaning that even if the number of sampling sites within the sampling region increases, consistent estimation of the trend function cannot be achieved unless the sampling region expands.

Let $L^\infty(R_0)$ denote the space of functions of which $\|f\|_\infty<\infty$ and let
\begin{align*}
\|P_{J,n}\|_\infty &= \sup_{m \in L^\infty(R_0), \|m\|_\infty \neq 0}{\|P_{J,n}m\|_\infty \over \|m\|_\infty}
\end{align*}
denote the operator norm of $P_{J,n}$. 

The following result provides uniform convergence rates of general series estimators. Building upon the result, we will see in Section \ref{Subsec:STR-opt-unif-rate} that the spline and wavelet estimators achieve the optimal convergence rate when the trend function $m_0$ belongs to a H\"older space.  
\begin{corollary}\label{cor:unif-rate1}
Suppose the assumptions in Proposition \ref{prp:variance} hold. Then
\begin{align*}
\|\hat{m} - m_0\|_\infty &= O_p\left(\zeta_J\lambda_J\sqrt{{\log n \over A_n}}+ \varsigma_{J,n} \zeta_J^2  \right) + (1 + \|P_{J,n}\|_\infty)\inf_{m \in \Psi_J}\|m_0 - m\|_\infty.
\end{align*}
Further, if the linear sieve satisfies $\zeta_J\lesssim \sqrt{J}$, $\lambda_J \lesssim 1$, and $\|P_{J,n}\|_{\infty} = O_p(1)$, then
\[
\|\hat{m} - m_0\|_\infty = O_p\left(\sqrt{{J\log n \over A_n}} + \varsigma_{J,n} J  + \inf_{m \in \Psi_J}\|m_0 - m\|_\infty\right).
\]
\end{corollary}

Note that the term $(1 + \|P_{J,n}\|_\infty)\inf_{m \in \Psi_J}\|m_0 - m\|_\infty$ comes from the bound of $\|\tilde{m} - m_0\|_\infty$, which can be addressed similarly to the proof of Lemma 2.4 in \cite{ChCh15}. Corollary \ref{cor:unif-rate1} shows an advantage of series estimators over kernel estimators since for kernel estimators, the boundary of $R_0$ must be excluded to establish uniform convergence rates over a compact set. This arises from the fact that the series estimator is a global nonparametric estimator, whereas kernel estimators are local nonparametric estimators. In practice, uniform estimation of the spatial trend function on $R_0$ becomes crucial when analyzing spatio-temporal data within the framework of functional data analysis. Especially in scenarios where spatial data is observed at irregularly spaced sampling sites within a region at multiple time points, estimating the spatial trend at each time point using a series estimator and employing it as surface data enables the analysis of spatio-temporal data as surface time series. Corollary \ref{cor:unif-rate1} provides a building block for the theoretical validity of this approach. We refer to \cite{MaGe20} for a survey on recent works of surface time series.

As we will see in Section \ref{Subsec:STR-opt-unif-rate}, when the target function is a smooth function with higher-order partial derivatives, some may wish to consider the uniform rate under the (weighted) Sobolev or Besov norm. However, this paper is mainly interested in the global estimation of the regression function itself, and the partial derivatives, which represent the local structure of the regression function, are not the objects of estimation. Therefore, as a criterion for evaluating the properties of the estimator, the uniform rate under the (weighted) Sobolev norm, which involves derivatives of the function, is not appropriate for our purpose. Furthermore, from the perspective of statistical applications, convergence under the $L^\infty$ norm is easier to interpret compared to the (weighted) Sobolev or Besov norm.

\subsection{Optimal uniform rates for spline and wavelet estimators}\label{Subsec:STR-opt-unif-rate}

We first introduce a H\"older space of smoothness $r$ on the domain $R_0$. A real-valued function $f$ on $R_0$ is said to satisfy a H\"older condition with exponent $0<\gamma \leq 1$ if there is a positive number $c$ such that $|f(\bm{x}_1) - f(\bm{x}_2)| \leq c\|\bm{x}_1 - \bm{x}_2\|^\gamma$ for all $\bm{x}_1,\bm{x}_2 \in R_0$. Given a $d$-tuple $\bm{\alpha} = (\alpha_1,\dots,\alpha_d)'$ of nonnegative integers, let $D^{\bm{\alpha}}$ denote the differential operator defined by $D^{\bm{\alpha}} = \partial^{\bm{\alpha}}/\partial x_1^{\alpha_1} \cdots \partial x_d^{\alpha_d}$. Let $\rho$ be a nonnegative integer and set $r = \rho + \gamma$. A real-valued function $f$ on $R_0$ is said to be $r$-smooth if it is $\rho$-times continuously differentiable on $R_0$ and $D^{\bm{\alpha}}f$ satisfies a H\"older condition with exponent $\gamma$ for all $\bm{\alpha}$ with $|\bm{\alpha}| = \rho$. 

Define $\Lambda^r(R_0)$ as the class of all $r$-smooth real-valued functions on $R_0$ and we call $\Lambda^r(R_0)$ as a H\"older space of smoothness $r$ on the domain $R_0$. Let $\text{BSpl}(J_0,[-1/2,1/2],\varrho)$ denote a B-spline sieve of degree $\varrho$ and dimension $J_0= n_0 + \varrho$ on the domain $[-1/2,1/2]$, where $n_0$ is the number of knots inside $[-1/2,1/2]$. Let $\text{Wav}(J_0,[-1/2,1/2],\varrho)$ denote a wavelet sieve basis of regularity $\varrho$ and dimension $J_0 = 2^{j_0}$ on the domain $[-1/2,1/2]$, where $j_0$ is the resolution level of the wavelet space. We refer to \cite{De01} and \cite{Sc07} for detailed construction of univariate B-spline on $[0,1]$ and refer to \cite{CoDaVi93} and \cite{Jo17} for further details on a wavelet space on $[0,1]$. 

We construct tensor product B-spline or wavelet bases for $R_0$ as follows. First, for $\bm{z}=(z_1,\dots,z_d)' \in R_0$, we construct $d$ B-spline or wavelet bases on $[-1/2,1/2]$. Next, we form the tensor-product basis by taking the product of the elements of each of the univariate bases. Then $\psi_J(\bm{z})$ may be expressed as $\psi_J(\bm{z}) = \bigotimes_{k=1}^{d}\psi_{J_{0,k}}(z_k)$ where the elements of each vector $\psi_{J_{0,k}}(z_k)$ span $\text{BSpl}(J_{0,k},[-1/2,1/2],\varrho)$ or $\text{Wav}(J_{0,k},[-1/2,1/2],\varrho)$. For example, the spline basis functions $\psi_{J_{0,k},j}(x)$ on $x \in [-1/2,1/2]$, $j=1,\dots,J_{0,k}$, $k=1,2$ are extended to the basis functions on $\bm{z}=(z_1,z_2)' \in [-1/2,1/2]^2$ by $\psi_{J_{0,1},j_1}(z_1)\times \psi_{J_{0,2},j_2}(z_2)$, $j_k=1,\dots,J_{0,k}$, $k=1,2$. Let $\text{BSpl}(J,R_0,\varrho)$ and $\text{Wav}(J,R_0,\varrho)$ denote the resulting tensor-product spaces spanned by the $J=\prod_{k=1}^dJ_{0,k}=\prod_{k=1}^d(n_{0,k}+\varrho)$ or $J=\prod_{k=1}^dJ_{0,k}=2^{\sum_{k=1}^d j_{0,k}}$ elements of $\psi_J(\bm{z})$. See Section 6 in \cite{ChCh15} for more detailed construction of tensor-product spaces.

\begin{assumption}\label{Ass:sieve2}
Recall that $g$ is a probability density function on $R_0$. 
\begin{itemize}
\item[(i)] The function $g$ is uniformly bounded away from zero and infinity on $R_0$. 
\item[(ii)] The mean function $m_0$ belongs to $\Lambda^r(R_0)$ for some $r>0$. 
\item[(iii)] The sieve $\Psi_J$ is $\text{BSpl}(J,R_0,\varrho)$ or $\text{Wav}(J,R_0,\varrho)$ with $\varrho > \max\{r,1\}$. 
\end{itemize}
\end{assumption}

Conditions (i) and (iii) imply Assumption \ref{Ass:sieve1} with $\zeta_{J} \lesssim \sqrt{J}$ and $\lambda_{J} \sim 1$. See also p.450 of \cite{ChCh15} on this point. Moreover, Assumption \ref{Ass:sieve2} implies that $\inf_{m \in \Psi_J}\|m - m_0\|_\infty \lesssim J^{-r/d}$ (see, e.g. \cite{DeLo93} and \cite{Hu98}). One can see 
that $\|P_{J,n}\|_\infty \lesssim 1$ with probability approaching one for spline and wavelet bases for equally-spaced $\beta$-mixing time series. We have extended their result to irregularly spaced $\beta$-mixing random fields to establish the following result. 


\begin{theorem}\label{thm:unif-rate1}
Suppose that Assumptions \ref{Ass:sample}, \ref{Ass:model}, \ref{Ass:ze-lam}, and \ref{Ass:sieve2} hold. Additionally, assume that $r\geq d$, $\sqrt{J^3\log n/n} \to 0$, and $\varsigma_{J,n} (\sqrt{JA_n/(\log n)} +\sqrt{n/(J\log n)})\lesssim 1$. If $J \sim \left(A_n/\log n\right)^{{d \over 2r+d}}$, then
\[
\|\hat{m} - m_0\|_\infty = O_p\left(\left({\log n \over A_n}\right)^{r \over 2r+d}\right).
\]
\end{theorem}

Theorem \ref{thm:unif-rate1} states that the spline and wavelet estimators achieve the optimal uniform convergence rates of \cite{St82} under the pure increasing domain case ($\lim_{n \to \infty} A_n/n = \kappa \in (0,\infty)$).

\subsection{$L^2$ convergence rates}\label{Subsec:STR-L2}

In this section, we discuss $L^2$ convergence rates of series estimators. The next result provides a sharp upper bound of the $L^2$ convergence rates of both variance and bias terms of general series estimators. Note that mixing conditions for the random field $\bm{e}$ are not necessary to establish the result. 

\begin{proposition}\label{prp:L^2-rate1} Suppose that Assumptions \ref{Ass:sample}, \ref{Ass:model}(i), (iii), (iv), and \ref{Ass:sieve1}(iii) hold. Additionally, assume that $\|m_0\|_\infty<\infty$, $\varsigma_{J,n} \lesssim \zeta_J\lambda_J^{-1}\sqrt{\log n/n}$, and $\zeta_J^3 \lambda_J^{-1}\sqrt{\log n/n} \to 0$ as $n,J \to \infty$. Then
\begin{align*}
\|\hat{m} - \tilde{m}\|_{L^2(g)} &= O_p\left({\zeta_J\lambda_J \over \sqrt{A_n}} + \varsigma_{J,n} \zeta_J\lambda_J^{-1}\right),\ \|\tilde{m} - m_0\|_{L^2(g)} = O_p\left(\|m_0 - m_{0,J}\|_{L^2(g)}\right),
\end{align*}
where $m_{0,J}$ is the $L^2(g)$ orthogonal projection of $m_0$ onto $\Psi_J$. 
\end{proposition}

The following result establishes that the spline and wavelet estimators attain the optimal $L^2$ convergence rates of \cite{St82} under the pure increasing domain case. 
\begin{corollary}\label{cor:L^2-rate1}
Suppose that Assumptions \ref{Ass:sample}, \ref{Ass:model}(i), (iii), (iv), and \ref{Ass:sieve2} hold. Additionally, assume that $r \geq d$, $\sqrt{J^3\log n/n} \to 0$ and $\varsigma_{J,n}(\sqrt{A_n} + \sqrt{n/(J\log n)}) \lesssim 1$ as $n,J \to \infty$. If $J \sim A_n^{{d \over 2r+d}}$, then
\[
\|\hat{m} - m_0\|_{L^2(g)} = O_p\left(A_n^{-{r \over 2r+d}}\right). 
\]
\end{corollary}

\subsection{Asymptotic normality}\label{Subsec:STR-AN}

In this subsection, we assume that the random field $\bm{e}$ is $\alpha$-mixing and there exist a non-increasing function $\alpha_1$ with $\alpha_1(a)\to 0$ as $a \to \infty$ and a non-decreasing function $\alpha_2$ (that may be unbounded) such that $\alpha(a;b) \leq \alpha_1(a)\alpha_2(b)$. These assumptions are standard in the literature of spatial data analysis on $\mathbb{R}^d$ (see, e.g. \cite{LaZh06}, \cite{BaLaNo15}, \cite{Ku22a}, and \cite{KuMa23}). Specifically, we assume the following conditions. 

\begin{assumption}\label{Ass:asy-norm-trend}
Let $\overline{A}_n = \max_{1 \leq j \leq d}A_{n,j}$, $\underline{A}_{n1}=\min_{1 \leq j \leq d}A_{n1,j}$, $\overline{A}_{n2}=\max_{1 \leq j \leq d}A_{n2,j}$, and $q>4$ be an integer. As $n,J \to \infty$, 
\begin{itemize}
\item[(i)] $A_n (A_n^{(1)})^{-1}\alpha(\underbar{A}_{n2};A_n) \to 0$,
\item[(ii)] $A_n^{-1}(\overline{A}_{n1})^d(\zeta_J\lambda_J)^4\left(1 + \sum_{k=1}^{\overline{A}_{n1}}k^{2d-1}\alpha_1^{1-4/q}(k)\right) \to 0$,
\item[(iii)] $\left({(\overline{A}_{n1})^{d-1}\overline{A}_{n2} \over A_n^{(1)}} + {A_n^{(1)} \over A_n}\left({\overline{A}_n \over \underbar{A}_{n1}}\right)^{d-1}\right)(\zeta_J\lambda_J)^2\left(1 + \sum_{k=1}^{\overline{A}_{n1}}k^{d-1}\alpha_1^{1-2/q}(k)\right) \to 0$,
\item[(iv)] $A_n^{(1)}(\zeta_J\lambda_J)^2\left(\alpha_1^{1-2/q}(\underbar{A}_{n2}) + \sum_{k=\underbar{A}_{n1}}^{\infty}k^{d-1}\alpha_1^{1-2/q}(k)\right)\alpha_2^{1-2/q}(A_n^{(1)}) \to 0$. 
\end{itemize}
\end{assumption}

Assumption \ref{Ass:asy-norm-trend} is concerned with the large-block-small-block argument and we will see that a wide class of random fields satisfies this assumption (see Section \ref{Appendix:example} for more details). Condition (i) is concerned with approximating $\hat{m} - \bar{m}$ by a sum of independent large blocks where $\bar{m}(\bm{z}) = \psi_J(\bm{z})'(\Psi'_{J,n}\Psi_{J,n}/n + \varsigma_{J,n} I_J)^{-1} \Psi'_{J,n} M_0/n$. Condition (ii) is concerned with the asymptotic normality of the sum of independent large blocks. Conditions (iii) and (iv) are concerned with the asymptotic negligibility of a sum of small blocks. See the proof of Theorem \ref{thm:asy-norm1} for detailed definitions of large and small blocks. 

The following result establishes the asymptotic normality of the series estimator with an arbitrary series basis. 

\begin{theorem}\label{thm:asy-norm1}
Let $\bm{z}_1,\dots, \bm{z}_L \in R_0$. Suppose that Assumptions \ref{Ass:sample}, \ref{Ass:model}(i), (iii), (iv) (with $q>4$), \ref{Ass:sieve1}(i), (iii), \ref{Ass:ze-lam}(i), and \ref{Ass:asy-norm-trend} hold. Additionally, assume that $\inf_{\bm{z} \in R_0}\eta(\bm{z})>0$ and 
\begin{itemize}
\item[(a)] $\zeta_J^2(\lambda_J^2 + \lambda_J^{-2}) \lesssim \sqrt{n/(\log n)^2}$, 
\item[(b)] $\|P_{J,n}\|_\infty = O_p(1)$, $\sqrt{A_n}(\min_{1\leq \ell \leq L}\|\tilde{\psi}_J(\bm{z}_{\ell})\|)^{-1}\inf_{m \in \Psi_J}\|m_0 - m\|_\infty \to 0$, 
\item[(c)] $\|m_0\|_\infty<\infty$, $\zeta_J^3\lambda_J^{-1}\sqrt{\log n/n} \to 0$, and $\varsigma_{J,n} \zeta_J^2\sqrt{A_n}(\min_{1\leq \ell \leq L}\|\tilde{\psi}_J(\bm{z}_{\ell})\|)^{-1} \to 0$
\end{itemize}
as $n,J \to \infty$. Then we have
\begin{align*}
\sqrt{A_n}\Omega_J^{-1/2}\left({\hat{m}(\bm{z}_1) - m_0(\bm{z}_1) \over \sqrt{\Omega_J(\bm{z}_1)}},\dots,{\hat{m}(\bm{z}_L) - m_0(\bm{z}_L) \over \sqrt{\Omega_J(\bm{z}_L)}}\right)\stackrel{d}{\to} N(\bm{0},I_L),
\end{align*}
where $I_L$ is the $L$-dimensional identity matrix, 
\begin{align*}
&\Omega_J = (\Omega_J^{(\ell_1,\ell_2)})_{1 \leq \ell_1, \ell_2 \leq L},\ \Omega_J^{(\ell_1,\ell_2)} = {\Omega_J(\bm{z}_{\ell_1},\bm{z}_{\ell_2}) \over \sqrt{\Omega_J(\bm{z}_{\ell_1})}\sqrt{\Omega_J(\bm{z}_{\ell_2})}},\\
&\Omega_J(\bm{z}) = \tilde{\psi}_J(\bm{z})'G_J \tilde{\psi}_J(\bm{z}),\ \Omega_J(\bm{z}_{\ell_1},\bm{z}_{\ell_2}) = \tilde{\psi}_J(\bm{z}_{\ell_1})'G_J \tilde{\psi}_J(\bm{z}_{\ell_2}),\\
&G_J = \kappa \!\!\int \!(\eta^2(\bm{v}) + \sigma_\varepsilon^2(\bm{v}))\tilde{\psi}_J(\bm{v})\tilde{\psi}_J(\bm{v})'g(\bm{v})d\bm{v} \!+ \! \left(\int \! \eta^2(\bm{v})\tilde{\psi}_J(\bm{v})\tilde{\psi}_J(\bm{v})'g^2(\bm{v})d\bm{v}\! \right) \!\! \left(\int \sigma_{\bm{e}}(\bm{x})d\bm{x}\right).
\end{align*}
\end{theorem}

Conditions (a) and (c) are concerned with replacing $(\tilde{\Psi}'_{J,n}\tilde{\Psi}_{J,n}/n + \varsigma_{J,n} \check{\Psi}_J^{-1})^{-1}$ in the definition of $\hat{m}$ with $I_J$ where $\check{\Psi}_J = \E[\psi_J(\bm{S}_1/A_n)\psi_J(\bm{S}_1/A_n)']$. Condition (b) is required to show the asymptotic negligibility of the bias term of the series estimator. Similar to the discussion on the uniform rates, an advantage of series estimators over kernel estimators is that Theorem \ref{thm:asy-norm1} holds for points on the boundary of $R_0$. Here, we discuss the boundary effect in series estimation of regression functions and highlight its differences from kernel estimators. In the case of kernel estimators, it is necessary to exclude the boundary of the support $R_0$ of the regression function when establishing (pointwise) asymptotic normality. This is because, when evaluating the bias of the estimator, the sufficiently small neighborhood around the point of interest $\bm{x} \in R_0$, which is determined by the bandwidth of the kernel estimator, needs to be fully contained within $R_0$. In contrast, for wavelet or spline bases with compact support, it is possible to perform a uniform bias evaluation over $R_0$, including its boundary. As a result, the (pointwise) asymptotic normality of series estimators also holds at the boundary points of $R_0$. In particular, for wavelet bases with compact support, detailed explanations on their construction can be found in Section 6 of \cite{ChCh15} and \cite{CoDaVi93}, while references for uniform bias evaluation include Proposition 2.5 of \cite{Me92} and \cite{Hu98}. Existing papers on series estimators for time series usually assume that the error terms are a martingale difference sequence and hence a long-run variance does not appear in the asymptotic variance (see, e.g. \cite{ChCh15} and \cite{LiLi20}). On the other hand, the asymptotic variance in theorem \ref{thm:asy-norm1} includes the long-run variance of the random field $\bm{e}$.

Now we provide a consistent estimator of $\Omega_J(\bm{z}_1, \bm{z}_2)$ for $\bm{z}_1, \bm{z}_2 \in R_0$, which enables us to construct confidence intervals for $m_0(\bm{z})$, $\bm{z} \in R_0$. Define $\hat{\Omega}_J(\bm{z}_1,\bm{z}_2) := \psi_J(\bm{z}_1)'\hat{G}_{J}\psi_J(\bm{z}_2)$ where 
\begin{align*}
\hat{G}_J &= {A_n \over n^2}\sum_{i,j=1}^n \left({\Psi'_{J,n}\Psi_{J,n} \over n} + \varsigma_{J,n} I_J\right)^{-1}\psi_J\left({\bm{S}_i \over A_n}\right)\psi_J\left({\bm{S}_j \over A_n}\right)'\left({\Psi'_{J,n}\Psi_{J,n} \over n} + \varsigma_{J,n} I_J\right)^{-1}\\
&\quad \quad \times \left(Y(\bm{S}_i) - \hat{m}\left({\bm{S}_i \over A_n}\right)\!\right)\left(Y(\bm{S}_j) - \hat{m}\left({\bm{S}_j \over A_n}\right)\!\right)\bar{K}_b(\bm{S}_i - \bm{S}_j)
\end{align*}
where $\bar{K}(\bm{w}) :\mathbb{R}^d \to [0,1]$ is a kernel function, $\bar{K}_b(\bm{w}) = \bar{K}\left({w_1 \over b_1},\dots,{w_d \over b_d}\right)$, and $b_j$ is a sequence of positive constants such that $b_j \to \infty$ as $n \to \infty$. We assume the following conditions for $\bar{K}$: 
\begin{assumption}\label{Ass:cov-kernel}
Let $\bar{K}:\mathbb{R}^d \to [0,1]$ is a continuous function such that
\begin{itemize}
\item[(i)] $\bar{K}(\bm{0}) = 1$, $\bar{K}(\bm{w}) = 0$ for $\|\bm{w}\|>1$, 
\item[(ii)] $|1 - \bar{K}(\bm{w})| \leq \bar{C}\|\bm{w}\|$ for $\|\bm{w}\| \leq \bar{c}$ where $\bar{C}$ and  $\bar{c}$ are some positive constants.  
\end{itemize}
\end{assumption}
An example of $\bar{K}$ is the Bartlett kernel: $\bar{K}(\bm{w})=(1-\|\bm{w}\|)$ for $\|\bm{w}\| \leq 1$ and 0 for $\|\bm{w}\|>1$. 

\begin{proposition}\label{prp:var-est-trend}
Let $r \geq d$. Assume $b_j/A_{n,j} = o(J^{-2/d})$, $j=1,\dots,d$ as $n \to \infty$, and $\sqrt{J^3\log n/n} \to 0$ and $\varsigma_{J,n} (\sqrt{JA_n/(\log n)} +\sqrt{n/(J\log n)})\lesssim 1$ as $n,J \to \infty$. Suppose that 
Assumptions \ref{Ass:sample}, \ref{Ass:model}, \ref{Ass:ze-lam}, \ref{Ass:sieve2}, \ref{Ass:asy-norm-trend} (ii)-(iv), and \ref{Ass:cov-kernel} 
 hold with $q >4$ and with $\alpha$-mixing coefficients replaced by $\beta$-mixing coefficients. Additionally, suppose that Conditions (a)-(c) in Theorem \ref{thm:asy-norm1} hold and $\inf_{\bm{z} \in R_0}\eta(\bm{z})>0$. Then, for $\bm{z}_1,\bm{z}_2 \in R_0$, $\|\tilde{\psi}_J(\bm{z}_1)\|^{-1}(\hat{\Omega}_J(\bm{z}_1,\bm{z}_2) - \Omega_J(\bm{z}_1,\bm{z}_2))\|\tilde{\psi}_J(\bm{z}_2)\|^{-1} \stackrel{p}{\to} 0$ as $n, J \to \infty$. 
\end{proposition}

Theorem \ref{thm:asy-norm1} and Proposition \ref{prp:var-est-trend} enable us to construct confidence intervals of $m_0(\bm{z})$. Define $\hat{\Omega}_J(\bm{z}) := \hat{\Omega}_J(\bm{z},\bm{z})$ and consider a confidence interval of the form
\begin{align*}
\hat{C}_{1-\tau}(\bm{z}) = \left[\hat{m}(\bm{z}) -  \sqrt{{\hat{\Omega}_J(\bm{z}) \over A_n}}q_{1-\tau/2}, \hat{m}(\bm{z}) +  \sqrt{{\hat{\Omega}_J(\bm{z}) \over A_n}}q_{1-\tau/2}\right],\ \tau \in (0,1),\ \bm{z} \in R_0,
\end{align*}
where $q_{1-\tau}$ is the $(1-\tau)$-quantile of the standard normal random variable. Then under assumptions in Proposition \ref{prp:var-est-trend}, we have $\Prob(m_0(\bm{z}) \in \hat{C}_{1-\tau}(\bm{z})) \to 1-\tau$ as $n,J \to \infty$.

\section{Examples}\label{Appendix:example}
In this section, we discuss examples of random fields that satisfy our assumptions regarding the dependence structure. In particular, we focus on spatial trend regression models that require more detailed conditions concerning spatial dependence than spatial regression models and show that our regularity conditions are satisfied by a wide class of L\'evy-driven moving average (MA) random fields. We should note that L\'evy-driven MA random fields include many Gaussian and non-Gaussian random fields as special cases and constitute a flexible class of models for spatial data. 

\subsection{L\'evy-driven MA random fields}\label{Sec:LevyMARF}

Let $\mathcal{B}(\mathbb{R}^d)$ denote the Borel subsets on $\mathbb{R}^d$ and let $L=\{L(A): A \in \mathcal{B}(\mathbb{R}^{d})\}$ be an infinitely divisible random measure defined on some probability space $(\Omega, \mathcal{A}, \Prob)$, i.e., a random measure such that 
\begin{itemize}
\item[1.] for each sequence $(E_{m})_{m \in \mathbb{N}}$ of disjoint sets in $\mathcal{B}(\mathbb{R}^{d})$,
\begin{itemize}
\item[(a)] $L(\cup_{m=1}^{\infty}E_{m}) = \sum_{m=1}^{\infty}L(E_{m})$ a.s. whenever $\cup_{m=1}^{\infty}E_{m} \in \mathcal{B}(\mathbb{R}^{d})$,
\item[(b)] $(L(E_{m}))_{m \in \mathbb{N}}$ is a sequence of independent random variables,
\end{itemize}
\item[2.] the random variable $L(A)$ has an infinitely divisible distribution for any $A \in \mathcal{B}(\mathbb{R}^{d})$. 
\end{itemize}
The characteristic function of $L(A)$, which will be denoted by $\varphi_{L(A)}(t)$, has a L\'evy--Khintchine representation of the form $\varphi_{L(A)}(t) = \exp\left(|A|\mathfrak{e}(t)\right)$ with 
\begin{align*}
\mathfrak{e}(t) &= \mathrm{i}t\gamma_{0} - {1 \over 2}t^{2}\sigma_{0} + \int_{\mathbb{R}}\left\{e^{\mathrm{i}tx}-1-\mathrm{i}tx1_{[-1,1]}(x)\right\}\nu_{0}(dx)
\end{align*}
where $\mathrm{i} = \sqrt{-1}$, $\gamma_{0} \in \mathbb{R}$, $0 \leq \sigma_{0} <\infty$, and $\nu_{0}$ is a L\'evy measure with $\int_{\mathbb{R}}\min\{1,x^{2}\}\nu_{0}(dx)<\infty$. If $\nu_0(dx)$ has a Lebesgue density, i.e., $\nu_0(dx) = \nu_0(x)dx$, we call $\nu_0(x)$ as the L\'evy density. The triplet $(\gamma_{0}, \sigma_{0}, \nu_{0})$ is called the L\'evy characteristic of $L$ and it uniquely determines the distribution of the random measure $L$. The following are a couple of examples of L\'evy random measures. 
\begin{itemize}
\item If $\mathfrak{e}(t) = -\sigma_0^2 t^2/2$, then $L$ is a Gaussian random measure. 

\item If $\mathfrak{e}(t) = \lambda \int_{\mathbb{R}}(e^{\mathrm{i}tx}- 1)F(dx)$, where $\lambda>0$ and $F$ is a probability distribution function with no jump at the origin, then $L$ is a compound Poisson random measure with intensity $\lambda$ and jump size distribution $F$. 
\end{itemize}

Let $\theta(\bm{x})$ be a measurable function on $\mathbb{R}^{d}$ such that $\int_{\mathbb{R}^d} |\theta(\bm{x})|d\bm{x}<\infty$ and $\sup_{\bm{x} \in \mathbb{R}^d}|\theta(\bm{x})|<\infty$. A L\'evy-driven MA random field with kernel $\theta$ driven by a L\'evy random measure $L$ is defined by
\begin{align}\label{eq:LevyMA}
e(\bm{x}) &= \int_{\mathbb{R}^{d}}\theta(\bm{x} - \bm{u})L(d\bm{u}),\ \bm{x} \in \mathbb{R}^{d}. 
\end{align}
We refer to \cite{Be96} and \cite{Sa99} for standard references on L\'evy processes, and \cite{RaRo89} and \cite{Ku22a} for details on the theory of infinitely divisible measures and fields.  

Before discussing theoretical results, we look at some examples of random fields defined by (\ref{eq:LevyMA}). Let $a_*(z) = z^{p_0} + a_{1}z^{p_0-1} + \cdots+a_{p_0} = \prod_{i=1}^{p_0}(z - \lambda_{i})$ be a polynomial of degree $p_0$ with real coefficients and distinct negative zeros $\lambda_{1},\hdots,\lambda_{p_0}$, and let $b_*(z) = b_{0} + b_{1}z + \cdots + b_{q_0}z^{q_0} = \prod_{i=1}^{q_0}(z - \xi_{i})$ be a polynomial of degree $q_0$ with real coefficients and real zeros $\xi_{1},\hdots, \xi_{q_0}$ such that $b_{q_0}=1$ and $0\leq q_0 < p_0$ and $\lambda_{i}^{2} \neq \xi_{j}^{2}$ for all $i$ and $j$. Define $a(z) = \prod_{i=1}^{p_0}(z^{2} - \lambda_{i}^{2})$ and $b(z) = \prod_{i=1}^{q_0}(z^{2} - \xi_{i}^{2})$. Then, the L\'evy-driven MA random field driven by an infinitely divisible random measure $L$ with $\theta(\bm{x}) = \sum_{i=1}^{p_0}{b(\lambda_{i}) \over a'(\lambda_{i})}e^{\lambda_{i}\|\bm{x}\|}$, where $a'$ denotes the derivative of the polynomial $a$, is called a univariate (isotropic) CARMA($p_0,q_0$) random field. We refer to \cite{BrMa17} for more details. 

Consider the following decomposition:
\begin{align*}
e(\bm{x}) &= \!\!\int_{\mathbb{R}^{d}}\!\!\!\!\theta(\bm{x} - \bm{u})\mathfrak{t}\left(\|\bm{x} - \bm{u}\| : m_{n}\right)\!L(d\bm{u}) \!+\! \int_{\mathbb{R}^{d}}\!\!\!\!\theta(\bm{x} - \bm{u})\!\!\left(1 - \mathfrak{t}\left(\|\bm{x} - \bm{u}\| : m_{n}\right)\right)\!L(d\bm{u})\\
&=: e_{1, m_{n}}(\bm{x}) + e_{2, m_{n}}(\bm{x}),
\end{align*}
where $m_{n}$ is a sequence of positive constants such that $m_{n} \to \infty$ as $n \to \infty$ and $\mathfrak{t}(\cdot:c) : \mathbb{R} \to [0,1]$ is a truncation function defined by
\begin{align*}
\mathfrak{t}(x:c) = 
\begin{cases}
1 & \text{if $|x| \leq c/4$},\\
-{4 \over c}\left(x-{c \over 2}\right) & \text{if $c/4 < |x| \leq c/2$},\\
0 & \text{if $x>c/2$}.
\end{cases}
\end{align*} 
The random field $\bm{e}_{1,m_n} = \{e_{1,m_{n}}(\bm{x})  : \bm{x} \in \mathbb{R}^d \}$ is $m_{n}$-dependent (with respect to the $\ell^{2}$-norm), i.e., $e_{1, m_n}(\bm{x}_{1})$ and $e_{1, m_n}(\bm{x}_{2})$ are independent if $\| \bm{x}_1-\bm{x}_2\| \geq m_{n}$. Also, if the tail of the kernel function $\theta(\bm{x})$ decays sufficiently fast, then we can see that the random field $\bm{e}_{2,m_n}=\{ e_{2, m_{n}}(\bm{x}) : \bm{x} \in \mathbb{R}^{d}\}$ is asymptotically negligible. In such case, we can replace the random field $\bm{e}$ with the $m_{n}$-dependent random field $\bm{e}_{1,m_n}$ and verify our conditions on the dependence structure for $\bm{e}_{1,m_n}$.

\begin{proposition}\label{prp:LevyMA-ex}
Suppose that Assumption \ref{Ass:sieve2} holds and $n^{d/2r}J^{-1} = o(1)$. Consider a L\'evy-driven MA random field $\bm{e}$ defined by (\ref{eq:LevyMA}). Assume that $\theta(\bm{x}) = r_0 e^{-r_{1} \|\bm{x}\|}$ where $|r_0| > 0$ and $r_1>0$. Additionally, assume that 
\begin{itemize}
\item[(a)] the random measure $L(\cdot)$ is Gaussian with triplet $(0, \sigma_0 , 0)$ or
\item[(b)] the random measure $L(\cdot)$ is non-Gaussian with triplet $(\gamma_0, 0 , \nu_0)$, $\E[L(A)] = 0$ for any $A \in \mathcal{B}(\mathbb{R}^d)$, and the L\'evy density $\nu_0(x)$ given by
\begin{align}\label{eq:LS-Levy} 
\nu_0(x) &= {1 \over |x|^{1 + \beta_0}}\left(C_0e^{-c_0|x|^{\alpha_0}} + {C_1 \over (1 + C_2|x|^{\beta_1})^{\alpha_1}}\right)1_{\mathbb{R} \backslash \{0\}}(x), 
\end{align}
where $\alpha_0>0$, $\alpha_1 > 0$, $\beta_0 \in [-1,2)$, $\beta_1 >0$, $\beta_0 + \alpha_1\beta_1>9$, $c_0>0$, $C_0 \geq 0$, $C_1\geq 0$, $C_2>0$, and $C_0 + C_1>0$.
\end{itemize}
Then $\bm{e}_{2,m_n}$ is asymptotically negligible, that is, we can replace $\bm{e}$ with $\bm{e}_{1,m_n}$ in the results in Section \ref{Sec:reg-trend}.
Further, $\bm{e}_{1,m_n}$ satisfies Assumptions \ref{Ass:model} (i), (ii), \ref{Ass:ze-lam}, \ref{Ass:asy-norm-trend}, and $(\zeta_J\lambda_J)^2 \lesssim \sqrt{n/(\log n)^2}$ with $A_{n,j} \sim n^{\mathfrak{c}_{0}/d}$, $A_{n1,j} = A_{n,j}^{\mathfrak{c}_1}$, $A_{n2,j} = A_{n1,j}^{\mathfrak{c}_2}$, $m_{n} = \underline{A}_{n2}^{1/2}$ where $\mathfrak{c}_{0}$, $\mathfrak{c}_{1}$, and $\mathfrak{c}_{2}$ are positive constants such that
\begin{align*}
&J = o\left(\min\{n^{\mathfrak{c}_0\left\{1-\mathfrak{c}_1(1+\mathfrak{c}_2)\right\}/2}, n^{\mathfrak{c}_0 \mathfrak{c}_1\{(1-\mathfrak{c}_2)/d - \mathfrak{c}_2/2\}}, n^{\zeta_0(1-2\mathfrak{c}_1)- {2/9}}, n^{7\mathfrak{c}_0/9}, n^{1/2}\}\right).
\end{align*}
\end{proposition}
When $d = 2$, the conditions on $\mathfrak{c}_0,\mathfrak{c}_1,\mathfrak{c}_2$ and $J$ are typically satisfied when $\mathfrak{c}_0 = 1$, $\mathfrak{c}_1 = 1/4$, $\mathfrak{c}_2 \in \left(0, 1/2\right)$, $n^{d/2r}J^{-1}= o(1)$, and $Jn^{-(1-\mathfrak{c}_2)/4} = o(1)$. Condition (b) implies that a wide class of non-Gaussian random fields including compound Poisson ($\beta_0 \in [-1,0)$), variance Gamma ($\alpha_0=1$, $\beta_0=0$, $C_1=0$), tempered stable ($\beta_0 \in (0,1)$, $C_1=0$), and normal inverse Gaussian ($\beta_0=1$ for $|x|<1$, $\beta_0=1/2$ for $|x| \geq 1$, $C_1=0$) random fields satisfies our assumptions. See also \cite{KaKu20} and \cite{KuKaSh24} for more discussion on non-Gaussian L\'evy-driven random fields.  It is straightforward to extend Proposition \ref{prp:LevyMA-ex} to the case that $\theta(\bm{x})$ is a finite sum of kernel functions with exponential decay. Therefore, our results in Section \ref{Sec:reg-trend} can be applied to CARMA($p_0$, $q_0$) random fields. Further, extending the results to anisotropic CARMA random fields (cf. \cite{BrMa17}) is also straightforward. 

\section{Application to population data}\label{Sec:real-data}

This section aims at estimating nighttime population in Tokyo by the nonparametric regression model (\ref{eq:Spatial-Trend-Reg}) with the standard errors evaluated by Theorem \ref{thm:asy-norm1}. We also construct the confidence surfaces of the human population in Tokyo to demonstrate the practical usage of Theorem \ref{thm:asy-norm1}.

NTT DoCoMo, a mobile phone company in Japan, provides a spatio-temporal dataset that counts numbers of people in 1000m meshes every one hour all over Japan since 2016. We collected the population data over the 6,100 meshes in Tokyo, which is the area of 102km $\times$ 74km rectangular including the 23 central districts of Tokyo, at 4:00 am during one month period of May, 2016 and constructed the Tokyo nighttime population dataset in May by averaging them monthly, which were 5,975 observations in total by excluding NAs. Transforming the longitudes and latitudes of the 6,100 meshes to the 2-dimensional Cartesian coordinate by the formula in Geospatial Information Authority of Japan, we regard the 5,975 observations as a kind of spatial data on the 2-dim Euclidean space.

We applied the cubic B-spline functions over the $A_n=102\times 74$ ($A_{n,1}=102, A_{n,2}=74$) rectangular to the spatial data of human population in Tokyo in May, 2016, where the cubic B-spline function with $J=900$ basis. We estimated the spatial trend $m_0$ in (\ref{eq:Spatial-Trend-Reg}) by the series ridge estimator in Section \ref{Subsec:STR-unif-rate} with $n=5,975, \varsigma_{J,n}=0.5/n$. We used the Bertlett kernel to estimate asymptotic variance in Proposition \ref{prp:var-est-trend} with $b_1=0.1A_{n,1}$, $b_2=0.1A_{n,2}$. For confidence surfaces, we first calculate the 95\% confidence interval at each mesh point and then linearly interpolate between them.

\begin{figure}
  \begin{center}
  \includegraphics[width=12cm, height=8cm]{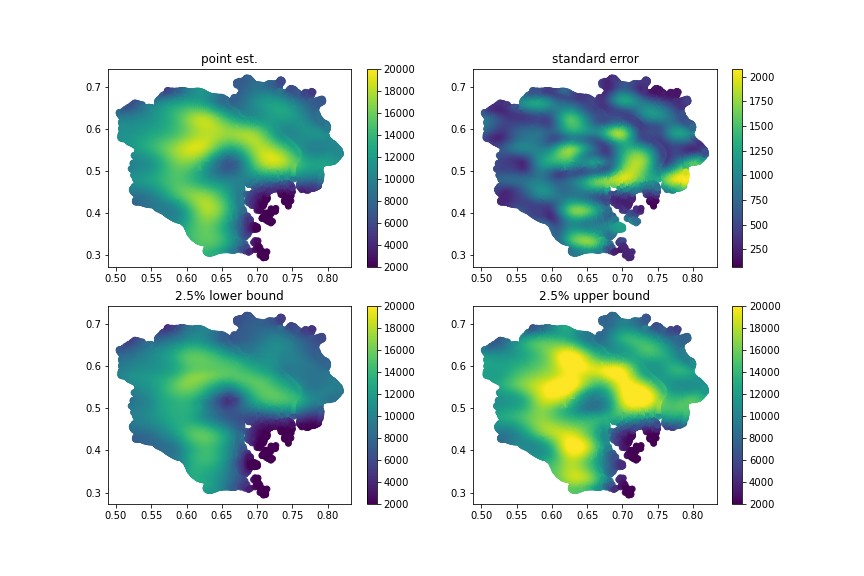}
  \end{center}
  \caption{Point estimator (upper left), standard error (upper right), 2.5\% lower bound (lower left) and 2.5\% upper bound (lower right), for the trend surface of human population in the Tokyo 23 central districts at 4:00 AM in May.}\label{fig:density}
\end{figure}

Figure \ref{fig:density} shows the point estimation for the trend surface together with the standard errors and their 95\% confidence surface estimation, which identifies the central empty area of Emperor's palace surrounded by highest dense population areas of Shibuya, Shinjuku, Ikebukuro, Asakusa and Ueno.

\section{Conclusion}\label{Sec:conclusion}

In this paper, we have advanced the statistical theory of nonparametric regression for irregularly spaced spatial data. Specifically, we considered a spatial trend regression model and a spatial regression model with spatially dependent covariates, which are defined on a sampling region $R_n \subset \mathbb{R}^d$ and established uniform and $L^2$ convergence rates and asymptotic normality of series ridge estimators for the trend and regression functions under a stochastic sampling design. We also provided examples of random fields that satisfy our dependence conditions. In particular, our assumptions are satisfied by a wide class of L\'evy-driven moving average random fields, including popular Gaussian random fields as special cases. We are hopeful that our theory can be extended to a class of series or general (nonlinear) sieve estimators that minimize penalized loss functions as discussed in \cite{Sh97} and \cite{ChKLiWa23}. We leave these topics for future investigation.

\newpage
\begin{appendix}

\section{Proofs for Section \ref{Sec:reg-trend}}

We provide proofs for Proposition \ref{prp:variance} (outline with the first and second steps of the proof), Corollary \ref{cor:unif-rate1}, Theorem \ref{thm:unif-rate1}, Proposition \ref{prp:L^2-rate1}, and Corollary \ref{cor:L^2-rate1}. 
Proofs for Proposition \ref{prp:variance} (the remaining part of the proof), Theorem \ref{thm:asy-norm1}, and Proposition \ref{prp:var-est-trend} will be provided in the supplementary material. 

\subsection{Proof of Proposition \ref{prp:variance}}

Define $\check{\Psi}_J = \E[\psi_J(\bm{S}_i/A_n)\psi_J(\bm{S}_i/A_n)']$, $v_n = (v_{n,1},\dots,v_{n,n})'$ with $v_{n,i} = e_{n,i} + \varepsilon_{n,i}$, and \begin{align*}
\bar{m}(\bm{z}) &= \psi_J(\bm{z})'(\Psi'_{J,n}\Psi_{J,n} + \varsigma_{J,n} nI_J)^{-1} \Psi'_{J,n} M_0 = \tilde{\psi}_J(\bm{z})'(\tilde{\Psi}'_{J,n}\tilde{\Psi}_{J,n} + \varsigma_{J,n} n\check{\Psi}_J^{-1})^{-1} \tilde{\Psi}'_{J,n} M_0. 
\end{align*} 
Note that
\begin{align*}
\hat{m}(\bm{z}) - \tilde{m}(\bm{z}) &= \{\hat{m}(\bm{z}) - \bar{m}(\bm{z})\} + \{\bar{m}(\bm{z}) - \tilde{m}(\bm{z})\} =: \check{m}(\bm{z}) + \dot{m}(\bm{z}).
\end{align*}

We show (\ref{eq:var-unif-rate}) in several steps. Let $\mathcal{S}_n$ be the smallest subset of $R_0$ such that for each $\bm{z} \in R_0$ there exists a $\bm{z}_n \in R_0$ with $\|\bm{z} - \bm{z}_n\| \leq \eta_2 n^{-\eta_1}$. In Step 1, we will show that for any $C \geq 1$, we have
\begin{align*}
&\Prob\left(\|\check{m}\|_\infty \geq 8C\zeta_J\lambda_J\sqrt{{\log n \over A_n}}\right)\\
&\leq \Prob\left(\max_{\bm{z}_n \in \mathcal{S}_n}\left|\tilde{\psi}_J(\bm{z}_n)'\left\{(\tilde{\Psi}'_{J,n}\tilde{\Psi}_{J,n}/n +\varsigma_{J,n} \check{\Psi}_J^{-1})^{-1} - I_J\right\} \tilde{\Psi}'_{J,n}v_n/n \right| \geq 2C\zeta_J\lambda_J\sqrt{{\log n \over A_n}}\right)\\
&\quad + \Prob\left(\max_{\bm{z}_n \in \mathcal{S}_n}\left|\tilde{\psi}_J(\bm{z}_n)'\tilde{\Psi}'_{J,n}v_n/n \right| \geq 2C\zeta_J\lambda_J\sqrt{{\log n \over A_n}}\right) + o(1) =: P_{n,1} + P_{n,2} + o(1). 
\end{align*}
In Step 2, we will show that $P_{n,1}$ can be made arbitrarily small for large enough $C \geq 1$. In Step 3, we will show $P_{n,2} = o(1)$ as $n \to \infty$. In Step 4, we will show 
\begin{align}\label{eq:ridge-bias}
\|\dot{m}\|_\infty = O_p\left(\varsigma_{J,n} \zeta_J^2\|m_0\|_\infty  \right). 
\end{align}
Proofs for Steps 3 and 4 will be provided in Section \ref{Appendix:sec2} of the supplementary material. Combining the results in Steps 1-4, we obtain the desired result.

(Step 1) By the mean value theorem, for any $\bm{z},\bm{z}^* \in R_0$ we have
\begin{align*}
|\check{m}(\bm{z}) - \check{m}(\bm{z}^*)| &= \left|(\tilde{\psi}_J(\bm{z})-\tilde{\psi}_J(\bm{z}^*))'(\tilde{\Psi}'_{J,n}\tilde{\Psi}_{J,n}/n + \varsigma_{J,n} \check{\Psi}_J^{-1})^{-1} \tilde{\Psi}'_{J,n}v_n/n\right|\\
&= \left|(\bm{z}-\bm{z}^*)'\nabla \tilde{\psi}_J(\bm{z}^{**})'(\tilde{\Psi}'_{J,n}\tilde{\Psi}_{J,n}/n + \varsigma_{J,n} \check{\Psi}_J^{-1})^{-1} \tilde{\Psi}'_{J,n}v_n/n\right|\\
&\leq \|\nabla \tilde{\psi}_J(\bm{z}^{**})\|\|\bm{z} - \bm{z}^*\|\|(\tilde{\Psi}'_{J,n}\tilde{\Psi}_{J,n}/n + \varsigma_{J,n} \check{\Psi}_J^{-1})^{-1}\| \|\tilde{\Psi}'_{J,n}v_n/n\|\\
&\leq \|\check{\Psi}_J^{-1/2}\|\|\nabla \psi_J(\bm{z}^{**})\|\|\bm{z} - \bm{z}^*\|\|(\tilde{\Psi}'_{J,n}\tilde{\Psi}_{J,n}/n + \varsigma_{J,n} \check{\Psi}_J^{-1})^{-1}\| \|\tilde{\Psi}'_{J,n}v_n/n\|\\
&= \lambda_{\text{max}}(\check{\Psi}_J^{-1/2})\|\nabla \psi_J(\bm{z}^{**})\|\|\bm{z} - \bm{z}^*\|\|(\tilde{\Psi}'_{J,n}\tilde{\Psi}_{J,n}/n + \varsigma_{J,n} \check{\Psi}_J^{-1})^{-1}\| \|\tilde{\Psi}'_{J,n}v_n/n\|\\
&\leq C_{\nabla}\lambda_J J^\omega\|\bm{z} - \bm{z}^*\|\|(\tilde{\Psi}'_{J,n}\tilde{\Psi}_{J,n}/n + \varsigma_{J,n} \check{\Psi}_J^{-1})^{-1}\| \|\tilde{\Psi}'_{J,n}v_n/n\|
\end{align*}
for some $\bm{z}^{**}$ in the segment between $\bm{z}$ and $\bm{z}^*$ and some finite constant $C_{\nabla}$ which is independent of $\bm{z},\bm{z}^*,n$, and $J$. 

Observe that
\begin{align}
\E\left[\|\tilde{\Psi}'_{J,n}v_n/n\|^2\right] &= {1 \over n^2}\sum_{i_1=1}^n\sum_{i_2=1}^n \E\left[\psi'_J\left({\bm{S}_{i_1} \over A_n}\right)\check{\Psi}_J^{-1}\psi_J\left({\bm{S}_{i_2} \over A_n}\right)v_{n,i_1}v_{n,i_2}\right] \nonumber \\
&= {1 \over n^2}\sum_{i_1=1}^n\sum_{i_2=1}^n\E\left[\psi'_J\left({\bm{S}_{i_1} \over A_n}\right)\check{\Psi}_J^{-1}\psi_J\left({\bm{S}_{i_2} \over A_n}\right)e_{n,i_1}e_{n,i_2}\right] \nonumber \\
&= {1 \over n}\E\left[\eta^2\left({\bm{S}_1 \over A_n}\right)\psi'_J\left({\bm{S}_1 \over A_n}\right)\check{\Psi}_J^{-1}\psi_J\left({\bm{S}_1 \over A_n}\right)e^2(\bm{0})\right] \nonumber \\
&\quad + {n(n-1) \over n^2}\E\left[\psi'_J\left({\bm{S}_1 \over A_n}\right)\check{\Psi}_J^{-1}\psi_J\left({\bm{S}_2 \over A_n}\right)\eta\left({\bm{S}_1 \over A_n}\right)\eta\left({\bm{S}_2 \over A_n}\right)e(\bm{S}_1)e(\bm{S}_2)\right] \nonumber \\
&= {1 \over n}\E\left[\eta^2\left({\bm{S}_1 \over A_n}\right)\psi'_J\left({\bm{S}_1 \over A_n}\right)\check{\Psi}_J^{-1}\psi_J\left({\bm{S}_1 \over A_n}\right)\right] \nonumber \\
&\quad + {n(n-1) \over n^2}\E\left[\psi'_J\left({\bm{S}_1 \over A_n}\right)\check{\Psi}_J^{-1}\psi_J\left({\bm{S}_2 \over A_n}\right)\eta\left({\bm{S}_1 \over A_n}\right)\eta\left({\bm{S}_2 \over A_n}\right)\sigma_{\bm{e}}(\bm{S}_1-\bm{S}_2)\right] \nonumber \\
&=: I_{n,1} + I_{n,2}. \label{eq: In12}
\end{align}
For $I_{n,1}$, we have
\begin{align}
I_{n,1} &\lesssim {1 \over nA_n}\int \psi'_J\left({\bm{s} \over A_n}\right)\check{\Psi}_J^{-1}\psi_J\left({\bm{s} \over A_n}\right)g\left({\bm{s} \over A_n}\right)d\bm{s} = {1 \over n}\int \psi'_J(\bm{z})\check{\Psi}_J^{-1}\psi_J(\bm{z})g(\bm{z})d\bm{z} \lesssim {\zeta_J^2\lambda_J^2 \over n}. \label{ineq: In1}
\end{align}
Define $\bar{R}_n = \{\bm{x} \in \mathbb{R}^d: \bm{x} = \bm{x}_1 - \bm{x}_2\ \text{for some}\ \bm{x}_1, \bm{x}_2 \in R_n\}$. For $I_{n,2}$, we have
\begin{align}
&|I_{n,2}| \nonumber \\ 
&\leq {n(n-1) \over n^2A_n^2}\int \left|\psi'_J\left({\bm{s}_1 \over A_n}\right)\check{\Psi}_J^{-1}\psi_J\left({\bm{s}_2 \over A_n}\right)\sigma_{\bm{e}}(\bm{s}_1-\bm{s}_2)\right|\eta\left({\bm{s}_1 \over A_n}\right)\eta\left({\bm{s}_2 \over A_n}\right)g\left({\bm{s}_1 \over A_n}\right)g\left({\bm{s}_2 \over A_n}\right)d\bm{s}_1 d\bm{s}_2 \nonumber \\
&\lesssim {n(n-1) \over n^2}\int \left|\psi'_J(\bm{z}_1)\check{\Psi}_J^{-1}\psi_J(\bm{z}_2)\sigma_{\bm{e}}(A_n(\bm{z}_1-\bm{z}_2))\right|g(\bm{z}_1)g(\bm{z}_2)d\bm{z}_1 d\bm{z}_2 \nonumber \\
&= {n(n-1) \over n^2A_n}\int_{\bar{R}_n}|\sigma_{\bm{e}}(\bm{x})|\left(\int \left|\psi'_J\left({\bm{x} \over A_n} + \bm{z}_2\right)\check{\Psi}_J^{-1}\psi_J(\bm{z}_2)\right|g\left({\bm{x} \over A_n} + \bm{z}_2\right)g(\bm{z}_2)d\bm{z}_2\right)d\bm{x} \nonumber \\
&\lesssim {\zeta_J^2\lambda_J^2 \over A_n}\left(\int |\sigma_{\bm{e}}(\bm{x})|d\bm{x}\right)\left(\int g^2(\bm{z})d\bm{z}\right). \label{ineq:In2}
\end{align}
Combining (\ref{eq: In12}), (\ref{ineq: In1}), and (\ref{ineq:In2}) and applying Markov's inequality, we obtain
\begin{align}\label{eq:Psi-v-norm}
\|\tilde{\Psi}'_{J,n}v_n/n\| &= O_p\left({\zeta_J\lambda_J \over \sqrt{ A_n}}\right). 
\end{align}
From Lemma \ref{lem:psi} and Assumptions \ref{Ass:ze-lam} (i), We have 
\begin{align}\label{ineq:Psi-norm}
\|(\tilde{\Psi}'_{J,n}\tilde{\Psi}_{J,n}/n + \varsigma_{J,n} \check{\Psi}_J^{-1})^{-1}\| &\leq {1 \over 1 - \|(\tilde{\Psi}'_{J,n}\tilde{\Psi}_{J,n}/n + \varsigma_{J,n} \check{\Psi}_J^{-1}) - I_J\|} \nonumber \\
&\leq {1 \over 1 - \|(\tilde{\Psi}'_{J,n}\tilde{\Psi}_{J,n}/n) - I_J\| - \varsigma_{J,n}\lambda_J^2} = O_p(1). 
\end{align}
Together with (\ref{eq:Psi-v-norm}) and (\ref{ineq:Psi-norm}), we have $\limsup_{n \to \infty}\Prob(C_\nabla\|(\tilde{\Psi}'_{J,n}\tilde{\Psi}_{J,n}/n + \varsigma_{J,n} \check{\Psi}_J^{-1})^{-1}\| \|\tilde{\Psi}'_{J,n}v_n/n\|>\bar{M}) = 0$ for any fixed $\bar{M}>0$ since Assumption \ref{Ass:ze-lam} (i) implies ${\zeta_J\lambda_J \over \sqrt{ A_n}} = o(1)$. Let $\mathcal{A}_n$ denote the event on which $C_\nabla\|(\tilde{\Psi}'_{J,n}\tilde{\Psi}_{J,n}/n + \varsigma_{J,n} \check{\Psi}_J^{-1})^{-1}\| \|\tilde{\Psi}'_{J,n}v_n/n\| \leq \bar{M}$ and observe that $\Prob(\mathcal{A}_n^c) = o(1)$. On $\mathcal{A}_n$, for any $C \geq 1$, a finite positive $\eta_1 = \eta_1(C)$ and $\eta_2 = \eta_2(C)$ can be chosen such that 
\begin{align*}
C_{\nabla}\lambda_J J^{\omega}\|\bm{z} - \bm{z}^*\|\|(\tilde{\Psi}'_{J,n}\tilde{\Psi}_{J,n}/n + \varsigma_{J,n} \check{\Psi}_J^{-1})^{-1}\| \|\tilde{\Psi}'_{J,n}v_n/n\| &\leq C\zeta_J\lambda_J\sqrt{{\log n \over A_n}}
\end{align*} 
whenever $\|\bm{z} - \bm{z}^*\| \leq \eta_2 n^{-\eta_1}$. Let $\mathcal{S}_n$ be the smallest subset of $R_0$ such that for each $\bm{z} \in R_0$ there exists a $\bm{z}_n \in R_0$ with $\|\bm{z} - \bm{z}_n\| \leq \eta_2 n^{-\eta_1}$. For any $\bm{z} \in R_0$ let $\bm{z}_n(\bm{z})$ denote the $\bm{z}_n \in \mathcal{S}_n$ nearest to $\bm{z}$ in Euclidean distance. Then on $\mathcal{A}_n$, we have $|\check{m}(\bm{z}) - \check{m}(\bm{z}_n(\bm{z}))| \leq C\zeta_J\lambda_J\sqrt{\log n/A_n}$ for any $\bm{z} \in R_0$. Then we have
\begin{align*}
\Prob\left(\|\check{m}\|_\infty \geq 8C\zeta_J\lambda_J\sqrt{{\log n \over A_n}}\right)
&\leq \Prob\left(\left\{\|\check{m}\|_\infty \geq 8C\zeta_J\lambda_J\sqrt{{\log n \over A_n}}\right\} \cap \mathcal{A}_n\right) + \Prob(\mathcal{A}_n^c)\\
&\leq \Prob\left(\left\{\sup_{\bm{z} \in R_0}|\check{m}(\bm{z}) - \check{m}(\bm{z}_n(\bm{z}))| \geq 4C\zeta_J\lambda_J\sqrt{{\log n \over A_n}}\right\} \cap \mathcal{A}_n\right)\\
&\quad + \Prob\left(\left\{\max_{\bm{z}_n \in \mathcal{S}_n}|\check{m}(\bm{z}_n)| \geq 4C\zeta_J\lambda_J\sqrt{{\log n \over A_n}}\right\} \cap \mathcal{A}_n\right) + \Prob(\mathcal{A}_n^c)\\
&= \Prob\left(\left\{\max_{\bm{z}_n \in \mathcal{S}_n}|\check{m}(\bm{z}_n)| \geq 4C\zeta_J\lambda_J\sqrt{{\log n \over A_n}}\right\} \cap \mathcal{A}_n\right) + o(1) =: P_n + o(1). 
\end{align*}
For $P_n$, we have
\begin{align*}
P_n &\leq \Prob\left(\max_{\bm{z}_n \in \mathcal{S}_n}|\check{m}(\bm{z}_n)| \geq 4C\zeta_J\lambda_J\sqrt{{\log n \over A_n}}\right)\\
&\leq \Prob\left(\max_{\bm{z}_n \in \mathcal{S}_n}\left|\tilde{\psi}_J(\bm{z}_n)'\left\{(\tilde{\Psi}'_{J,n}\tilde{\Psi}_{J,n}/n + \varsigma_{J,n} \check{\Psi}_J^{-1})^{-1} - I_J\right\} \tilde{\Psi}'_{J,n}v_n/n \right| \geq 2C\zeta_J\lambda_J\sqrt{{\log n \over A_n}}\right)\\
&\quad + \Prob\left(\max_{\bm{z}_n \in \mathcal{S}_n}\left|\tilde{\psi}_J(\bm{z}_n)'\tilde{\Psi}'_{J,n}v_n/n \right| \geq 2C\zeta_J\lambda_J\sqrt{{\log n \over A_n}}\right) =: P_{n,1} + P_{n,2}. 
\end{align*}

(Step 2) Now we show that $P_{n,1}$ can be made arbitrarily small for large enough $C \geq 1$. The Cauchy-Schwarz inequality yields
\begin{align*}
&\left|\tilde{\psi}_J(\bm{z}_n)'\left\{(\tilde{\Psi}'_{J,n}\tilde{\Psi}_{J,n}/n + \varsigma_{J,n} \check{\Psi}_J^{-1})^{-1} - I_J\right\} \tilde{\Psi}'_{J,n}v_n/n \right| \\
&\leq \|\tilde{\psi}_J(\bm{z}_n)\| \|(\tilde{\Psi}'_{J,n}\tilde{\Psi}_{J,n}/n + \varsigma_{J,n} \check{\Psi}_J^{-1})^{-1} - I_J\| \|\tilde{\Psi}'_{J,n}v_n/n\|\\
&\lesssim \zeta_J\lambda_J\|(\tilde{\Psi}'_{J,n}\tilde{\Psi}_{J,n}/n + \varsigma_{J,n} \check{\Psi}_J^{-1})^{-1} - I_J\| \times O_p\left(\zeta_J\lambda_J\sqrt{{1 \over A_n}}\right)
\end{align*}
uniformly for $\bm{z}_n \in \mathcal{S}_n$. Let $\mathcal{B}_n$ denote the event on which $\|(\tilde{\Psi}'_{J,n}\tilde{\Psi}_{J,n}/n + \varsigma_{J,n} \check{\Psi}_J^{-1}) - I_J\| \leq 1/2$ and observe that $\Prob(\mathcal{B}_n^c) = o(1)$. On $\mathcal{B}_n$, we have
\begin{align}\label{ineq:op-inv}
\|(\tilde{\Psi}'_{J,n}\tilde{\Psi}_{J,n}/n + \varsigma_{J,n} \check{\Psi}_J^{-1})^{-1}\| &\leq {1 \over 1 - \|(\tilde{\Psi}'_{J,n}\tilde{\Psi}_{J,n}/n + \varsigma_{J,n} \check{\Psi}_J^{-1}) - I_J\|} \leq 2. 
\end{align} 
Then we have 
\begin{align*}
\|(\tilde{\Psi}'_{J,n}\tilde{\Psi}_{J,n}/n + \varsigma_{J,n} \check{\Psi}_J^{-1})^{-1} - I_J\| &\leq \left\|(\tilde{\Psi}'_{J,n}\tilde{\Psi}_{J,n}/n + \varsigma_{J,n} \check{\Psi}_J^{-1})^{-1}\left\{(\tilde{\Psi}'_{J,n}\tilde{\Psi}_{J,n}/n + \varsigma_{J,n} \check{\Psi}_J^{-1}) - I_J\right\}\right\|\\
&\leq \|(\tilde{\Psi}'_{J,n}\tilde{\Psi}_{J,n}/n + \varsigma_{J,n} \check{\Psi}_J^{-1})^{-1}\| \|(\tilde{\Psi}'_{J,n}\tilde{\Psi}_{J,n}/n + \varsigma_{J,n} \check{\Psi}_J^{-1}) - I_J\|\\
&\leq 2 \|(\tilde{\Psi}'_{J,n}\tilde{\Psi}_{J,n}/n + \varsigma_{J,n} \check{\Psi}_J^{-1}) - I_J\|.
\end{align*}
Hence we have 
\begin{align*}
&\left|\tilde{\psi}_J(\bm{z}_n)'\left\{(\tilde{\Psi}'_{J,n}\tilde{\Psi}_{J,n}/n + \varsigma_{J,n} \check{\Psi}_J^{-1})^{-1} - I_J\right\} \tilde{\Psi}'_{J,n}v_n/n \right| \\
&\lesssim \zeta_J\lambda_J\|(\tilde{\Psi}'_{J,n}\tilde{\Psi}_{J,n}/n + \varsigma_{J,n} \check{\Psi}_J^{-1}) - I_J\| \times O_p\left(\zeta_J\lambda_J\sqrt{{1 \over A_n}}\right) = O_p\left(\zeta_J\lambda_J\sqrt{{\log n \over A_n}}\right).
\end{align*}
Thus, $P_{n,1}$ can be made arbitrarily small for large enough $C\geq 1$. 

(Step 3)(Step 4) See Section \ref{Appendix:sec2} of the supplementary material for the proofs of Steps 3 and 4.

\subsection{Proof of Corollary \ref{cor:unif-rate1}}

Note that $\tilde{m} = P_{J,n}m_0$. Therefore, for any $m \in \Psi_J$, we have
\begin{align*}
\|\hat{m} - m_0\|_\infty &= \|\hat{m} - \tilde{m} + \tilde{m} - m_0\|_\infty\\
&\leq \|\hat{m} - \tilde{m}\|_\infty + \|P_{J,n}m_0 - m + m - m_0\|_\infty\\
&=  \|\hat{m} - \tilde{m}\|_\infty + \|P_{J,n}(m_0 - m) + m - m_0\|_\infty\\
&\leq \|\hat{m} - \tilde{m}\|_\infty + \|P_{J,n}(m_0 - m)\|_\infty + \|m - m_0\|_\infty\\
&\leq \|\hat{m} - \tilde{m}\|_\infty + (1 + \|P_{J,n}\|_\infty)\|m - m_0\|_\infty.
\end{align*}
Taking the infimum over $m \in \Psi_J$ yields the desired result. 

\subsection{Proof of Theorem \ref{thm:unif-rate1}}

For the bias term, it is well-known that $\inf_{m \in \Psi_J}\|m_0 - m\|_\infty = O(J^{-r/d})$ under Assumption \ref{Ass:sieve2} (e.g. \cite{Hu98}). Therefore, it is sufficient to show 
\begin{align}\label{eq:emp-projection}
\|P_{J,n}\|_\infty &= O_p(1).
\end{align}

First, we show (\ref{eq:emp-projection}) when $\Psi_J = \text{BSpl}(J,R_0,\varrho)$. For this, we will show that
\begin{align}\label{eq:emp-true-L^2}
\sup_{\psi \in \Psi_J}\left|{n^{-1}\sum_{i=1}^{n}\psi(\bm{S}_i/A_n)^2 - \E[\psi(\bm{S}_1/A_n)^2] \over \E[\psi(\bm{S}_1/A_n)^2]}\right| = \left\|\tilde{\Psi}'_{J,n}\tilde{\Psi}_{J,n}/n - I_J\right\|^2 = o_p(1).
\end{align}
This implies that the empirical $L^2$ norm $\|\cdot\|_{L^2,n}$ and true $L^2$ norm $\|\cdot\|_{L^2(g)}$ are equivalent over $\Psi_J$ with probability approaching one. Then from Corollary A.1 in \cite{Hu03}, we obtain (\ref{eq:emp-projection}). 

Now we show (\ref{eq:emp-true-L^2}). Recall that $\check{\Psi}_J = \E[\psi_J(\bm{S}_1/A_n)\psi_J(\bm{S}_1/A_n)']$. Observe that 
\begin{align*}
\sup_{\psi \in \Psi_J, \E[\psi(\bm{S}_1/A_n)^2]=1} \left|{1 \over n}\!\sum_{i=1}^{n}\!\psi \!\left({\bm{S}_i \over A_n}\right)^2\!\!-1\right| 
&= \sup_{\bm{x} \in \mathbb{R}^J, \|\check{\Psi}_J^{1/2}\bm{x}\| = 1}\!\!\!\!\!\!\!\!\!\!\!|\bm{x}'(\Psi'_{J,n}\Psi_{J,n}/n - \check{\Psi}_J)\bm{x}|\\
&= \sup_{\bm{x} \in \mathbb{R}^J, \|\check{\Psi}_J^{1/2}\bm{x}\| = 1}\!\!\!\!\!\!\!\!\!\!\!|\bm{x}'\check{\Psi}_J^{1/2}(\check{\Psi}_J^{-1/2}(\Psi'_{J,n}\Psi_{J,n}/n)\check{\Psi}_J^{-1/2} \!- I_J)\check{\Psi}_J^{1/2}\bm{x}|\\
&= \sup_{\bm{v} \in \mathbb{R}^J, \|\bm{v}\| = 1}|\bm{v}'(\tilde{\Psi}'_{J,n}\tilde{\Psi}_{J,n}/n) - I_J)\bm{v}|\\
&= \|\tilde{\Psi}'_{J,n}\tilde{\Psi}_{J,n}/n - I_J\|^2.
\end{align*}

Next, we show (\ref{eq:emp-projection}) when $\Psi_J = \text{Wav}(J,R_0,\varrho)$. This follows from almost the same argument in the proof of Theorem 5.2 in \cite{ChCh15}.

\subsection{Proof of Proposition \ref{prp:L^2-rate1}}

Observe that 
\begin{align*}
\|\hat{m} - \tilde{m}\|_{L^2(g)}^2 &\leq 2 \|\tilde{\psi}'_J\left\{(\tilde{\Psi}'_{J,n}\tilde{\Psi}_{J,n}/n + \varsigma_{J,n} \check{\Psi}_J^{-1})^{-1} - (\tilde{\Psi}'_{J,n}\tilde{\Psi}_{J,n}/n)^-\right\} \tilde{\Psi}'_{J,n}M_0/n\|_{L^2(g)}^2\\ 
&\quad + 2\|\tilde{\psi}'_J(\tilde{\Psi}'_{J,n}\tilde{\Psi}_{J,n}/n + \varsigma_{J,n} \check{\Psi}_J^{-1})^{-1} \tilde{\Psi}'_{J,n}v_n/n\|_{L^2(g)}^2\\
&= 2 \|\left\{(\tilde{\Psi}'_{J,n}\tilde{\Psi}_{J,n}/n + \varsigma_{J,n} \check{\Psi}_J^{-1})^{-1} - (\tilde{\Psi}'_{J,n}\tilde{\Psi}_{J,n}/n)^-\right\} \tilde{\Psi}'_{J,n}M_0/n\|_{L^2(g)}^2\\ 
&\quad + 2\|(\tilde{\Psi}'_{J,n}\tilde{\Psi}_{J,n}/n + \varsigma_{J,n} \check{\Psi}_J^{-1})^{-1} \tilde{\Psi}'_{J,n}v_n/n\|_{L^2(g)}^2\\
&\leq 2\|(\tilde{\Psi}'_{J,n}\tilde{\Psi}_{J,n}/n + \varsigma_{J,n} \check{\Psi}_J^{-1})^{-1} - (\tilde{\Psi}'_{J,n}\tilde{\Psi}_{J,n}/n)^-\|^2 \|\tilde{\Psi}'_{J,n}M_0/n\|^2\\
&\quad + 2\|(\tilde{\Psi}'_{J,n}\tilde{\Psi}_{J,n}/n + \varsigma_{J,n} \check{\Psi}_J^{-1})^{-1}\|^2 \| \tilde{\Psi}'_{J,n}v_n/n\|^2.
\end{align*}
By similar arguments in the proof of Proposition \ref{prp:variance}, we have 
\[
\|(\tilde{\Psi}'_{J,n}\tilde{\Psi}_{J,n}/n + \varsigma_{J,n} \check{\Psi}_J^{-1})^{-1} - (\tilde{\Psi}'_{J,n}\tilde{\Psi}_{J,n}/n)^-\| = O_p(\varsigma_{J,n} \lambda_J^{-2}),
\]
$\| \tilde{\Psi}'_{J,n}M_0/n\| = O_p(\zeta_J\lambda_J)$, and $\| \tilde{\Psi}'_{J,n}v_n/n\| = O_p(\zeta_J\lambda_J/\sqrt{A_n})$. Then we have $\|\hat{m} - \tilde{m}\|_{L^2(g)} = O_p(\zeta_J\lambda_J/ \sqrt{A_n} + \varsigma_{J,n}\zeta_J\lambda_J^{-1})$. Further, we can show that $\|\tilde{m} - m_0\|_{L^2(g)} = O_p(\|m_0 - m_{0,J}\|_{L^2(g)})$ by the same argument in the proof of Lemma 2.5 in \cite{ChCh15}. 

\subsection{Proof of Corollary \ref{cor:L^2-rate1}}

Let $P_J$ denote the $L^2(g)$ orthogonal projection operator onto $\Psi_J$ given by
\[
P_Jm(\bm{z}) = \psi_J(\bm{z})'\E\left[\psi_J\left({\bm{S}_1 \over A_n}\right)\psi_J\left({\bm{S}_1 \over A_n}\right)'\right]^{-1}\E\left[\psi_J\left({\bm{S}_1 \over A_n}\right)m\left({\bm{S}_1 \over A_n}\right)\right].
\]
Observe that for any $\tilde{m} \in \Psi_J$, 
\begin{align*}
\|m_0 - m_{J,0}\|_{L^2(g)} &=  \|m_0 - \tilde{m} + \tilde{m} - m_{J,0}\|_{L^2(g)}\\
&\leq \|m_0 - \tilde{m}\|_{L^2(g)} + \|\tilde{m} - m_{J,0}\|_{L^2(g)}\\
&\leq \|m_0 - \tilde{m}\|_{L^2(g)} + \|P_J(\tilde{m} - m_{0})\|_{L^2(g)}\\
&\leq 2\|m_0 - \tilde{m}\|_{L^2(g)}\\
&\lesssim \|m_0 - \tilde{m}\|_\infty.
\end{align*}
For the second inequality, we used $P_Jm_0 = m_{0,J}$. For the third inequality, we used the fact that $P_J$ is an orthogonal projection on $L^2(g)$. For the last wave relation, we used Assumption \ref{Ass:sieve2} (i). The above inequality holds uniformly in $\tilde{m} \in \Psi_J$. Then by taking the $\inf$ over $\tilde{m} \in \Psi_J$, we obtain the desired result. 

\section*{Supplementary Material}
The supplementary material includes the introduction of the spatial regression model with the uniform and $L^2$ convergence rates, a multivariate CLT, and an estimator of the asymptotic variance for series ridge estimators (Section \ref{Sec:reg-covariate}), proofs for Section \ref{Sec:reg-trend} (Section \ref{Appendix:sec2}), proofs for Section \ref{Appendix:example} (Section \ref{Appendix:ex-proof}), proofs for Section \ref{Sec:reg-covariate} (Section \ref{Appendix:sec3}), and auxiliary lemmas (Section \ref{Appendix:lemmas}).

\section{Series regression estimators}\label{Sec:reg-covariate}

In this section, we establish the asymptotic properties of series ridge estimators of a spatial regression model. For this, we first introduce our model and series estimators (Sections \ref{Subsec:model2} and \ref{Subsec:SR-unif}). Then we provide uniform and $L^2$ convergence rates and a multivariate CLT of general series estimators (Sections \ref{Subsec:SR-unif}, \ref{Subsec:SR-L2}, and \ref{Subsec:SR-AN}), and establish that spline and wavelet estimators attain the optimal uniform and $L^2$ convergence rates (Sections \ref{Subsec:SR-opt-unif} and \ref{Subsec:SR-L2}).

\subsection{Spatial regression model}\label{Subsec:model2}
Consider the nonparametric regression model
\begin{align}\label{eq:Spatial-Reg}
Y(\bm{S}_i) &= \mathfrak{m}_0\left({\bm{S}_i \over A_n},\bm{X}(\bm{S}_i)\right) + \mathfrak{h}\left({\bm{S}_i \over A_n}, \bm{X}(\bm{S}_i) \right)\varepsilon_i \nonumber \\
&=: \mathfrak{m}_0\left({\bm{S}_i \over A_n},\bm{X}(\bm{S}_i)\right) + \nu_{n,i},\ i=1,\dots,n,
\end{align}
where $\{\bm{S}_i\} \subset R_n$ is a sequence of (stochastic) sampling sites with the density $A_n^{-1}g(\cdot/A_n)$, $\bm{X}:= \{\bm{X}(\bm{s}) = (X_1(\bm{s}),\dots, X_p(\bm{s}))':\bm{s} \in \mathbb{R}^d\}$ is a strictly stationary random field defined on $\mathbb{R}^d$, $\mathfrak{m}_0: [-1/2,1/2]^d \times \mathcal{X} \to \mathbb{R}$ is the spatial regression function with $\mathcal{X} \subset \mathbb{R}^p$, $\{\varepsilon_{i}\}$ is a sequence of i.i.d. random variables such that $\E[\varepsilon_{i}] = 0$ and $\E[\varepsilon_{i}^2] = 1$, $\mathfrak{h}:[-1/2,1/2]^d \times \mathcal{X} \to (0,\infty)$ is the variance function $\{\varepsilon_{i}\}$. In what follows, we assume that $\bm{X}=\{\bm{X}(\bm{s}):\bm{s} \in \mathbb{R}^d\}$, $\{\bm{S}_i\}_{i=1}^{n}$, and $\{\varepsilon_i\}_{i=1}^{n}$ are mutually independent. The model (\ref{eq:Spatial-Reg}) 
can be seen as a nonlinear extension of geographically weighted regression models (cf. \cite{LuChHaSt14} and \cite{GoLuChBrHa15}). In our spatial regression model, we assume that the spatial correlation of the response variable $Y(\bm{S}_i)$ is explained by the covariate process $\bm{X}$. Therefore, $\nu_{n,i}$ can be interpreted as an exogenous variable that represents the observation error at each location or does not affect the conditional expectation of the response variable given the covariate vector $\bm{X}(\bm{S}_i)$.

For the model (\ref{eq:Spatial-Reg}), we assume the following conditions. 

\begin{assumption}\label{Ass:model2}
For $j=1,\dots,d$, let $\{A_{n1,j}\}_{n \geq 1}$ and $\{A_{n2,j}\}_{n \geq 1}$ be sequences of positive numbers such that $\min\{A_{n2,j}, {A_{n1,j} \over A_{n2,j}}\} \to \infty$ and $A_n(A_n^{(1)}\log n)^{-1} \to \infty$ as $n \to \infty$ where $A_n^{(1)} = \prod_{j=1}^{d}A_{n1,j}$. 
\begin{itemize}
\item[(i)] The random field $\bm{X}=\{\bm{X}(\bm{s}) = (X_1(\bm{s}),\dots,X_p(\bm{s}))':\bm{s} \in \mathbb{R}^d\}$ is strictly stationary and $\mathcal{X}$ is convex and has nonempty interior. 
\item[(ii)] The random field $\bm{X}$ is $\beta$-mixing with mixing coefficients $\beta(a;b)$ such that as $n \to \infty$,  $A_n(A_n^{(1)})^{-1}\beta(\underbar{A}_{n2};A_n) \to 0$ where $\underbar{A}_{n2} = \min_{1 \leq j \leq d}A_{n2,j}$.
\item[(iii)] $\{\varepsilon_i\}_{i=1}^n$ is a sequence of i.i.d. random variables such that $\E[\varepsilon_1]=0$, $\E[\varepsilon_1^2]=1$, and $\E[|\varepsilon_1|^q]<\infty$ for some $q>2$.
\item[(iv)] The function $\mathfrak{h}$ is continuous on $R_0 \times \mathcal{X}$ and $\sup_{(\bm{z}, \bm{x}) \in R_0 \times \mathcal{X}}|\mathfrak{h}(\bm{z},\bm{x})|<\infty$. 
\end{itemize}
\end{assumption}

By a similar argument in the proof of Proposition \ref{prp:LevyMA-ex}, we can confirm that Assumption \ref{Ass:model2} is satisfied for a certain class of Lévy-driven moving average random fields.

To allow possibly unbounded support $\mathcal{X}$, we construct a series estimator of $\mathfrak{m}_0$ over $R_0 \times D_n$ that satisfies the following conditions. 
\begin{assumption}\label{Ass:domain}
Let $D_n$ be compact, convex, and have a nonempty interior. 
\begin{itemize}
\item[(i)] $D_n \subset D_{n+1} \subset \mathcal{X}$ for all $n$. 
\item[(ii)] There exist $v_1,v_2>0$ such that $N(D_n,\|\cdot\|, \delta)\lesssim n^{v_1}\delta^{-v_2}$.  
\end{itemize}
\end{assumption}

If $\mathcal{X} = \mathbb{R}^p$, one can take $D_n = \{\bm{x} \in \mathbb{R}^p: \|\bm{x}\|\leq r_n\}$ or $[-r_n,r_n]^p$ where $\{r_n\}$ is a sequence of constants such that $r_n \to \infty$ as $n \to \infty$. We can cover these $D_n$ with $(2r_n/\delta)^p$ $\ell^\infty$-balls of radius $\delta$, each of which is contained in an Euclidean ball of radius $\delta\sqrt{p}$. Then we have $N(D_n,\|\cdot\|,\delta) \leq (2\sqrt{p}r_n/\delta)^p \lesssim n^{vp}\delta^{-p}$ for some $v>0$.

\subsection{Uniform convergence rates}\label{Subsec:SR-unif}
We estimate the regression function $\mathfrak{m}_0$ by the following series ridge estimator.
\begin{align*}
\hat{\mathfrak{m}}(\bm{z},\bm{x}) &= \hat{\mathfrak{m}}_n(\bm{z},\bm{x}) := b_J^{(w)}(\bm{z},\bm{x})'\left({B_{J,n}^{(w)'}B_{J,n}^{(w)} \over n} + \varsigma_{J,n} I_J\right)^{-1}{B_{J,n}^{(w)'} \bm{Y} \over n},\\
(\bm{z}', \bm{x}')' &= (z_1,\dots,z_d,x_1,\dots,x_p)' \in R_0 \times D_n,
\end{align*}
where $\varsigma_{J,n}$ is a sequence of positive constants with $\varsigma_{J,n} \to 0$ as $n, J \to \infty$, $\bm{Y} = (Y(\bm{S}_1),\dots,Y(\bm{S}_n))'$, $b_{J,1},\dots, b_{J,J}$ are a collection of $J$ sieve basis functions, 
\begin{align*}
b_J^{(w)}(\bm{z},\bm{x}) &= (b_{J,1}(\bm{z},\bm{x})w_n(\bm{x}),\dots,b_{J,J}(\bm{z},\bm{x})w_n(\bm{x}))',\ 
w_n(\bm{x}) = 
\begin{cases}
1 & \text{if $\bm{x} \in D_n$},\\
0 & \text{otherwise},
\end{cases}\\ 
B_{J,n}^{(w)} &= \left(b_J^{(w)}\left({\bm{S}_1 \over A_n},\bm{X}(\bm{S}_1)\right),\dots,b_J^{(w)}\left({\bm{S}_n \over A_n},\bm{X}(\bm{S}_n)\right)\right)'. 
\end{align*}
In our study, we use the above weight function $w_n(\bm{x})$, but other candidates for the weight function include $w_n(\bm{x})=(1+\|\bm{x}\|^2)^{-w}$ and $w_n(\bm{x})=\exp(-\|\bm{x}\|^w)$ with some $w \geq 0$. In fact, in \cite{LeRo16} consider a (non-penalized) series estimator with $w_n(\bm{x})=(1+\|\bm{x}\|^2)^{-w}$. We leave the extension of our theoretical results when using other weight functions for future research.

In Section \ref{Subsec:SR-opt-unif}, we focus on the case where the basis functions are supported on $D_n=D_0=[-1/2,1/2]^p$. In this case, we can construct tensor product B-spline or wavelet bases supported on $R_0 \times D_0$ in the same manner as in Section \ref{Subsec:STR-opt-unif-rate}. For example, the spline basis functions $\psi_{J_{0,k},j}(v)$ on $v \in [-1/2,1/2]$, $j=1,\dots,J_{0,k}$, $k=1,\dots,d+p$ are extended to the basis functions on $(\bm{z},\bm{x})=(z_1,\dots,z_d,x_1,\dots,x_p)' \in [-1/2,1/2]^{d+p}$ by $\psi_{J_{0,1},j_1}(z_1)\times \dots \times \psi_{J_{0,d},j_d}(z_d) \times \psi_{J_{0,d+1},j_{d+1}}(x_1) \times \dots \times \psi_{J_{0,d+p},j_{d+p}}(x_p)$,
$j_k=1,\dots,J_{0,k}$, $k=1,\dots,d+p$. In the case where $D_n = [-r_n,r_n]^p$, we can construct basis functions supported on $R_0 \times D_n$ by rescaling the components corresponding to $D_0$ in basis functions supported on $R_0 \times D_0$.

Define 
\[
\zeta_{J,n} = \sup_{(\bm{z},\bm{x}) \in R_0 \times \mathbb{R}^p}\|b_J^{(w)}(\bm{z},\bm{x})\|,\ \lambda_{J,n} = \lambda_{\text{min}}\left(\E\left[b_J^{(w)}\left({\bm{S}_1 \over A_n},\bm{X}(\bm{S}_1)\right)b_J^{(w)}\left({\bm{S}_1 \over A_n},\bm{X}(\bm{S}_1)\right)'\right]\right)^{-1/2}.
\]
We assume the following conditions on the sieve basis functions. 

\begin{assumption}\label{Ass:sieve3}
Let $\nabla b_J^{(w)}(\bm{z},\bm{x}) = \left({\partial b_J^{(w)}(\bm{z},\bm{x}) \over \partial \bm{z}}, {\partial b_J^{(w)}(\bm{z},\bm{x}) \over \partial \bm{x}}\right) \in \mathbb{R}^{J \times (d+p)}$ where
\begin{align*}
{\partial b_J^{(w)}(\bm{z},\bm{x}) \over \partial \bm{z}} &\!=\! \left(\!{\partial b_{J,j}(\bm{z},\bm{x})w_n(\bm{x}) \over \partial z_k}\!\right)_{1 \leq j \leq J, 1 \leq k \leq d}, {\partial b_J^{(w)}(\bm{z},\bm{x}) \over \partial \bm{x}} \!=\! \left(\!{\partial b_{J,j}(\bm{z},\bm{x})w_n(\bm{x}) \over \partial x_k}\!\right)_{1 \leq j \leq J, 1 \leq k \leq p}. 
\end{align*}
\begin{itemize}
\item[(i)] There exist $\omega_1,\omega_2 \geq 0$ such that $\sup_{(\bm{z},\bm{x}) \in R_0 \times D_n}\|\nabla b_J^{(w)}(\bm{z},\bm{x})\| \lesssim n^{\omega_1}J^{\omega_2}$.
\item[(ii)] There exist $\varpi_1 \geq 0$, $\varpi_2>0$ such that $\zeta_{J,n} \lesssim n^{\varpi_1}J^{\varpi_2}$.
\item[(iii)] $\lambda_{\rm{min}}(\E[b_J^{(w)}(\bm{S}_1/A_n,\bm{X}(\bm{S}_1))b_J^{(w)}(\bm{S}_1/A_n,\bm{X}(\bm{S}_1))'])>0$ for each $J$ and $n$. 
\end{itemize}
\end{assumption}

Assumption \ref{Ass:sieve3} is satisfied with $\lambda_{J,n} \!\lesssim \! 1$ and $\zeta_{J,n} \!\lesssim \! \sqrt{J}$ for tensor-products of univariate polynomial spline, trigonometric polynomial or wavelet bases. See also the comments on Assumption \ref{Ass:sieve1}.

Let $\tilde{b}_J^{(w)}(\bm{z},\bm{x})$ denote the orthonormalized vector of basis functions, that is, 
\[
\tilde{b}_J^{(w)}(\bm{z},\bm{x}) = \E\left[b_J^{(w)}\left({\bm{S}_1 \over A_n},\bm{X}(\bm{S}_1)\right)b_J^{(w)}\left({\bm{S}_1 \over A_n},\bm{X}(\bm{S}_1)\right)'\right]^{-1/2}b_J^{(w)}(\bm{z},\bm{x}),
\]
and $\tilde{B}_{J,n}^{(w)} = (\tilde{b}_J^{(w)}(\bm{S}_1/A_n,\bm{X}(\bm{S}_1)),\dots,\tilde{b}_J(\bm{S}_n/A_n,\bm{X}(\bm{S}_n)))'$.

Define $B_{J}^{(w)}$ as the closed linear span of $\{b_{J,1}w_n,\dots, b_{J,J}w_n\}$ in $L^2(R_0\times D_n)$. Let $P_{J,n}^{(w)}$ be the empirical projection operator onto $B_J^{(w)}$, that is, 
\begin{align*}
P_{J,n}^{(w)}\mathfrak{m}(\bm{z},\bm{x}) &= b_J^{(w)}(\bm{z},\bm{x})'\left({B_{J,n}^{(w)'}B_{J,n}^{(w)} \over n}\right)^- {1 \over n}\sum_{i=1}^n b_J^{(w)}\left({\bm{S}_i \over A_n},\bm{X}(\bm{S}_i)\right)\mathfrak{m}\left({\bm{S}_i \over A_n},\bm{X}(\bm{S}_i)\right)\\
&= \tilde{b}_J^{(w)}(\bm{z},\bm{x})(\tilde{B}_{J,n}^{(w)'}\tilde{B}_{J,n}^{(w)})^- \tilde{B}_{J,n}^{(w)'}\mathfrak{M}
\end{align*}
where $\mathfrak{M} = (\mathfrak{m}(\bm{S}_1/A_n,\bm{X}(\bm{S}_1)),\dots,\mathfrak{m}(\bm{S}_n/A_n,\bm{X}(\bm{S}_n)))'$. The operator $P_{J,n}^{(w)}$ is well defined:

\noindent
if $L_{w,n}^2(R_0 \times \mathbb{R}^p)$ denotes the space of functions with norm $\|\cdot\|_{L^2,w,n}$ where 
\[
\|f\|_{L^2,w,n}^2 = {1 \over n}\sum_{i=1}^{n}f(\bm{S}_i/A_n,\bm{X}(\bm{S}_i))^2w_n(\bm{X}(\bm{S}_i)),
\] 
then $P_{J,n}^{(w)}: L_{w,n}^2(R_0 \times \mathbb{R}^p) \to L_{w,n}^2(R_0 \times \mathbb{R}^p)$ is an orthogonal projection onto $B_J^{(w)}$ whenever $B_{J,n}^{(w)'}B_{J,n}^{(w)}$ is invertible. 

Let $\tilde{\mathfrak{m}}$ denote the projection of $\mathfrak{m}_0$ onto $B_{J}^{(w)}$ under the empirical measure, that is, 
\begin{align*}
\tilde{\mathfrak{m}}(\bm{z},\bm{x}) &= P_{J,n}^{(w)}\mathfrak{m}_0(\bm{z},\bm{x}) = \tilde{b}_J^{(w)}(\bm{z},\bm{x})'(\tilde{B}_{J,n}^{(w)'}\tilde{B}_{J,n}^{(w)})^- \tilde{B}_{J,n}^{(w)'} \mathfrak{M}_0
\end{align*}
where $\mathfrak{M}_0 = \mathfrak{M}_{0,n} = (\mathfrak{m}_0(\bm{S}_1/A_n,\bm{X}(\bm{S}_1)),\dots,\mathfrak{m}_0(\bm{S}_n/A_n,\bm{X}(\bm{S}_n))'$. 

The following result provides uniform convergence rates of the variance term of general series estimators. 
\begin{proposition}\label{prp:variance2}
Suppose that Assumptions \ref{Ass:sample}(i), (ii), \ref{Ass:model2}, \ref{Ass:domain}, and \ref{Ass:sieve3} hold. Assume that $\bm{X}=\{\bm{X}(\bm{s}):\bm{s} \in \mathbb{R}^d\}$, $\{\bm{S}_i\}_{i=1}^{n}$, and $\{\varepsilon_i\}_{i=1}^{n}$ are mutually independent, $\max\{(\zeta_{J,n}\lambda_{J,n})^{{q \over q-2}}, \zeta_{J,n}^2(\zeta_{J,n}^2+\zeta_{J,n}\lambda_{J,n}^{-1}+\lambda_{J,n}^2) \}\lesssim \sqrt{A_n/(\log n)^2}$, and $\varsigma_{J,n} \lesssim \zeta_{J,n}^2\sqrt{\log n/A_n}$ as $n, J \to \infty$. Then
\begin{align}\label{eq:var-unif-rate2}
\|\hat{\mathfrak{m}} - \tilde{\mathfrak{m}}\|_\infty &= O_p\left(\zeta_{J,n}\lambda_{J,n}\sqrt{{\log n \over n}} + \varsigma_{J,n}\zeta_{J,n}^2\|\mathfrak{m}_0\|_{w,\infty}\right)\ \text{as $n,J \to \infty$},
\end{align}
where $\|\cdot\|_{w,\infty}$ denote the weighted sup norm, that is, $\|f\|_{w,\infty} = \sup_{(\bm{z},\bm{x}) \in R_0 \times \mathbb{R}^p}|f(\bm{z},\bm{x})w_n(\bm{x})| = \sup_{(\bm{z},\bm{x}) \in R_0 \times D_n}|f(\bm{z},\bm{x})|$.
\end{proposition}

Note that $\|\hat{\mathfrak{m}} - \tilde{\mathfrak{m}}\|_\infty = \|\hat{\mathfrak{m}} - \tilde{\mathfrak{m}}\|_{w,\infty}$ since both $\hat{\mathfrak{m}}$ and $\tilde{\mathfrak{m}}$ have support $R_0 \times D_n$. The convergence rate of the first term in (\ref{eq:var-unif-rate2}) is the same as the case of i.i.d. data and this is attributed to the fact that the error terms are conditionally independent in the model (\ref{eq:Spatial-Reg}). 

Now we give uniform convergence rates of $\hat{\mathfrak{m}}$. Let $L_{w,n}^\infty(R_0 \times \mathbb{R}^p)$ denote the space of functions of which $\|f\|_{w,\infty}<\infty$ and let
\begin{align*}
\|P_{J,n}^{(w)}\|_{w,\infty} &= \sup_{\mathfrak{m} \in L_{w,n}^\infty(R_0 \times \mathbb{R}^p), \|\mathfrak{m}\|_{w,\infty} \neq 0}{\|P_{J,n}^{(w)}\mathfrak{m}\|_{w,\infty} \over \|\mathfrak{m}\|_{w,\infty}}
\end{align*}
denote the (weighted sup) operator norm of $P_{J,n}^{(w)}$. 

The following result establishes uniform convergence rates of general series estimators. 

\begin{corollary}\label{cor:unif-rate2}
Suppose that assumptions in Proposition \ref{prp:variance2} hold. Then: (1)
\begin{align*}
\|\hat{\mathfrak{m}} - \mathfrak{m}_0\|_{w,\infty} &= O_p\!\left(\!\zeta_{J,n}\lambda_{J,n}\sqrt{{\log n \over n}} + \varsigma_{J,n}\zeta_{J,n}^2\|\mathfrak{m}_0\|_{w,\infty}\!\right) \! + \! (1 \!+\! \|P_{J,n}^{(w)}\|_{w,\infty})\inf_{\mathfrak{m} \in B_J^{(w)}}\!\|\mathfrak{m}_0 - \mathfrak{m}\|_{w,\infty}.
\end{align*}
(2) Further, if the linear sieve satisfies $\zeta_{J,n} \lesssim \sqrt{J}$, $\lambda_{J,n} \lesssim 1$, and $\|P_{J,n}^{(w)}\|_{w,\infty} = O_p(1)$, then
\[
\|\hat{\mathfrak{m}} - \mathfrak{m}_0\|_{w,\infty} = O_p\left(\sqrt{{J\log n \over n}} + \varsigma_{J,n}J\|\mathfrak{m}_0\|_{w,\infty} + \inf_{\mathfrak{m} \in B_J^{(w)}}\|\mathfrak{m}_0 - \mathfrak{m}\|_{w,\infty}\right).
\]
\end{corollary}

Note that the term $(1 \!+\! \|P_{J,n}^{(w)}\|_{w,\infty})\inf_{\mathfrak{m} \in B_J^{(w)}}\!\|\mathfrak{m}_0 - \mathfrak{m}\|_{w,\infty}$ comes from the bound of $\|\tilde{\mathfrak{m}} - \mathfrak{m}_0\|_\infty$, which can be addressed in a similar way to the proof of Lemma 2.4 in \cite{ChCh15}. The main difference from Lemma 2.4 in \cite{ChCh15} arises from the stochastic term associated with the evaluation of $\|\hat{\mathfrak{m}} - \tilde{\mathfrak{m}}_0\|_\infty$. In particular, the term $\varsigma_{J,n}\zeta_{J,n}^2\|\mathfrak{m}_0\|_{w,\infty}$, which is due to the ridge penalty, has emerged as an additional term. The evaluation of this term requires a new inequality concerning the spectral norm of the difference between the inverses of two non-singular matrices (Lemma \ref{lem:matrix-inv-norm}).

\subsection{Optimal uniform rates for spline and wavelet estimators}\label{Subsec:SR-opt-unif}

In this subsection, we establish that the spline and wavelet estimators attain the optimal uniform convergence rate when the regression function belongs to a H\"older space. 

\begin{assumption}\label{Ass:sieve4}
Assume that $\bm{X}(\bm{0})$ has the density function $f_{\bm{X}}$.  
\begin{itemize}
\item[(i)] $D_n = D_0 = [-1/2,1/2]^p \subset \mathcal{X}$ for all $n$. 
\item[(ii)] The density functions $g$ and $f_{\bm{X}}$ are uniformly bounded away from zero and infinity on $R_0$ and $D_0$, respectively. 
\item[(iii)] The restriction of $\mathfrak{m}_0$ to $R_0 \times D_0$ belongs to $\Lambda^r(R_0 \times D_0)$ for some $r>0$. 
\item[(iv)] Let $B_J$ be the closed linear span of $\{b_{J,1},\dots,b_{J,J}\}$. The sieve $B_J$ is $\text{BSpl}(J,R_0 \times D_0,\varrho)$ or $\text{Wav}(J,R_0 \times D_0,\varrho)$ with $\varrho > \max\{r,1\}$. 
\end{itemize}
\end{assumption}

One can verify that Condition (i) implies Assumption \ref{Ass:domain} and Conditions (ii) and (iv) imply Assumption \ref{Ass:sieve3} with $\zeta_{J,n} \lesssim \sqrt{J}$ and $\lambda_{J,n} \sim 1$. Note that given the spline or wavelet basis functions on $[-1/2,1/2]$, we can construct basis functions on $R_0$ by taking the product of the elements of each of the univariate bases. See also the discussion before Assumption \ref{Ass:sieve2} for the construction of multivariate basis functions.

The following result establishes that the spline and wavelet estimators achieve the optimal uniform convergence rates of \cite{St82}. 

\begin{theorem}\label{thm:unif-rate2}
Suppose that Assumptions \ref{Ass:sample}(i), (ii), \ref{Ass:model2}, \ref{Ass:domain}, and \ref{Ass:sieve4} hold. Assume that $\bm{X}=\{\bm{X}(\bm{s}):\bm{s} \in \mathbb{R}^d\}$, $\{\bm{S}_i\}_{i=1}^{n}$, and $\{\varepsilon_i\}_{i=1}^{n}$ are mutually independent, $2r \geq 3d+p$, $\max\{J^{{q \over 2(q-2)}}, J^2 \}\lesssim \sqrt{A_n/(\log n)^2}$, and $\varsigma_{J,n}(\|\mathfrak{m}_0\|_{w,\infty}\sqrt{J^2n/(\log n)} + \sqrt{A_n/(J^2\log n)})\lesssim 1$ as $n, J \to \infty$.\\ 
If $J \sim \left(n/\log n\right)^{{d+p \over 2r+d+p}}$, then
\[
\|\hat{\mathfrak{m}} - \mathfrak{m}_0\|_{w,\infty} = O_p\left(\left({\log n \over n}\right)^{r \over 2r+d+p}\right).
\]
\end{theorem}

\subsection{$L^2$ convergence rates}\label{Subsec:SR-L2}

Let $L^p(\bm{S},\bm{X})$ denote the function space consisting of all (equivalence class) of measurable functions $f$ for which the $L^p(\bm{S},\bm{X})$ norm $\|f\|_{L^p(\bm{S},\bm{X})} \!=\! \E\left[|f(\bm{S}_1/A_n, \bm{X}(\bm{S}_1))|^p\right]^{1/p}$ is finite. 

The next result provides a sharp upper bound of the $L^2$ convergence rates of general series estimators for the spatial regression model. Note that mixing conditions for the random field $\bm{X}$ are not necessary to establish the result. 

\begin{proposition}\label{prp:L^2-rate2} Suppose that Assumptions \ref{Ass:sample}(i), (ii), \ref{Ass:model2}, and \ref{Ass:sieve3}(iii) hold. Additionally, assume that $\bm{X}=\{\bm{X}(\bm{s}):\bm{s} \in \mathbb{R}^d\}$, $\{\bm{S}_i\}_{i=1}^{n}$, and $\{\varepsilon_i\}_{i=1}^{n}$ are mutually independent, $\varsigma_{J,n}\lesssim \zeta_{J,n}^2\sqrt{\log n/n}$, and $\zeta_{J,n}^3(\zeta_{J,n} + \lambda_{J,n}^{-1}) \lesssim \sqrt{A_n/(\log n)^2}$. Then
\begin{align*}
\|\hat{\mathfrak{m}} - \tilde{\mathfrak{m}}\|_{L^2(\bm{S},\bm{X})} &= O_p\left({\zeta_{J,n}\lambda_{J,n} \over \sqrt{n}} + \varsigma_{J,n} \zeta_{J,n}\lambda_{J,n}^{-1}\|\mathfrak{m}_0\|_{L^2(\bm{S},\bm{X})}\right),\\ \|\tilde{\mathfrak{m}} - \mathfrak{m}_0\|_{L^2(\bm{S},\bm{X})} &= O_p\left(\|\mathfrak{m}_0 - \mathfrak{m}_{0,J}\|_{L^2(\bm{S},\bm{X})}\right),
\end{align*}
where $\mathfrak{m}_{0,J}$ is the $L^2(\bm{S},\bm{X})$ orthogonal projection of $\mathfrak{m}_0$ onto $B_J^{(w)}$. 
\end{proposition}

The following result establishes that the spline and wavelet estimators attain the optimal $L^2$ convergence rates of \cite{St82}.

\begin{corollary}\label{cor:L^2-rate2}
Suppose that Assumptions \ref{Ass:sample}(i), (ii), \ref{Ass:model2}, and \ref{Ass:sieve4} hold. Additionally, assume that $\bm{X}=\{\bm{X}(\bm{s}):\bm{s} \in \mathbb{R}^d\}$, $\{\bm{S}_i\}_{i=1}^{n}$, and $\{\varepsilon_i\}_{i=1}^{n}$ are mutually independent, $2r \geq 3d+p$, $\|\mathfrak{m}_0\|_{L^2(\bm{S},\bm{X})}<\infty$, $\varsigma_{J,n}(\sqrt{n}\lambda_{J,n}^{-2} + \sqrt{n/(J^2\log n)})\lesssim 1$, and $J^2 \lesssim \sqrt{A_n/(\log n)^2}$. If $J \sim n^{{d+p \over 2r+d+p}}$, then
\[
\|\hat{\mathfrak{m}} - \mathfrak{m}_0\|_{L^2(\bm{S},\bm{X})} = O_p\left(n^{-{r \over 2r+d+p}}\right).
\] 
\end{corollary}

\subsection{Asymptotic normality}\label{Subsec:SR-AN}

The following result establishes the asymptotic normality of the series estimator with an arbitrary basis. 

\begin{theorem}\label{thm:asy-norm2}
Let $\{(\bm{z}_\ell, \bm{x}_\ell)\}_{1 \leq \ell \leq L}$ be a set of points such that $\bm{z}_\ell \in R_0$ and $\bm{x}_\ell$ are any interior points of $\mathcal{X}$.
Suppose that Assumptions \ref{Ass:sample}(i), (ii), \ref{Ass:model2}, and \ref{Ass:sieve3}(iii) hold. Additionally, assume that $\bm{X}=\{\bm{X}(\bm{s}):\bm{s} \in \mathbb{R}^d\}$, $\{\bm{S}_i\}_{i=1}^{n}$, and $\{\varepsilon_i\}_{i=1}^{n}$ are mutually independent, $\inf_{(\bm{z},\bm{x}) \in R_0 \times \mathcal{X}}\mathfrak{h}(\bm{z},\bm{x})>0$, and
\begin{itemize}
\item[(a)] $\max\{(\zeta_{J,n}\lambda_{J,n})^{{2q \over q-2}}, (\zeta_{J,n}\lambda_{J,n})^4, \zeta_{J,n}^3\lambda_{J,n}^{-1}\} \lesssim \sqrt{A_n/(\log n)^2}$, 
\item[(b)] $\|P_{J,n}^{(w)}\|_{w,\infty} = O_p(1)$, $\sqrt{n}(\min_{1 \leq \ell \leq L}\|\tilde{b}_J^{(w)}(\bm{z}_\ell,\bm{x}_\ell)\|)^{-1}\inf_{\mathfrak{m} \in B_J^{(w)}}\|\mathfrak{m}_0 - \mathfrak{m}\|_\infty\to 0$,
\item[(c)] $\varsigma_{J,n} \lesssim \zeta_{J,n}^2\sqrt{\log n/A_n}$, $\varsigma_{J,n}\zeta_{J,n}^2\sqrt{n}\|\mathfrak{m}_0\|_{w,\infty}(\min_{1 \leq \ell \leq L}\|\tilde{b}_J^{(w)}(\bm{z}_\ell,\bm{x}_\ell)\|)^{-1} \to 0$
\end{itemize}
as $n,J \to \infty$. Then we have
\[
\sqrt{n}V_J^{-{1 \over 2}}\left({\hat{\mathfrak{m}}(\bm{z}_1,\bm{x}_1) - \mathfrak{m}_0(\bm{z}_1,\bm{x}_1) \over \sqrt{V_J(\bm{z}_1,\bm{x}_1)}},\dots, {\hat{\mathfrak{m}}(\bm{z}_L,\bm{x}_L) - \mathfrak{m}_0(\bm{z}_L,\bm{x}_L) \over \sqrt{V_J(\bm{z}_L,\bm{x}_L)}}\right) \stackrel{d}{\to} N(\bm{0},I_L)\ \text{as $n,J \to \infty$},
\] 
where 
\begin{align*}
V_J &= (V_J^{(\ell_1,\ell_2)})_{1 \leq \ell_1,\ell_2 \leq L},\ V_J^{(\ell_1,\ell_2)} = {V_J(\bm{z}_{\ell_1},\bm{x}_{\ell_1},\bm{z}_{\ell_2},\bm{x}_{\ell_2}) \over \sqrt{V_J(\bm{z}_{\ell_1},\bm{x}_{\ell_1})}\sqrt{V_J(\bm{z}_{\ell_2},\bm{x}_{\ell_2})}},\\
V_J(\bm{z},\bm{x}) &= \tilde{b}_J^{(w)}(\bm{z},\bm{x})'H_J\tilde{b}_J^{(w)}(\bm{z},\bm{x}),\ V_J(\bm{z}_{\ell_1},\bm{x}_{\ell_1},\bm{z}_{\ell_2},\bm{x}_{\ell_2}) = \tilde{b}_J^{(w)}(\bm{z}_{\ell_1},\bm{x}_{\ell_1})'H_J\tilde{b}_J^{(w)}(\bm{z}_{\ell_2},\bm{x}_{\ell_2}),\\
H_J &= \E\left[\mathfrak{h}\left({\bm{S}_1 \over A_n},\bm{X}(\bm{S}_1)\right)^2  \tilde{b}_J^{(w)}\left({\bm{S}_1 \over A_n},\bm{X}(\bm{S}_1)\right)\tilde{b}_J^{(w)}\left({\bm{S}_1 \over A_n},\bm{X}(\bm{S}_1)\right)'\right].
\end{align*}
\end{theorem}

Conditions (a) and (c) are concerned with replacing $(\tilde{B}^{(w)'}_{J,n}\tilde{B}^{(w)}_{J,n}/n + \varsigma_{J,n}(\check{B}_{J}^{(w)})^{-1})^{-1}$ in the definition of $\hat{\mathfrak{m}}$ with $I_J$ where $\check{B}_{J}^{(w)} = \E[b_J^{(w)}(\bm{S}_1/A_n, \bm{X}(\bm{S}_1))b_J^{(w)}(\bm{S}_1/A_n, \bm{X}(\bm{S}_1))']$. Condition (b) is required to show that the bias term of the series estimator is asymptotically negligible.

Now we provide a consistent estimator of $V_J(\bm{z}_1, \bm{x}_1, \bm{z}_2, \bm{x}_2)$ for $(\bm{z}_1, \bm{x}_1), (\bm{z}_2,\bm{x}_2) \in R_0 \times D_0$, which enables us to construct confidence intervals for $\mathfrak{m}_0(\bm{z},\bm{x})$. As with the spatial trend regression model, it should be noted that in the spatial regression model as well, the series estimator is not affected by boundary estimation of $R_0$ on the support of the regression function, unlike the kernel estimator. For a more detailed discussion, refer to the discussion following Theorem \ref{thm:asy-norm1}.

Define $\hat{V}_J(\bm{z}_1, \bm{x}_1, \bm{z}_2, \bm{x}_2) := b_J^{(w)}(\bm{z}_1,\bm{x}_1)'\hat{H}_{J}b_J^{(w)}(\bm{z}_2,\bm{x}_2)$ where 
\begin{align*}
\hat{H}_J &= {1 \over n}\sum_{i=1}^n \left({B_{J,n}^{(w)'}B_{J,n}^{(w)} \over n} + \varsigma_{J,n} I_J\right)^{-1}\!\!\!\!\!b_J^{(w)}\!\!\left({\bm{S}_i \over A_n},\bm{X}(\bm{S}_i)\right)\!b_J^{(w)}\!\!\left({\bm{S}_i \over A_n},\bm{X}(\bm{S}_i)\right)'\!\!\left({B_{J,n}^{(w)'}B_{J,n}^{(w)} \over n} + \varsigma_{J,n} I_J\right)^{-1}\\
&\quad \quad  \times \left(Y(\bm{S}_i) - \hat{\mathfrak{m}}\left({\bm{S}_i \over A_n},\bm{X}(\bm{S}_i)\right)\!\right)^2.
\end{align*}

\begin{proposition}\label{prp:var-est-reg}
Let $2r \geq 3d + p$. Assume that $\bm{X}=\{\bm{X}(\bm{s}):\bm{s} \in \mathbb{R}^d\}$, $\{\bm{S}_i\}_{i=1}^{n}$, and $\{\varepsilon_i\}_{i=1}^{n}$ are mutually independent and $\varsigma_{J,n} (\|\mathfrak{m}_0\|_{w,\infty}\sqrt{J^2n/(\log n)} +\sqrt{A_n/(J^2\log n)})\lesssim 1$ as $n,J \to \infty$. Suppose that 
Assumptions \ref{Ass:sample} (i), (ii), \ref{Ass:model2}, \ref{Ass:sieve3} (iii), and \ref{Ass:sieve4} hold. Additionally, suppose that Conditions (a)-(c) in Theorem \ref{thm:asy-norm2} hold and $\inf_{(\bm{z},\bm{x}) \in R_0 \times D_0}\mathfrak{h}(\bm{z},\bm{x})>0$. Then, for $(\bm{z}_j,\bm{x}_j), j=1,2 \in R_0 \times D_0$,  $\|\tilde{b}_J^{(w)}(\bm{z}_1,\bm{x}_1)\|^{-1}(\hat{V}_J(\bm{z}_1,\bm{x}_1,\bm{z}_2,\bm{x}_2) - V_J(\bm{z}_1,\bm{x}_1,\bm{z}_2,\bm{x}_2))\|\tilde{b}_J^{(w)}(\bm{z}_2,\bm{x}_2)\|^{-1} \stackrel{p}{\to} 0$ as $n, J \to \infty$. 
\end{proposition}

Theorem \ref{thm:asy-norm2} and Proposition \ref{prp:var-est-reg} enable us to construct confidence intervals of $\mathfrak{m}_0(\bm{z},\bm{x})$. Define $\hat{V}_J(\bm{z},\bm{x}) := \hat{V}_J(\bm{z}, \bm{x}, \bm{z}, \bm{x})$ and consider a confidence interval of the form
\begin{align*}
\hat{C}_{1-\tau}(\bm{z},\bm{x}) = \left[\hat{\mathfrak{m}}(\bm{z},\bm{x}) -  \sqrt{{\hat{V}_J(\bm{z},\bm{x}) \over n}}q_{1-\tau/2}, \hat{\mathfrak{m}}(\bm{z},\bm{x}) +  \sqrt{{\hat{V}_J(\bm{z},\bm{x}) \over n}}q_{1-\tau/2}\right]
\end{align*}
for $\tau \in (0,1)$, $(\bm{z},\bm{x}) \in R_0 \times D_0$ where $q_{1-\tau}$ is the $(1-\tau)$-quantile of the standard normal random variable. Then under assumptions in Proposition \ref{prp:var-est-reg}, we have $\Prob(\mathfrak{m}_0(\bm{z},\bm{x}) \in \hat{C}_{1-\tau}(\bm{z},\bm{x})) \to 1-\tau$ as $n,J \to \infty$. 

\section{Proofs for Section \ref{Sec:reg-trend}}\label{Appendix:sec2}

\subsection{The remaining proof of Proposition \ref{prp:variance}}

Now we provide proofs for Steps 3 and 4 of Proposition \ref{prp:variance}.

(Step 3) Now we show that $P_{n,2} = o(1)$ as $n \to \infty$. Define $e_n = (e_{n,1},\dots, e_{n,n})'$ and $\varepsilon_n = (\varepsilon_{n,1},\dots,\varepsilon_{n,n})'$. Observe that 
\begin{align*}
P_{n,2} &\leq \!\Prob\left(\!\max_{\bm{z}_n \in \mathcal{S}_n}\!\!\left|\tilde{\psi}_J(\bm{z}_n)'\tilde{\Psi}'_{J,n}e_n/n \right| \!\geq \! C\zeta_J\lambda_J\!\sqrt{{\log n \over A_n}}\right) \!+ \Prob \left(\!\max_{\bm{z}_n \in \mathcal{S}_n}\!\!\left|\tilde{\psi}_J(\bm{z}_n)'\tilde{\Psi}'_{J,n}\varepsilon_n/n \right| \!\geq \! C\zeta_J\lambda_J\!\sqrt{{\log n \over A_n}}\right)\\
&=: P_{n,21} + P_{n,22}. 
\end{align*}

Now we show that $P_{n,21} = o(1)$ as $n,J \to \infty$. We can show that $P_{n,22} = o(1)$ from almost the same argument in the proof of Lemma 3.2 in \cite{ChCh15}. Let $\tau_n = n^{1/q}(\log n)^\iota$ for some $\iota>0$. Define $e_{n,i}^{(1)} = \eta(\bm{S}_i/A_n)e(\bm{s}_i)1\{|e(\bm{S}_i)| \leq \tau_n\}$, $e_{n,i}^{(2)} = \eta(\bm{S}_i/A_n)e(\bm{s}_i) - e_{n,i}^{(1)}$, $e_n^{(1)} = (e_{n,1}^{(1)},\dots,e_{n,n}^{(1)})'$, $e_n^{(2)} = (e_{n,1}^{(2)},\dots,e_{n,n}^{(2)})'$.
Further, define $\Phi_{n,j}(\bm{z}) = \tilde{\psi}_J(\bm{z})'\tilde{\Psi}'_{J,n}e_n^{(j)}/n - \E[\tilde{\psi}_J(\bm{z})'\tilde{\Psi}'_{J,n}e_n^{(j)}/n]$, $j=1,2$. Then we have
\begin{align*}
P_{n,21} &\leq \Prob\left(\max_{\bm{z}_n \in \mathcal{S}_n}\left|\Phi_{n,1}(\bm{z}_n) \right| \geq {C \over 2}\zeta_J\lambda_J\sqrt{{\log n \over A_n}}\right) + \Prob\left(\max_{\bm{z}_n \in \mathcal{S}_n}\left|\Phi_{n,2}(\bm{z}_n) \right| \geq {C \over 2}\zeta_J\lambda_J\sqrt{{\log n \over A_n}}\right)\\
&=: P_{n,211} + P_{n,212}. 
\end{align*}

First, we control $P_{n,212}$. Note that
\begin{align*}
P_{n,212} &\leq \Prob\left(\max_{\bm{z}_n \in \mathcal{S}_n}\left|\tilde{\psi}_J(\bm{z}_n)'\tilde{\Psi}'_{J,n}e_n^{(2)}/n \right| \geq {C \over 4}\zeta_J\lambda_J\sqrt{{\log n \over A_n}}\right)\\
&\quad + \Prob\left(\max_{\bm{z}_n \in \mathcal{S}_n}\left|\E\left[\tilde{\psi}_J(\bm{z}_n)'\tilde{\Psi}'_{J,n}e_n^{(2)}/n\right] \right| \geq {C \over 4}\zeta_J\lambda_J\sqrt{{\log n \over A_n}}\right).
\end{align*}
Observe that
\begin{align}
\Prob\left(\max_{\bm{z}_n \in \mathcal{S}_n}\left|\tilde{\psi}_J(\bm{z}_n)'\tilde{\Psi}'_{J,n}e_n^{(2)}/n \right| \geq {C \over 4}\zeta_J\lambda_J\sqrt{{\log n \over A_n}}\right) &\leq \Prob\left(|e(\bm{S}_i)|>\tau_n\ \text{for some $i = 1,\dots,n$}\right) \nonumber \\
&\leq \tau_n^{-q}\sum_{i=1}^{n}\E[|e(\bm{S}_i)|^q] = n\tau_n^{-q} = (\log n)^{-q\iota} \to 0. \label{ineq:P_n212_1}
\end{align}
Moreover, we have
\begin{align}
&\E\left[\left|\tilde{\psi}_J(\bm{z})'\tilde{\Psi}'_{J,n}e_n^{(2)}/n\right|\right] \nonumber \\
&\leq {1 \over n}\sum_{i=1}^{n}\E\left[\left|\psi_J(\bm{z})'\check{\Psi}_J^{-1}\psi_J\left({\bm{S}_i \over A_n}\right)e_{n,i}^{(2)}\right|\right] \lesssim \E\left[\left|\psi_J(\bm{z})'\check{\Psi}_J^{-1}\psi_J\left({\bm{S}_i \over A_n}\right)\right| |e(\bm{S}_i)|1\{|e(\bm{S}_i)|>\tau_n\}\right] \nonumber \\
&\lesssim {\zeta_J^2\lambda_J^2 \over \tau_n^{q-1}}\E\left[\E\left[|e(\bm{S}_1)|^q|\bm{S}_1\right]\right] \lesssim {\zeta_J^2\lambda_J^2 \over \tau_n^{q-1}} \lesssim \zeta_J\lambda_J\sqrt{{\log n \over A_n}} \label{ineq:P_n212_2}
\end{align}
uniformly over $\bm{z} \in R_0$. For the last wave relation, we used Assumption \ref{Ass:ze-lam} (ii). Then (\ref{ineq:P_n212_1}) and (\ref{ineq:P_n212_2}) yield that $P_{n,212}$ can be made arbitrarily small for large enough $C \geq 1$. 

Next, we control $P_{n,211}$. For this, we introduce some notations. Define $\Gamma_n(\bm{\ell};\bm{0}) = \prod_{j=1}^{d}((\ell_j-1/2)A_{n3,j}, (\ell_j+1/2)A_{n3,j}]$ with $A_{n3,j} = A_{n1,j} + A_{n2,j}$, and the following hypercubes, $\Gamma_n(\bm{\ell};\bm{\Delta}) = \prod_{j=1}^{d}I_j(\Delta_j)$, $\bm{\Delta} = (\Delta_1,\dots, \Delta_d)' \in \{1,2\}^d$, where 
\begin{align*}
I_j(\Delta_j) &= 
\begin{cases}
((\ell_j-1/2)A_{n3,j}, (\ell_j-1/2)A_{n3,j} + A_{n1,j}] & \text{if $\Delta_j = 1$},\\
((\ell_j-1/2)A_{n3,j} + A_{n1,j}, (\ell_j+1/2)A_{n3,j}] & \text{if $\Delta_j = 2$}. 
\end{cases}
\end{align*}
Let $\bm{\Delta}_{0} = (1,\dots, 1)'$. The partitions $\Gamma_n(\bm{\ell};\bm{\Delta}_{0})$ correspond to ``large blocks'' and the partitions $\Gamma_n(\bm{\ell};\bm{\Delta})$ for $\bm{\Delta} \neq \bm{\Delta}_{0}$ correspond to ``small blocks''. 
Let $L_{n1} = \{\bm{\ell} \in \mathbb{Z}^{d}: \Gamma_n(\bm{\ell};\bm{0}) \subset R_n \}$ denote the index set of all hypercubes $\Gamma_n(\bm{\ell};\bm{0})$ that are contained in $R_n$, and let $L_{n2} = \{\bm{\ell} \in \mathbb{Z}^{d}:  \Gamma_n(\bm{\ell};\bm{0}) \cap R_n \neq \emptyset,  \Gamma_{n}(\bm{\ell};\bm{0}) \cap R_n^c \neq \emptyset \}$ be the index set of boundary hypercubes. 

Define $\phi_{n,i}(\bm{z}) = \tilde{\psi}_J(\bm{z})'\tilde{\psi}_J\left({\bm{S}_i \over A_n}\right)e_{n,i}^{(1)} - \E\left[\tilde{\psi}_J(\bm{z})'\tilde{\psi}_J\left({\bm{S}_i \over A_n}\right)e_{n,i}^{(1)}\right]$. 
Note that 
\begin{align*}
n\Phi_{n,1}(\bm{z}) &= \sum_{i=1}^{n}\phi_{n,i}(\bm{z}) = \sum_{\bm{\ell} \in L_{n1}}\!\!\!\!\phi_1^{(\bm{\ell};\bm{\Delta}_0)}(\bm{z}) \!+\! \sum_{\bm{\Delta} \neq \bm{\Delta}_0}\sum_{\bm{\ell} \in L_{n1}}\!\!\!\!\phi_1^{(\bm{\ell};\bm{\Delta})}(\bm{z}) \!+ \!\sum_{\bm{\Delta} \in \{1,2\}^d}\sum_{\bm{\ell} \in L_{n2}}\!\!\!\!\phi_1^{(\bm{\ell};\bm{\Delta})}(\bm{z}),
\end{align*}
where $\phi_1^{(\bm{\ell};\bm{\Delta})}(\bm{z}) = \sum_{i=1}^{n}\phi_{n,i}(\bm{z})1\{\bm{S}_i \in \Gamma_n(\bm{\ell};\bm{\Delta}) \cap R_n\}$. 

For $\bm{\Delta} \in \{1,2\}^d$, let $\{\tilde{\phi}_1^{(\bm{\ell};\bm{\Delta})}(\bm{z})\}_{\bm{\ell} \in L_{n1} \cup L_{n2}}$ be independent random variables such that $\phi_1^{(\bm{\ell};\bm{\Delta})}(\bm{z})$ and $\tilde{\phi}_1^{(\bm{\ell};\bm{\Delta})}(\bm{z})$ have the same distribution. Applying Lemma \ref{lem:indep_block} with $M_h = 1$, $m \sim A_n(A_n^{(1)})^{-1}$, and $\tau \sim \beta(\underline{A}_{n2}; A_n)$, we have that for $\bm{\Delta} \in \{1,2\}^d$, 
\begin{align}
&\sup_{t >0}\left|\Prob\left(\left|\sum_{\bm{\ell} \in L_{n1}}\phi_1^{(\bm{\ell};\bm{\Delta})}(\bm{z})\right|>t\right) - \Prob\left(\left|\sum_{\bm{\ell} \in L_{n1}}\tilde{\phi}_1^{(\bm{\ell};\bm{\Delta})}(\bm{z})\right|>t\right)\right| \lesssim \left({A_n \over A_n^{(1)}}\right)\beta(\underline{A}_{n2}; A_n) = o(1), \label{ineq:phi1-beta-block1}\\
&\sup_{t >0}\left|\Prob\left(\left|\sum_{\bm{\ell} \in L_{n2}}\phi_1^{(\bm{\ell};\bm{\Delta})}(\bm{z})\right|>t\right) - \Prob\left(\left|\sum_{\bm{\ell} \in L_{n2}}\tilde{\phi}_1^{(\bm{\ell};\bm{\Delta})}(\bm{z})\right|>t\right)\right| \lesssim \left({A_n  \over A_n^{(1)}}\right)\beta(\underline{A}_{n2}; A_n ) = o(1) \label{ineq:phi1-beta-block2}.
\end{align}
These results imply that
\begin{align*}
\sum_{\bm{\ell} \in L_{n1}}\phi_1^{(\bm{\ell};\bm{\Delta})}(\bm{z}) &= O_p\left(\sum_{\bm{\ell} \in L_{n1}}\tilde{\phi}_1^{(\bm{\ell};\bm{\Delta})}(\bm{z})\right),\ \sum_{\bm{\ell} \in L_{n2}}\phi_1^{(\bm{\ell};\bm{\Delta})}(\bm{z}) = O_p\left(\sum_{\bm{\ell} \in L_{n2}}\tilde{\phi}_1^{(\bm{\ell};\bm{\Delta})}(\bm{z})\right). 
\end{align*}
From (\ref{ineq:phi1-beta-block1}) and (\ref{ineq:phi1-beta-block2}), we have
\begin{align*}
&\Prob\left(\max_{\bm{z}_n \in \mathcal{S}_n}\left|\Phi_{n,1}(\bm{z}_n)\right|>{C \over 2}\zeta_J\lambda_J\sqrt{{\log n \over A_n}}\right)\\ 
&\leq [\![\mathcal{S}_n]\!]\max_{\bm{z}_n \in \mathcal{S}_n}\Prob\left(\left|\Phi_{n,1}(\bm{z}_n)\right|>{C \over 2}\zeta_J\lambda_J\sqrt{{\log n \over A_n}}\right)\\
&\leq \sum_{\bm{\Delta} \in \{1,2\}^{d}}\hat{Q}_{n1}(\bm{\Delta}) + \sum_{\bm{\Delta} \in \{1,2\}^{d}}\hat{Q}_{n2}(\bm{\Delta}) + 2^{d+1}[\![\mathcal{S}_n]\!]\left({A_n \over A_n^{(1)}}\right)\beta(\underline{A}_{n2};A_n)\\
&\lesssim \sum_{\bm{\Delta} \in \{1,2\}^{d}}\hat{Q}_{n1}(\bm{\Delta}) + \sum_{\bm{\Delta} \in \{1,2\}^{d}}\hat{Q}_{n2}(\bm{\Delta}) + \underbrace{\left({n^{d\eta_1}A_n \over A_n^{(1)}}\right)\beta(\underline{A}_{n2};A_n)}_{=o(1)},
\end{align*}
where 
\begin{align*}
\hat{Q}_{nj}(\bm{\Delta}) &= [\![\mathcal{S}_n]\!]\max_{\bm{z}_n \in \mathcal{S}_n}\Prob\left(\left|\sum_{\bm{\ell} \in L_{nj}}\tilde{\phi}_1^{(\bm{\ell};\bm{\Delta})}(\bm{z}_n)\right|>{C \over 2^{d+2}}n\zeta_J\lambda_J\sqrt{{\log n \over A_n}}\right),\ j=1,2.
\end{align*}

Now we restrict our attention to show that $\hat{Q}_{n1}(\bm{\Delta}) = o(1)$ for $\bm{\Delta} \neq \bm{\Delta}_{0}$. The proofs for other cases are similar. Note that
\begin{align*}
\Prob\left(\left|\sum_{\bm{\ell} \in L_{n1}}\tilde{\phi}_1^{(\bm{\ell};\bm{\Delta})}(\bm{z})\right|>{C \over 2^{d+2}}n\zeta_J\lambda_J\sqrt{{\log n \over A_n}}\right)
&\leq 2\Prob\left(\sum_{\bm{\ell} \in L_{n1}}\tilde{\phi}_1^{(\bm{\ell};\bm{\Delta})}(\bm{z})>{C \over 2^{d+2}}n\zeta_J\lambda_J\sqrt{{\log n \over A_n}}\right).
\end{align*}
Observe that $\tilde{\phi}_1^{(\bm{\ell};\bm{\Delta})}(\bm{z})$ are zero-mean independent random variables, 
\begin{align}\label{ineq:Bernstein-1}
\left|\tilde{\phi}_1^{(\bm{\ell};\bm{\Delta})}(\bm{z})\right| &\leq C_{\phi_{11}}\zeta_J^2\lambda_J^2(\overline{A}_{n1})^{d-1}\overline{A}_{n2}nA_n^{-1}\tau_{n},\ a.s.\ (\text{from Lemma \ref{lem:number-summands}})
\end{align}
for some $C_{\phi_{11}}>0$ and 
\begin{align*}
&\sum_{\bm{\ell} \in L_{n1}}\E\left[\left(\tilde{\phi}_1^{(\bm{\ell};\bm{\Delta})}(\bm{z})\right)^{2}\right] \nonumber \\ 
&\lesssim \sum_{\bm{\ell} \in L_{n1}}n\E\left[\left|\psi_J(\bm{z})'\check{\Psi}_J^{-1}\psi_J\left({\bm{S}_1 \over A_n}\right)\psi_J\left({\bm{S}_1 \over A_n}\right)'\check{\Psi}_J^{-1}\psi_J(\bm{z})\right|1\{\bm{S}_1 \in \Gamma_n(\bm{\ell};\bm{\Delta}) \cap R_n \}\right] \nonumber \\
&\quad + \sum_{\bm{\ell} \in L_{n1}}\sum_{i_1 \neq i_2}\E\left[\left|\psi_J(\bm{z})'\check{\Psi}_J^{-1}\psi_J\left({\bm{S}_{i_1} \over A_n}\right)\psi_J\left({\bm{S}_{i_2} \over A_n}\right)'\check{\Psi}_J^{-1}\psi_J(\bm{z})\right| \right. \nonumber \\
&\left. \quad \quad \times \left|\sigma_{\bm{e}}(\bm{S}_{i_1} - \bm{S}_{i_2})\right|1\{\bm{S}_{i_1}, \bm{S}_{i_2} \in \Gamma_n(\bm{\ell};\bm{\Delta}) \cap R_n \}\right] \nonumber \\
&= n\E\left[\left|\psi_J(\bm{z})'\check{\Psi}_J^{-1}\psi_J\left({\bm{S}_1 \over A_n}\right)\psi_J\left({\bm{S}_1 \over A_n}\right)'\check{\Psi}_J^{-1}\psi_J(\bm{z})\right|1\{\bm{S}_1 \in \cup_{\bm{\ell} \in L_{n1}}\Gamma_n(\bm{\ell};\bm{\Delta}) \cap R_n \}\right] \nonumber \\
&\quad + \sum_{i_1 \neq i_2}\E\left[\left|\psi_J(\bm{z})'\check{\Psi}_J^{-1}\psi_J\left({\bm{S}_{i_1} \over A_n}\right)\psi_J\left({\bm{S}_{i_2} \over A_n}\right)'\check{\Psi}_J^{-1}\psi_J(\bm{z})\right| \right. \nonumber \\
&\left. \quad \quad \times \left|\sigma_{\bm{e}}(\bm{S}_{i_1} - \bm{S}_{i_2})\right|1\{\bm{S}_{i_1}, \bm{S}_{i_2} \in \left(\cup_{\bm{\ell} \in L_{n1}}\Gamma_n(\bm{\ell};\bm{\Delta})\right) \cap R_n \}\right] \nonumber \\
&\leq n\E\left[\left|\psi_J(\bm{z})'\check{\Psi}_J^{-1}\psi_J\left({\bm{S}_1 \over A_n}\right)\psi_J\left({\bm{S}_1 \over A_n}\right)'\check{\Psi}_J^{-1}\psi_J(\bm{z})\right|\right] \nonumber \\
&\quad + n(n-1)\E\left[\left|\psi_J(\bm{z})'\check{\Psi}_J^{-1}\psi_J\left({\bm{S}_1 \over A_n}\right)\psi_J\left({\bm{S}_2 \over A_n}\right)'\check{\Psi}_J^{-1}\psi_J(\bm{z})\right| \left|\sigma_{\bm{e}}(\bm{S}_1 - \bm{S}_2)\right|\right] =: V_{n,1} + V_{n,2}.
\end{align*}
For $V_{n,1}$, we have 
\begin{align*}
V_{n,1} &= n\psi_J(\bm{z})'\check{\Psi}_J^{-1}\E\left[\psi_J\left({\bm{S}_1 \over A_n}\right)\psi_J\left({\bm{S}_1 \over A_n}\right)'\right]\check{\Psi}_J^{-1}\psi_J(\bm{z}) =n\|\psi_J(\bm{z})\|^2 \lesssim n\zeta_J^2 \lambda_J^2. 
\end{align*}
For $V_{n,2}$, we have
\begin{align*}
V_{n,2}&\leq {n^2 \over A_n^2}\int |\sigma_{\bm{e}}(\bm{s}_1 - \bm{s}_2)|\left|\psi_J(\bm{z})'\check{\Psi}_J^{-1}\psi_J\left({\bm{s}_1 \over A_n}\right)\psi_J\left({\bm{s}_2 \over A_n}\right)'\check{\Psi}_J^{-1}\psi_J(\bm{z})\right|g\left({\bm{s}_1 \over A_n}\right)g\left({\bm{s}_2 \over A_n}\right)d\bm{s}_1 d\bm{s}_2\\
&= n^2\int |\sigma_{\bm{e}}(A_n(\bm{w}_1 - \bm{w}_2))|\left|\psi_J(\bm{z})'\check{\Psi}_J^{-1}\psi_J(\bm{w}_1)\psi_J(\bm{w}_2)'\check{\Psi}_J^{-1}\psi_J(\bm{z})\right|g(\bm{w}_1)g(\bm{w}_2)d\bm{w}_1 d\bm{w}_2\\
&= {n^2 \over A_n}\!\int_{\bar{R}_n} \!\! |\sigma_{\bm{e}}(\bm{x})|\!\left(\int \left|\psi_J(\bm{z})'\check{\Psi}_J^{-1}\psi_J\!\left(\!{\bm{x} \over A_n} \!+ \! \bm{w}_2 \!\right)\!\psi_J(\bm{w}_2)'\check{\Psi}_J^{-1}\psi_J(\bm{z})\right|g\!\left(\!{\bm{x} \over A_n} \!+\! \bm{w}_2 \!\right)g(\bm{w}_2)d\bm{w}_2 \! \right)d\bm{x}\\
&\lesssim {n^2 \over A_n}\left(\int |\sigma_{\bm{e}}(\bm{x})|d\bm{x}\right)\left(\int \left|\psi_J(\bm{z})'\check{\Psi}_J^{-1}\psi_J\left( \bm{w}_2\right)\psi_J(\bm{w}_2)'\check{\Psi}_J^{-1}\psi_J(\bm{z})\right|g(\bm{w}_2)d\bm{w}_2\right) \\
&\leq {n^2 \zeta_J^2 \lambda_J^2 \over A_n}\left(\int |\sigma_{\bm{e}}(\bm{x})|d\bm{x}\right). 
\end{align*}
Therefore, we have
\begin{align}
\sum_{\bm{\ell} \in L_{n1}} \E\left[\left(\tilde{\phi}_1^{(\bm{\ell};\bm{\Delta})}(\bm{z})\right)^{2}\right] &\leq C_{\phi_{12}}\zeta_J^2 \lambda_J^2n^2A_n^{-1} \label{ineq:Bernstein-2}
\end{align}
for some $C_{\phi_{12}}>0$. Then Lemma \ref{lem:Bernstein}, (\ref{ineq:Bernstein-1}), and (\ref{ineq:Bernstein-2}) yield that 
\begin{align*}
\Prob\left(\sum_{\bm{\ell} \in L_{n1}}\tilde{\phi}_1^{(\bm{\ell};\bm{\Delta})}(\bm{z}_n)>{C \over 2^{d+2}}n\zeta_J\lambda_J\sqrt{{\log n \over A_n}}\right) &\leq \exp \left(-{E_{n0} \over E_{n1} + E_{n2} }\right),
\end{align*}
where 
\[
E_{n0} \!=\! {C^2\zeta_J^2\lambda_J^2 n^2 \!  \log n \over 2^{2d+5}A_n},\ E_{n1} \!=\!  {C_{\phi_{11}}\zeta_J^2 \lambda_J^2 n^2\over A_n},\ E_{n2} \!=\! {CC_{\phi_{12}}\zeta_J^3\lambda_J^3(\overline{A}_{n1})^{d-1}\overline{A}_{n2}n^2\!\tau_{n}(\log n)^{1/2} \over 3\cdot 2^{d+2}A_n^{3/2}}.
\]
From Assumption \ref{Ass:ze-lam} (iii), we have $E_{n0}/E_{n1} \gtrsim \log n$ and $E_{n0}/E_{n2} \gtrsim \log n$. This yields $\hat{Q}_{n1}(\bm{\Delta}) = o(1)$ for large enough $C \geq 1$. Likewise, we can show that $\hat{Q}_{nj}(\bm{\Delta}) = o(1)$, $\bm{\Delta} \in \{1,2\}^d$, $j=1,2$. Therefore, we obtain $P_{n,211} = o(1)$ for large enough $C \geq 1$.

(Step 4) Now we show (\ref{eq:ridge-bias}). By the mean value theorem, for any $\bm{z},\bm{z}^* \in R_0$ we have
\begin{align*}
&|\dot{m}(\bm{z}) - \dot{m}(\bm{z}^*)|\\ 
&= \left|(\tilde{\psi}_J(\bm{z})-\tilde{\psi}_J(\bm{z}^*))'\left\{(\tilde{\Psi}'_{J,n}\tilde{\Psi}_{J,n}/n + \varsigma_{J,n} \check{\Psi}_J^{-1})^{-1} - (\tilde{\Psi}'_{J,n}\tilde{\Psi}_{J,n}/n)^- \right\} \tilde{\Psi}'_{J,n}M_0/n\right|\\
&= \left|(\bm{z}-\bm{z}^*)'\nabla \tilde{\psi}_J(\bm{z}^{**})'\left\{(\tilde{\Psi}'_{J,n}\tilde{\Psi}_{J,n}/n + \varsigma_{J,n} \check{\Psi}_J^{-1})^{-1} - (\tilde{\Psi}'_{J,n}\tilde{\Psi}_{J,n}/n)^- \right\} \tilde{\Psi}'_{J,n}M_0/n\right|\\
&\leq \|\nabla \tilde{\psi}_J(\bm{z}^{**})\|\|\bm{z} - \bm{z}^*\|\|\left\{(\tilde{\Psi}'_{J,n}\tilde{\Psi}_{J,n}/n + \varsigma_{J,n} \check{\Psi}_J^{-1})^{-1} - (\tilde{\Psi}'_{J,n}\tilde{\Psi}_{J,n}/n)^- \right\}| \|\tilde{\Psi}'_{J,n}M_0/n\|\\
&\leq \|\check{\Psi}_J^{-1/2}\|\|\nabla \psi_J(\bm{z}^{**})\|\|\bm{z} - \bm{z}^*\|\|\left\{(\tilde{\Psi}'_{J,n}\tilde{\Psi}_{J,n}/n + \varsigma_{J,n} \check{\Psi}_J^{-1})^{-1} - (\tilde{\Psi}'_{J,n}\tilde{\Psi}_{J,n}/n)^- \right\}\| \|\tilde{\Psi}'_{J,n}M_0/n\|\\
&= \lambda_{\text{max}}(\check{\Psi}_J^{-1/2})\|\nabla \psi_J(\bm{z}^{**})\|\|\bm{z} - \bm{z}^*\|\|\left\{(\tilde{\Psi}'_{J,n}\tilde{\Psi}_{J,n}/n + \varsigma_{J,n} \check{\Psi}_J^{-1})^{-1} - (\tilde{\Psi}'_{J,n}\tilde{\Psi}_{J,n}/n)^- \right\}\| \|\tilde{\Psi}'_{J,n}M_0/n\|\\
&\leq C_{\nabla}\lambda_J J^\omega\|\bm{z} - \bm{z}^*\|\|\left\{(\tilde{\Psi}'_{J,n}\tilde{\Psi}_{J,n}/n + \varsigma_{J,n} \check{\Psi}_J^{-1})^{-1} - (\tilde{\Psi}'_{J,n}\tilde{\Psi}_{J,n}/n)^- \right\}\| \|\tilde{\Psi}'_{J,n}M_0/n\|
\end{align*}
for some $\bm{z}^{**}$ in the segment between $\bm{z}$ and $\bm{z}^*$ and some finite constant $C_{\nabla}$ which is independent of $\bm{z},\bm{z}^*,n$, and $J$. 

Observe that
\begin{align*}
\E\left[\left\|\tilde{\Psi}'_{J,n}M_0/n\right\|^2\right] &= {1 \over n^2}\!\sum_{i_1=1}^n\sum_{i_2=1}^n \!\E\!\left[\psi'_J\!\left({\bm{S}_{i_1} \over A_n}\right)\check{\Psi}_J^{-1}\psi_J\!\left({\bm{S}_{i_2} \over A_n}\right)m_0\!\left({\bm{S}_{i_1} \over A_n}\right)m_0\!\left({\bm{S}_{i_2} \over A_n}\right)\right] \leq \left({1 \over n} \!+\! 1\right){\zeta_J^2 \lambda_J^2\|m_0\|_\infty^2}. 
\end{align*}
Then, applying Markov's inequality, we have
\begin{align}\label{eq:ridge-Psi-M-norm}
\|\tilde{\Psi}'_{J,n}M_0/n\| = O_p\left(\zeta_J\lambda_J\|m_0\|_\infty\right). 
\end{align}
From Lemmas \ref{lem:psi} and \ref{lem:matrix-inv-norm} and Assumptions \ref{Ass:ze-lam} (i), We have 
\begin{align}\label{ineq:ridge-Psi-norm}
&\|(\tilde{\Psi}'_{J,n}\tilde{\Psi}_{J,n}/n + \varsigma_{J,n} \check{\Psi}_J^{-1})^{-1} - (\tilde{\Psi}'_{J,n}\tilde{\Psi}_{J,n}/n)^{-1}\| \nonumber \\ 
&\lesssim {1 \over \varsigma_{J,n}^{-1}\lambda_{\text{min}}(\check{\Psi}_J) - \varsigma_{J,n}^{-1}\lambda_{\text{max}}(\check{\Psi}_J)\left\{\|(\tilde{\Psi}'_{J,n}\tilde{\Psi}_{J,n}/n + \varsigma_{J,n} \check{\Psi}_J^{-1}) - I_J\| + \|\tilde{\Psi}'_{J,n}\tilde{\Psi}_{J,n}/n  - I_J\|\right\} } \nonumber \\
&= O_p\left({1 \over \varsigma_{J,n}^{-1}\lambda_J^2 - \varsigma_{J,n}^{-1}\zeta_J^2(\zeta_J\lambda_J\sqrt{\log J/n})}\right) = O_p(\varsigma_{J,n}\lambda_J^{-2}). 
\end{align}
Together with (\ref{eq:ridge-Psi-M-norm}) and (\ref{ineq:ridge-Psi-norm}), we have 
\begin{align*}
\limsup_{n \to \infty}\Prob \left(\!C_\nabla \!\left\|\left\{\!\!\left({\tilde{\Psi}'_{J,n}\tilde{\Psi}_{J,n} \over n} \!+ \varsigma_{J,n} \check{\Psi}_J^{-1}\!\! \right)^{-1} \!\!\!-\! \left({\tilde{\Psi}'_{J,n}\tilde{\Psi}_{J,n} \over n}\!\right)^-\!\right\}\right\| \! \left\|{\tilde{\Psi}'_{J,n}M_0 \over n} \right\|\!> \bar{M}_1\right)\! = 0
\end{align*}
for any fixed $\bar{M}_1>0$ since $\varsigma_{J,n}\zeta_J\lambda_J^{-1} \lesssim \zeta_J^2 \lambda_J^{-2}\sqrt{\log n/n} = o(1)$. Let $\mathcal{C}_n$ denote the event on which $C_\nabla\|(\tilde{\Psi}'_{J,n}\tilde{\Psi}_{J,n}/n \!+ \varsigma_{J,n} \check{\Psi}_J^{-1}/n)^{-1} \!- (\tilde{\Psi}'_{J,n}\tilde{\Psi}_{J,n}/n)^-\| \|\tilde{\Psi}'_{J,n}M_0/n\| \leq \bar{M}_1$ and observe that $\Prob(\mathcal{C}_n^c) = o(1)$. On $\mathcal{C}_n$, for any $C \geq 1$, a finite positive $\eta_3 = \eta_3(C)$ and $\eta_4 = \eta_4(C)$ can be chosen such that 
\begin{align*}
C_{\nabla}\lambda_J J^{\omega}\|\bm{z} - \bm{z}^*\|\|(\tilde{\Psi}'_{J,n}\tilde{\Psi}_{J,n}/n + \varsigma_{J,n} \check{\Psi}_J^{-1}/n)^{-1} - (\tilde{\Psi}'_{J,n}\tilde{\Psi}_{J,n}/n)^-\| \|\tilde{\Psi}'_{J,n}M_0/n\| \leq C\varsigma_{J,n} \zeta_J^2\|m_0\|_\infty 
\end{align*} 
whenever $\|\bm{z} - \bm{z}^*\| \leq \eta_4 n^{-\eta_3}$. Let $\mathcal{T}_n$ be the smallest subset of $R_0$ such that for each $\bm{z} \in R_0$ there exists a $\bm{z}_n \in R_0$ with $\|\bm{z} - \bm{z}_n\| \leq \eta_4 n^{-\eta_3}$. For any $\bm{z} \in R_0$ let $\bm{z}_n(\bm{z})$ denote the $\bm{z}_n \in \mathcal{T}_n$ nearest to $\bm{z}$ in Euclidean distance. Then on $\mathcal{C}_n$, we have $|\dot{m}(\bm{z}) - \dot{m}(\bm{z}_n(\bm{z}))| \leq C\varsigma_{J,n} \zeta_J ^2\|m_0\|_\infty$ for any $\bm{z} \in R_0$. Then we have
\begin{align*}
\Prob\left(\|\dot{m}\|_\infty \geq 4C\varsigma_{J,n} \zeta_J^2\|m_0\|_\infty \right) &\leq \Prob\left(\left\{\|\dot{m}\|_\infty \geq 4C\varsigma_{J,n} \zeta_J^2|m_0\|_\infty \right\} \cap \mathcal{C}_n\right) + \Prob(\mathcal{C}_n^c)\\
&\leq \Prob\left(\left\{\sup_{\bm{z} \in R_0}|\dot{m}(\bm{z}) - \dot{m}(\bm{z}_n(\bm{z}))| \geq 2C\varsigma_{J,n} \zeta_J^2\|m_0\|_\infty \right\} \cap \mathcal{C}_n\right)\\
&\quad + \Prob\left(\left\{\max_{\bm{z}_n \in \mathcal{T}_n}|\dot{m}(\bm{z}_n)| \geq 2C\varsigma_{J,n} \zeta_J^2\|m_0\|_\infty \right\} \cap \mathcal{C}_n\right) + \Prob(\mathcal{C}_n^c)\\
&= \Prob\left(\left\{\max_{\bm{z}_n \in \mathcal{T}_n}|\dot{m}(\bm{z}_n)| \geq 2C\varsigma_{J,n} \zeta_J^2\|m_0\|_\infty \right\} \cap \mathcal{C}_n\right) + o(1) =: P_{n,3} + o(1). 
\end{align*}
Since
\begin{align*}
&\left|\tilde{\psi}_J(\bm{z}_n)'\left\{(\tilde{\Psi}'_{J,n}\tilde{\Psi}_{J,n}/n + \varsigma_{J,n} \check{\Psi}_J^{-1}/n)^{-1} - (\tilde{\Psi}'_{J,n}\tilde{\Psi}_{J,n}/n)^-\right\} \tilde{\Psi}'_{J,n}M_0/n \right|\\ 
&\lesssim \zeta_J\lambda_J\|(\tilde{\Psi}'_{J,n}\tilde{\Psi}_{J,n}/n + \varsigma_{J,n} \check{\Psi}_J^{-1}/n)^{-1} - (\tilde{\Psi}'_{J,n}\tilde{\Psi}_{J,n}/n)^-\| \|\tilde{\Psi}'_{J,n}M_0/n\|\\
&= O_p\left(\varsigma_{J,n} \zeta_J^2 \|m_0\|_\infty \right),
\end{align*}
$P_{n,3}$ can be made arbitrarily small for large enough $C\geq 1$. Therefore, we complete the proof. 

\subsection{Proof of Theorem \ref{thm:asy-norm1}}

Let $\bm{z}=\bm{z}_1$. We prove 
\begin{align*}
{\sqrt{A_n}(\hat{m}(\bm{z}) - m_0(\bm{z})) \over \sqrt{\Omega_J(\bm{z})}} &\stackrel{d}{\to} N(0,1).
\end{align*}
For the multivariate case (i.e., $L \geq 2$), the desired result follows from the Cam\'er-Wold device. 

Note that 
\begin{align*}
{\sqrt{A_n}(\hat{m}(\bm{z}) - m_0(\bm{z})) \over \sqrt{\Omega_J(\bm{z})}} &= {\sqrt{A_n}(\hat{m}(\bm{z}) - \bar{m}(\bm{z})) \over \sqrt{\Omega_J(\bm{z})}} + {\sqrt{A_n}(\bar{m}(\bm{z}) - \tilde{m}(\bm{z})) \over \sqrt{\Omega_J(\bm{z})}} + {\sqrt{A_n}(\tilde{m}(\bm{z}) - m_0(\bm{z})) \over \sqrt{\Omega_J(\bm{z})}}.
\end{align*}

Observe that
\begin{align*}
\Omega_J(\bm{z}) &=  \tilde{\psi}_J(\bm{z})'\left\{\kappa\int (\eta^2(\bm{v}) + \sigma_\varepsilon^2(\bm{v}))\tilde{\psi}_J(\bm{v})\tilde{\psi}_J(\bm{v})'g(\bm{v})d\bm{v} \right.\\
&\left. \quad + \left(\int \eta^2(\bm{v})\tilde{\psi}_J(\bm{v})\tilde{\psi}_J(\bm{v})'g^2(\bm{v})d\bm{v}\right)\left(\int \sigma_{\bm{e}}(\bm{x})d\bm{x}\right)\right\}\tilde{\psi}_J(\bm{z})\\
&\sim \tilde{\psi}_J(\bm{z})'\left\{\int \tilde{\psi}_J(\bm{v})\tilde{\psi}_J(\bm{v})'g(\bm{v})d\bm{v}\right\}\tilde{\psi}_J(\bm{z}) = \|\tilde{\psi}_J(\bm{z})\|^2.
\end{align*}
Condition (b) yields that 
\[
\left|{\sqrt{A_n}(\tilde{m}(\bm{z}) - m_0(\bm{z})) \over \sqrt{\Omega_J(\bm{z})}}\right| = O_p\left(\sqrt{A_n}\|\tilde{\psi}_J(\bm{z})\|^{-1}\inf_{m \in \Psi_J}\|m_0 - m\|_\infty\right) = o_p(1). 
\]
Condition (c) yields that 
\begin{align*}
\left|{\sqrt{A_n}(\bar{m}(\bm{z}) - \tilde{m}(\bm{z})) \over \sqrt{\Omega_J(\bm{z})}}\right| &= O_p\left(\sqrt{A_n}\|\tilde{\psi}_J(\bm{z})\|^{-1}\|\bar{m} - \tilde{m}\|_\infty\right) = O_p(\sqrt{A_n}\|\tilde{\psi}_J(\bm{z})\|^{-1}\varsigma_{J,n} \zeta_J^2) = o_p(1). 
\end{align*}
Then it is sufficient to show that 
\begin{align}\label{eq:asy-norm-db1}
{\sqrt{A_n}(\hat{m}(\bm{z}) - \bar{m}(\bm{z})) \over \sqrt{\Omega_J(\bm{z})}} \stackrel{d}{\to} N(0,1). 
\end{align}

We show (\ref{eq:asy-norm-db1}) in several steps. 

(Step 1) In this step, we give an overview of the proof of (\ref{eq:asy-norm-db1}). Recall that 
\begin{align*}
&{\sqrt{A_n}\left(\hat{m}(\bm{z}) - \bar{m}(\bm{z})\right) \over \sqrt{\Omega_J(\bm{z}})} \\
&= \sqrt{{A_n \over  \Omega_J(\bm{z})}}\tilde{\psi}_J(\bm{z})'(\tilde{\Psi}'_{J,n}\tilde{\Psi}_{J,n}/n + \varsigma_{J,n} \check{\Psi}_J^{-1})^{-1} \left(\tilde{\Psi}'_{J,n}v_n \over n \right)\\
&= \sqrt{{A_n \over  \Omega_J(\bm{z})}}\tilde{\psi}_J(\bm{z})'\left(\tilde{\Psi}'_{J,n}v_n \over n \right) + \sqrt{{A_n \over  \Omega_J(\bm{z})}}\tilde{\psi}_J(\bm{z})'\left\{(\tilde{\Psi}'_{J,n}\tilde{\Psi}_{J,n}/n + \varsigma_{J,n} \check{\Psi}_J^{-1})^{-1} - I_J\right\}\left(\tilde{\Psi}'_{J,n}v_n \over n \right)\\
&=: K_{n,1} + K_{n,2}.
\end{align*}
For $K_{n,2}$, from Lemma \ref{lem:psi} and (\ref{eq:Psi-v-norm}), we have
\begin{align}\label{eq:K_n2-asy-neg}
|K_{n,2}| &\leq \sqrt{A_n}\left\|{\tilde{\psi}_J(\bm{z}) \over \sqrt{\Omega_J(\bm{z})}}\right\|\|(\tilde{\Psi}'_{J,n}\tilde{\Psi}_{J,n}/n + \varsigma_{J,n} \check{\Psi}_J^{-1})^{-1} - I_J\| \left\|{\tilde{\Psi}'_{J,n}v_n \over n }\right\| \nonumber \\
&= \sqrt{A_n}\times O_p\left(\zeta_J \lambda_J\sqrt{\log J \over n} + \varsigma_{J,n} \lambda_J^2\right)\times O_p\left({\zeta_J\lambda_J \over \sqrt{A_n}}\right) = o_p(1). 
\end{align}
For $K_{n,1}$, note that 
\begin{align*}
\tilde{K}_{n,1} &:= {\sqrt{n^2A_n^{-1}\Omega_J(\bm{z})} \over \|\tilde{\psi}_J(\bm{z})\|}K_{n,1} = \sum_{i=1}^{n}\left({\tilde{\psi}_J(\bm{z}) \over \|\tilde{\psi}_J(\bm{z})\|}\right)'\tilde{\psi}_J\left({\bm{S}_i \over A_n}\right)(e_{n,i}+\varepsilon_{n,i}) =: \sum_{i=1}^{n}k_J\left({\bm{S}_i \over A_n}\right)(e_{n,i}+\varepsilon_{n,i})
\end{align*}
and $\|k_J\|_\infty \leq \zeta_J \lambda_J$. Decompose
\begin{align*}
\tilde{K}_{n,1} &= 
 \sum_{\bm{\ell} \in L_{n1}}\!\!\!\!k_1^{(\bm{\ell};\bm{\Delta}_0)} \!+\! \sum_{\bm{\Delta} \neq \bm{\Delta}_0}\sum_{\bm{\ell} \in L_{n1}}\!\!\!\!k_1^{(\bm{\ell};\bm{\Delta})} \!+ \!\sum_{\bm{\Delta} \in \{1,2\}^d}\sum_{\bm{\ell} \in L_{n2}}\!\!\!\!k_1^{(\bm{\ell};\bm{\Delta})} =: K_{n,11} + K_{n,12} + K_{n,13},
\end{align*}
where $k_1^{(\bm{\ell};\bm{\Delta})} = \sum_{i=1}^{n}k_J(\bm{S}_i/A_n)(e_{n,i}+\varepsilon_{n,i})1\{\bm{S}_i \in \Gamma_n(\bm{\ell};\bm{\Delta}) \cap R_n \}$. 

Note that for $\bm{\ell}_1, \bm{\ell}_2 \in L_{n1}$ with $\bm{\ell}_1 \neq \bm{\ell}_2$, 
\begin{align}\label{ineq:Ga12-dist}
d\left(\Gamma_n(\bm{\ell}_1;\bm{\Delta}_0), \Gamma_n(\bm{\ell}_2;\bm{\Delta}_0)\right) &\geq \max\{|\bm{\ell}_1 - \bm{\ell}_2|-d,0\}\underline{A}_{n3} + \underline{A}_{n2},
\end{align}
where $\underline{A}_{n3} = \min_{1 \leq j \leq d}A_{n3,j}$, $A_{n3,j}=A_{n1,j} + A_{n2,j}$, and $\underline{A}_{n2} = \min_{1 \leq j \leq d}A_{n2,j}$. 

Hence, by the Volkonskii-Rozanov inequality (cf. Proposition 2.6 in \cite{FaYa03}), we have
\begin{align}\label{Vn1-indep-approx}
\left|\E[\exp (\mathrm{i}u K_{n,11})] - \prod_{\ell \in L_{n1}}\E[\exp(\mathrm{i}u k_1^{(\bm{\ell};\bm{\Delta}_0)})]\right| &\lesssim {A_n \over A_n^{(1)}}\alpha(\underline{A}_{n2};A_n). 
\end{align}

From Lyapounov's central limit theorem, it is sufficient to verify the following conditions to show (\ref{eq:asy-norm-db1}): As $n,J \to \infty$, 
\begin{align}
&{A_n \over n^2}\E[\tilde{K}_{n,1}^2] =  \left({\tilde{\psi}_J(\bm{z}) \over \|\tilde{\psi}_J(\bm{z})\|}\right)'G_J\left({\tilde{\psi}_J(\bm{z}) \over \|\tilde{\psi}_J(\bm{z})\|}\right)(1 + o(1)), \label{eq:tilde_K_n1-var}\\
&\sum_{\bm{\ell} \in L_{n1}}\E[(k_1^{(\bm{\ell}; \bm{\Delta}_0)})^2] - \E[\tilde{K}_{n,1}^2] = o\left(n^2 A_n^{-1}\right), \label{eq:Kn-var-diff}\\
&\sum_{\bm{\ell} \in L_{n1}}\E[(k_1^{(\bm{\ell}; \bm{\Delta}_0)})^4] = o\left(n^4 A_n^{-2}\right), \label{eq:Kn-4}\\
&\Var(K_{n,12}) = o\left(n^2A_n^{-1}\right), \label{eq:Kn12-var}\\
&\Var(K_{n,13}) = o\left(n^2A_n^{-1}\right). \label{eq:Kn13-var}
\end{align}

We show (\ref{eq:tilde_K_n1-var}) in Step 2, (\ref{eq:Kn-4}) in Step 3, (\ref{eq:Kn12-var}) and (\ref{eq:Kn13-var}) in Step 4, and (\ref{eq:Kn-var-diff}) in Step 5.

(Step 2) Now we show (\ref{eq:tilde_K_n1-var}). Observe that 
\begin{align*}
\E[\tilde{K}_{n,1}^2] &= n\E\left[k_J^2\left({\bm{S}_1 \over A_n}\right)(e_{n,1}^2+\varepsilon_{n,1}^2)\right]  + n(n-1)\E\left[k_J\left({\bm{S}_1 \over A_n}\right)k_J\left({\bm{S}_2 \over A_n}\right)e_{n,1}e_{n,2}\right]\\
&= {n \over A_n}\int k_J^2\left({\bm{s} \over A_n}\right)\left\{\eta^2\left({\bm{s} \over A_n}\right) + \sigma_\varepsilon^2\left({\bm{s} \over A_n}\right)\right\}g\left({\bm{s} \over A_n}\right)d\bm{s}\\
&\quad + {n(n-1) \over A_n^2}\!\!\int \! k_J\!\left({\bm{s}_1 \over A_n}\right)k_J\!\left({\bm{s}_2 \over A_n}\right)\eta \!\left({\bm{s}_1 \over A_n}\right)\eta \!\left({\bm{s}_2 \over A_n}\right)\sigma_{\bm{e}}(\bm{s}_1 - \bm{s}_2)g \!\left({\bm{s}_1 \over A_n}\right)g \!\left({\bm{s}_2 \over A_n}\right)d\bm{s}_1 d\bm{s}_2\\
&=: \tilde{K}_{n,11} + \tilde{K}_{n,12}.
\end{align*}
For $\tilde{K}_{n,11}$, we have
\begin{align}\label{eq:tilde-Kn11}
\tilde{K}_{n,11} &= n\int k_J^2(\bm{w})\left\{\eta^2(\bm{w}) + \sigma_\varepsilon^2(\bm{w})\right\}g(\bm{w})d\bm{w} \nonumber \\
&= n\left({\tilde{\psi}_J(\bm{z}) \over \|\tilde{\psi}_J(\bm{z})\|}\right)'\left(\int \left\{\eta^2(\bm{w}) + \sigma_\varepsilon^2(\bm{w})\right\}\tilde{\psi}_J(\bm{w})\tilde{\psi}_J(\bm{w})'g(\bm{w})d\bm{w}\right)\left({\tilde{\psi}_J(\bm{z}) \over \|\tilde{\psi}_J(\bm{z})\|}\right).
\end{align}
For $K_{n,12}$, we have
\begin{align*}
\tilde{K}_{n,12} &= n(n-1)\int k_J(\bm{w}_1)k_J(\bm{w}_2)\eta(\bm{w}_1)\eta(\bm{w}_2)\sigma_{\bm{e}}(A_n(\bm{w}_1 - \bm{w}_2))g(\bm{w}_1)g(\bm{w}_2)d\bm{w}_1 d\bm{w}_2\\
&= {n(n-1) \over A_n}\!\!\int_{\bar{R}_n}\!\!\!\sigma_{\bm{e}}(\bm{x})\!\!\left(\int \!\!k_J\!\!\left({\bm{x} \over A_n} \!+\! \bm{w}_2\right)\! k_J(\bm{w}_2)\eta \!\left({\bm{x} \over A_n} \!+\! \bm{w}_2\right)\eta(\bm{w}_2)g\!\left({\bm{x} \over A_n} \!+\! \bm{w}_2\right)\! g(\bm{w}_2)d\bm{w}_2\right)d\bm{x}
\end{align*}
where $\bar{R}_n = \{\bm{x} = \bm{x}_1 - \bm{x}_2: \bm{x}_1, \bm{x}_2 \in R_n\}$. Then as $n,J \to \infty$, we obtain
\begin{align}
\tilde{K}_{n,12} &= {n^2 \over A_n}\left(\int k_J^2(\bm{w})\eta^2(\bm{w})g^2(\bm{w})d\bm{w}\right)\left(\int \sigma_{\bm{e}}(\bm{x})d\bm{x}\right)(1 + o(1)) \nonumber \\
&= {n^2 \over A_n}\left({\tilde{\psi}_J(\bm{z}) \over \|\tilde{\psi}_J(\bm{z})\|}\right)'\left(\int \eta^2(\bm{w})\tilde{\psi}_J(\bm{w})\tilde{\psi}_J(\bm{w})'g^2(\bm{w})d\bm{w}\right)\left(\int \sigma_{\bm{e}}(\bm{x})d\bm{x}\right)\left({\tilde{\psi}_J(\bm{z}) \over \|\tilde{\psi}_J(\bm{z})\|}\right)(1 + o(1)). \label{eq:tilde-Kn12}
\end{align}
Then (\ref{eq:tilde-Kn11}) and (\ref{eq:tilde-Kn12}) yield (\ref{eq:tilde_K_n1-var}).

(Step 3) Now we show (\ref{eq:Kn-4}). Define $I_{n}(\bm{\ell}) \!=\! \{\bm{i} \! \in \! \mathbb{Z}^d \!: \bm{i} \!+\! (-1/2,1/2]^d \! \subset \! \Gamma_n(\bm{\ell}; \bm{\Delta}_0)\}$ for $\bm{\ell} \in L_{n1}$ and $\tilde{K}_n(\bm{i}) = \sum_{i=1}^{n}k_J(\bm{S}_1/A_n)(e_{n,i} + \varepsilon_{n,i})1\{\bm{S}_i \in [\bm{i} + (-1/2,1/2]^d] \cap R_n\}$. 

Observe that 
\begin{align*}
&\E[(k_1^{(\bm{\ell}; \bm{\Delta}_0)})^4] = \E\left[\left(\sum_{\bm{i} \in I_n(\bm{\ell})}\tilde{K}_n(\bm{i})\right)^4\right]\\ 
&= \sum_{\bm{i} \in I_n(\bm{\ell})}\E\left[\tilde{K}_n^4(\bm{i})\right] + \sum_{\bm{i}, \bm{j} \in I_n(\bm{\ell}), \bm{i} \neq \bm{j}}\E\left[\tilde{K}_n^3(\bm{i})\tilde{K}_n(\bm{j})\right] + \sum_{\bm{i}, \bm{j} \in I_n(\bm{\ell}), \bm{i} \neq \bm{j}}\E\left[\tilde{K}_n^2(\bm{i})\tilde{K}_n^2(\bm{j})\right]\\
&\quad + \sum_{\bm{i}, \bm{j}, \bm{k} \in I_n(\bm{\ell}), \bm{i} \neq \bm{j} \neq \bm{k}}\!\!\!\!\!\!\!\!\!\!\!\!\E \! \left[\tilde{K}_n^2(\bm{i})\tilde{K}_n(\bm{j})\tilde{K}_n(\bm{k})\right] + \sum_{\bm{i}, \bm{j}, \bm{k}, \bm{p} \in I_n(\bm{\ell}), \bm{i} \neq \bm{j} \neq \bm{k} \neq \bm{p}}\!\!\!\!\!\!\!\!\!\!\!\!\E \! \left[\tilde{K}_n(\bm{i})\tilde{K}_n(\bm{j})\tilde{K}_n(\bm{k})\tilde{K}_n(\bm{p})\right]\\
& =: Q_{n1} + Q_{n2} + Q_{n3} + Q_{n4} + Q_{n5}. 
\end{align*}
For $Q_{n1}$, we have
\begin{align*}
&\E[\tilde{K}_n^4(\bm{i})] \\
&= \sum_{j_1,j_2,j_3,j_4 =1}^{n}\E\left[\prod_{k=1}^{4}k_J\left({\bm{S}_{j_k} \over A_n}\right)1\{\bm{S}_{j_k} \in [\bm{i} + (-1/2,1/2]^d] \cap R_n\} (e_{n,j_k} + \varepsilon_{n,j_k})\right]\\
&\lesssim \sum_{j_1,j_2,j_3,j_4 =1}^{n}\!\!\!\!\!\!\!\E\left[\prod_{k=1}^{4}k_J\left({\bm{S}_{j_k} \over A_n}\right)1\{\bm{S}_{j_k}\!\! \in \! [\bm{i} \!+\! (-1/2,1/2]^d] \cap R_n\}\eta\!\left(\!{\bm{S}_{j_k} \over A_n}\!\right)e(\bm{S}_{j_k})\!\right]\\
&\quad + \sum_{j_1,j_2,j_3,j_4 =1}^{n}\!\!\!\!\!\! \E\left[\prod_{k=1}^{4}k_J\left({\bm{S}_{j_k} \over A_n}\right)1\{\bm{S}_{j_k} \!\! \in \!  [\bm{i} \!+\! (-1/2,1/2]^d] \!\cap\! R_n\}\sigma_{\varepsilon}\!\left(\!{\bm{S}_{j_k} \over A_n}\!\right)\varepsilon_{j_k}\!\right] =: Q_{n11} + Q_{n12}. 
\end{align*}
For $Q_{n11}$, we have
\begin{align*}
Q_{n11} &\lesssim n\E\left[\left|k_J\left({\bm{S}_1 \over A_n}\right)\right|^4 1\{\bm{S}_{1} \in [\bm{i} + (-1/2,1/2]^d] \cap R_n\}\eta^4(\bm{S}_{1}/A_n)\right] \\
&\quad + n^2\E\left[\left|k_J\left({\bm{S}_1 \over A_n}\right)\right|^3 1\{\bm{S}_{1} \in [\bm{i} + (-1/2,1/2]^d] \cap R_n\} \right. \\
&\quad \left.  \times \left|k_J\left({\bm{S}_{j_k} \over A_n}\right)\right|1\{\bm{S}_{2} \in [\bm{i} + (-1/2,1/2]^d] \cap R_n\} \eta^3(\bm{S}_{1}/A_n)\eta(\bm{S}_{2}/A_n)\right] \\
&\quad + n^2\E\left[\prod_{t=1}^{2}\left|k_J\left({\bm{S}_{t} \over A_n}\right)\right|^2 \!\! 1\{\bm{S}_t \in [\bm{i} + (-1/2,1/2]^d] \cap R_n\} \eta^2(\bm{S}_t/A_n)\right]\\
&\quad + n^3\E\left[\left|k_J\left({\bm{S}_1 \over A_n}\right)\right|^21\{\bm{S}_1 \in [\bm{i} + (-1/2,1/2]^d] \cap R_n\}\eta^2(\bm{S}_1/A_n) \right. \\
&\quad \left.  \times \prod_{t = 2}^{3}\left|k_J\left({\bm{S}_t \over A_n}\right)\right|1\{\bm{S}_t \in [\bm{i} + (-1/2,1/2]^d] \cap R_n\}\eta(\bm{S}_t/A_n)\right]\\
&\quad + n^4\E\left[\prod_{t=1}^{4}\left|k_J\left({\bm{S}_t \over A_n}\right)\right|1\{\bm{S}_t \in [\bm{i} + (-1/2,1/2]^d] \cap R_n\} \eta(\bm{S}_t/A_n)\right]\\
&=: Q_{n111} + Q_{n112} + Q_{n113} + Q_{n114}. 
\end{align*}
For $Q_{n111}$, we have
\begin{align*}
Q_{n111} &= nA_n^{-1}\int \left|k_J\left({\bm{s}\over A_n}\right)\right|^4 1\{\bm{s} \in [\bm{i} + (-1/2,1/2]^d] \cap R_n\}\eta^4(\bm{s}/A_n)g(\bm{s}/A_n)d\bm{s}\\
&\leq nA_n^{-1}\zeta_J^4\lambda_J^4 \int 1\{\bm{s} \in R_n\}\eta^4(\bm{s}/A_n)g(\bm{s}/A_n)d\bm{s} = O\left(nA_n^{-1}\zeta_J^4 \lambda_J^4\right). 
\end{align*}
Likewise, $Q_{n112} = O(n^2A_n^{-2}\zeta_J^4 \lambda_J^4)$, $Q_{n113} = O(n^3A_n^{-3}\zeta_J^4\lambda_J^4)$, and $Q_{n114} = O(n^4A_n^{-4}\zeta_J^4 \lambda_J^4)$. Then we have $Q_{n11} = O(n^4A_n^{-4}\zeta_J^4 \lambda_J^4)$. We can also show that $Q_{n12} = O(n^4A_n^{-4}\zeta_J^4 \lambda_J^4)$.
Therefore, we have
\begin{align}\label{ineq:Qn1}
Q_{n1} &\lesssim  [\![I_n(\bm{\ell})]\!]n^4A_n^{-4}\zeta_J^4 \lambda_J^4 \lesssim A_n^{(1)}(nA_n^{-1})^4\zeta_J^4 \lambda_J^4.
\end{align}
For $Q_{n2}$,  by the $\alpha$-mixing property of $\bm{e}$ and Proposition 2.5 in \cite{FaYa03}, we have
\begin{align}
Q_{n2} &\lesssim \sum_{k=1}^{\overline{A}_{n1}} \sum_{\bm{i}, \bm{j} \in I_n(\bm{\ell}), |\bm{i} - \bm{j}| = k}\alpha^{1-4/q}(\max\{k-d,0\};1)\E[|\tilde{K}_n(\bm{i})|^q]^{3/q}\E[|\tilde{K}_n(\bm{j})|^q]^{1/q} \nonumber \\
&\lesssim A_n^{(1)}(nA_n^{-1})^4\zeta_J^4 \lambda_J^4\left(1 + \sum_{k=1}^{\overline{A}_{n1}}k^{d-1}\alpha_1^{1-4/q}(k)\right).  \label{ineq:Qn2}
\end{align}
where $\overline{A}_{n1} = \max_{1 \leq j \leq d}A_{n1,j}$. Likewise, 
\begin{align}
Q_{n3} &\lesssim A_n^{(1)}(nA_n^{-1})^4\zeta_J^4 \lambda_J^4\left(1 + \sum_{k=1}^{\overline{A}_{n1}}k^{d-1}\alpha_1^{1-4/q}(k)\right).  \label{ineq:Qn3}
\end{align}

Now we evaluate $Q_{n4}$ and $Q_{n5}$. For distinct indices $\bm{i}, \bm{j}, \bm{k}, \bm{p} \in I_n(\bm{\ell})$, let 
\begin{align*}
d_1(\bm{i}, \bm{j}, \bm{k}) &= \max\{d(\{\bm{i}\}, \{\bm{j}, \bm{k}\}),d(\{\bm{k}\}, \{\bm{i}, \bm{j}\})\},\\
d_2(\bm{i}, \bm{j}, \bm{k}, \bm{p}) &= \max\{d(F, \{\bm{i}, \bm{j}, \bm{k}, \bm{p}\}): F \subset \{\bm{i}, \bm{j}, \bm{k}, \bm{p}\}, [\![ F]\!] = 1\},\\
d_3(\bm{i}, \bm{j}, \bm{k}, \bm{p}) &= \max\{d(F, \{\bm{i}, \bm{j}, \bm{k}, \bm{p}\}): F \subset \{\bm{i}, \bm{j}, \bm{k}, \bm{p}\}, [\![ F]\!] = 2\}. 
\end{align*}
Here, $d_1$ denotes the maximal gap in the set of integer-indices $\{\bm{i},\bm{j},\bm{k}\}$ from either $\bm{j}$ or $\bm{k}$ which corresponds to $\E\left[\tilde{K}_n^2(\bm{i})\tilde{K}_n(\bm{j})\tilde{K}_n(\bm{k})\right]$. Similarly, $d_2$ and $d_3$ are the maximal gap in the index set $\{\bm{i}, \bm{j}, \bm{k}, \bm{p}\}$ from any of its single index-subsets or two-index subsets, respectively. Applying the argument in the proof of Lemma 4.1 of \cite{La99}, for any given values $1 \leq d_{01}, d_{02}, d_{03} < [\![I_{n}(\bm{\ell})]\!]$, we have 
\begin{align}
&[\![\{(\bm{i}, \bm{j},\bm{k}) \in I_n^3(\bm{\ell}): \bm{i} \neq \bm{j} \neq \bm{k}\ \text{and}\ d_1(\bm{i},\bm{j},\bm{k})=d_{01}\}]\!] \lesssim d_{01}^{2d-1}[\![I_n(\bm{\ell})]\!],  \label{ineq:card-d01}\\
&[\![\{(\bm{i}, \bm{j},\bm{k}, \bm{p}) \in I_n^4(\bm{\ell}): \bm{i} \neq \bm{j} \neq \bm{k} \neq \bm{p},\ d_2(\bm{i},\bm{j},\bm{k}, \bm{p})=d_{02},\ \text{and}\ d_3(\bm{i},\bm{j},\bm{k}, \bm{p})=d_{03}\}]\!] \nonumber \\ 
&\quad \lesssim (d_{02} + d_{03})^{3d-1}[\![I_n(\bm{\ell})]\!].  \label{ineq:card-d023}
\end{align}
For $Q_{n4}$, by (\ref{ineq:card-d01}) and applying the same argument to show (\ref{ineq:Qn2}), we have
\begin{align}
Q_{n4} &\lesssim A_n^{(1)}\sum_{k=1}^{\overline{A}_{n1}} k^{2d-1}\alpha^{1-4/q}(\max\{k-d,0\};2)\E[|\tilde{K}_n(\bm{i})|^q]^{2/q}\E[|\tilde{K}_n(\bm{j})|^q]^{1/q}\E[|\tilde{K}_n(\bm{k})|^q]^{1/q} \nonumber \\
&\lesssim A_n^{(1)}(nA_n^{-1})^4\zeta_J^4 \lambda_J^4\left(1 + \sum_{k=1}^{\overline{A}_{n1}}k^{2d-1}\alpha_1^{1-4/q}(k)\right).  \label{ineq:Qn4}
\end{align}
Define
\begin{align*}
I_{n1}(\bm{\ell}) &= \{(\bm{i}, \bm{j},\bm{k}, \bm{p}) \in I_n^4(\bm{\ell}): \bm{i} \neq \bm{j} \neq \bm{k} \neq \bm{p},\ d_2(\bm{i},\bm{j},\bm{k}, \bm{p}) \geq d_3(\bm{i},\bm{j},\bm{k}, \bm{p})\},\\
I_{n2}(\bm{\ell}) &= \{(\bm{i}, \bm{j},\bm{k}, \bm{p}) \in I_n^4(\bm{\ell}): \bm{i} \neq \bm{j} \neq \bm{k} \neq \bm{p},\ d_2(\bm{i},\bm{j},\bm{k}, \bm{p}) < d_3(\bm{i},\bm{j},\bm{k}, \bm{p})\}.
\end{align*}
For $Q_{n5}$, by (\ref{ineq:card-d023}) and applying the same argument to show (\ref{ineq:Qn2}), we have
\begin{align}
Q_{n5} &= \sum_{(\bm{i}, \bm{j},\bm{k}, \bm{p}) \in I_{n1}(\bm{\ell})}\!\!\!\!\!\!\E\left[\tilde{K}_n(\bm{i})\tilde{K}_n(\bm{j})\tilde{K}_n(\bm{k})\tilde{K}_n(\bm{p})\right] + \sum_{(\bm{i}, \bm{j},\bm{k}, \bm{p}) \in I_{n2}(\bm{\ell})}\!\!\!\!\!\!\E\left[\tilde{K}_n(\bm{i})\tilde{K}_n(\bm{j})\tilde{K}_n(\bm{k})\tilde{K}_n(\bm{p})\right] \nonumber \\
&\lesssim A_n^{(1)}\sum_{k=1}^{\overline{A}_{n1}} k^{3d-1}\alpha^{1-4/q}(\max\{k-d,0\};3) \nonumber \\
&\quad \times \E[|\tilde{K}_n(\bm{i})|^q]^{1/q}\E[|\tilde{K}_n(\bm{j})|^q]^{1/q}\E[|\tilde{K}_n(\bm{k})|^q]^{1/q}\E[|\tilde{K}_n(\bm{p})|^q]^{1/q} \nonumber \\
&\quad + \left(\sum_{\bm{i}, \bm{j} \in I_n(\bm{\ell}), \bm{i} \neq \bm{j}}\left|\E[\tilde{K}_n(\bm{i})\tilde{K}_n(\bm{j})]\right|\right)^2 + A_n^{(1)}\sum_{k=1}^{\overline{A}_{n1}} k^{3d-1}\alpha^{1-4/q}(\max\{k-d,0\};2) \nonumber \\
&\quad \times \E[|\tilde{K}_n(\bm{i})|^q]^{1/q}\E[|\tilde{K}_n(\bm{j})|^q]^{1/q}\E[|\tilde{K}_n(\bm{k})|^q]^{1/q}\E[|\tilde{K}_n(\bm{p})|^q]^{1/q} \nonumber \\
&\lesssim A_n^{(1)}(\overline{A}_{n1})^d(nA_n^{-1})^4 \zeta_J^4 \lambda_J^4\left(1 + \sum_{k=1}^{\overline{A}_{n1}}k^{2d-1}\alpha_1^{1-4/q}(k)\right).  \label{ineq:Qn5}
\end{align}
Combining (\ref{ineq:Qn1}), (\ref{ineq:Qn2}), (\ref{ineq:Qn3}), (\ref{ineq:Qn4}), and (\ref{ineq:Qn5}), we have
\begin{align*}
&\sum_{\bm{\ell} \in L_{n1}}\E[\tilde{K}_n^4(\bm{\ell}; \bm{\Delta}_0)]\\ 
&= \sum_{\bm{\ell} \in L_{n1}}\E\left[\left(\sum_{\bm{i} \in I_n(\bm{\ell})}\tilde{K}_n(\bm{i})\right)^4\right] \lesssim [\![L_{n1}]\!]A_n^{(1)}(\overline{A}_{n1})^d(nA_n^{-1})^4\zeta_J^4 \lambda_J^4\left(1 + \sum_{k=1}^{\overline{A}_{n1}}k^{2d-1}\alpha_1^{1-4/q}(k)\right)\\
&\lesssim {A_n \over A_n^{(1)}} A_n^{(1)}(\overline{A}_{n1})^d(nA_n^{-1})^4 \zeta_J^4 \lambda_J^4\left(1 + \sum_{k=1}^{\overline{A}_{n1}}k^{2d-1}\alpha_1^{1-4/q}(k)\right) = o\left(n^4A_n^{-2}\right). 
\end{align*}

(Step 4) Now we show (\ref{eq:Kn12-var}) and (\ref{eq:Kn13-var}). Define 
\begin{align*}
J_n &= \{\bm{i} \in \mathbb{Z}^d: (\bm{i} + (-1/2,1/2]^d) \cap R_n \neq \emptyset\},\ J_{n1} = \cup_{\bm{\ell} \in L_{n1}}I_{n}(\bm{\ell}),\\
J_{n2} &= \{\bm{i} \in J_n \!:\! \bm{i} + (-1/2,1/2]^d \!\subset \! \Gamma_n(\bm{\ell};\bm{\Delta})\ \text{for some}\ \bm{\ell} \in L_{n1}, \bm{\Delta} \!\neq \! \bm{\Delta}_0\},\ J_{n3} = J_n \backslash (J_{n1} \!\cup \! J_{n2}). 
\end{align*}
Note that $[\![J_{n2}]\!] \lesssim (\overline{A}_{n1})^{d-1}\overline{A}_{n2}A_n(A_n^{(1)})^{-1}$ and $[\![J_{n3}]\!] \lesssim A_n^{(1)}\left({  \overline{A}_n \over \underline{A}_{n1}}\right)^{d-1}$. Then, applying the same argument to show (\ref{ineq:Qn2}), we have
\begin{align*}
\Var(K_{n,12}) &\lesssim [\![J_{n2}]\!](nA_n^{-1})^2\zeta_J^2\lambda_J^2\left(1 + \sum_{k=1}^{\overline{A}_{n1}}k^{d-1}\alpha_1^{1-2/q}(k)\right)\\
&\lesssim {(\overline{A}_{n1})^{d-1}\overline{A}_{n2} \over A_n^{(1)}}n^2A_n^{-1}\zeta_J^2\lambda_J^2\left(1 + \sum_{k=1}^{\overline{A}_{n1}}k^{d-1}\alpha_1^{1-2/q}(k)\right) = o\left( n^2A_n^{-1}\right). \\
\Var(K_{n,13}) &\lesssim [\![J_{n3}]\!](nA_n^{-1})^2 \zeta_J^2\lambda_J^2\left(1 + \sum_{k=1}^{\overline{A}_{n1}}k^{d-1}\alpha_1^{1-2/q}(k)\right)\\
&\lesssim\!  {A_n^{(1)} \over A_n}\left({\overline{A}_n \over \underline{A}_{n1}}\right)^{d-1}n^2A_n^{-1}\zeta_J^2\lambda_J^2 \!\left(\!\! 1 \!+\! \sum_{k=1}^{\overline{A}_{n1}}k^{d-1}\alpha_1^{1-2/q}(k) \!\! \right) = o\left( n^2A_n^{-1}\right).  
\end{align*}

(Step 5) Now we show (\ref{eq:Kn-var-diff}). By (\ref{eq:Kn12-var}) and (\ref{eq:Kn13-var}), we have for sufficiently large $n$, 
\begin{align*}
\E[K_{n,11}^2] &= \E[\{\tilde{K}_{n,1} - (K_{n,12} + K_{n,13})\}^2] \leq 2\left(\E[\tilde{K}_{n,1}^2] + \E[(\tilde{K}_{n,12} + \tilde{K}_{n,13})^2]\right) \leq 4\E[\tilde{K}_{n,1}^2]. 
\end{align*}
Thus, by (\ref{ineq:Ga12-dist}), (\ref{eq:Kn12-var}), and (\ref{eq:Kn13-var}), we have
\begin{align*}
&\left|\sum_{\bm{\ell} \in L_{n1}}\E[(k_1^{(\bm{\ell}; \bm{\Delta}_0)})^2] - \E[\tilde{K}_{n,1}^2]\right|\\
&\leq  \left|\sum_{\bm{\ell} \in L_{n1}}\E[(k_1^{(\bm{\ell}; \bm{\Delta}_0)})^2] - \E[K_{n,11}^2]\right| + 2\E[(K_{n,12} + K_{n,13})^2]^{1/2}\E[K_{n,11}^2]^{1/2} + \E[(K_{n,12} + K_{n,13})^2]\\
&\lesssim \left(A_n^{(1)}nA_n^{-1}\right)^2 \zeta_J^2\lambda_J^2 \!\!\sum_{\bm{\ell}_1 \neq \bm{\ell}_2} \!\! \alpha^{1-2/q}(\max\{|\bm{\ell}_1 \!-\! \bm{\ell}_2|-d, 0\}\underline{A}_{n3} \!+\! \underline{A}_{n2};A_n^{(1)}) \!+\! o\left(n^2A_n^{-1}\right)\\
&\lesssim \left(A_n^{(1)}nA_n^{-1}\right)^2\zeta_J^2\lambda_J^2 A_n (A_n^{(1)})^{-1}  \\ 
&\quad \times \left(\alpha^{1-2/q}(\underline{A}_{n2};A_n^{(1)}) + \sum_{k=1}^{\overline{A}_{n}/\underline{A}_{n1}}k^{d-1}\alpha^{1-2/q}(\max\{k-d, 0\}\underline{A}_{n3} + \underline{A}_{n2};A_n^{(1)})\right) + o\left(n^2A_n^{-1}\right)\\
& = o\left(n^2A_n^{-1}\right).
\end{align*}

\subsection{Proof of Proposition \ref{prp:var-est-trend}}

Note that $\hat{\Omega}_J(\bm{z}_1,\bm{z}_2) = \tilde{\psi}_J(\bm{z}_1)'\tilde{G}_J\tilde{\psi}_J(\bm{z}_2)$ where 
\begin{align*}
\tilde{G}_J &= {A_n \over n^2}\sum_{i,j=1}^n \left({\tilde{\Psi}'_{J,n}\tilde{\Psi}_{J,n} \over n} + \varsigma_{J,n} \check{\Psi}_J^{-1}\right)^{-1}\tilde{\psi}_J\left({\bm{S}_i \over A_n}\right)\tilde{\psi}_J\left({\bm{S}_j \over A_n}\right)'\left({\tilde{\Psi}'_{J,n}\tilde{\Psi}_{J,n} \over n} + \varsigma_{J,n}\check{\Psi}_J^{-1} \right)^{-1}\\
&\quad \quad \times \left(Y(\bm{S}_i) - \hat{m}\left({\bm{S}_i \over A_n}\right)\!\right)\left(Y(\bm{S}_j) - \hat{m}\left({\bm{S}_j \over A_n}\right)\!\right)\bar{K}_b(\bm{S}_i - \bm{S}_j).
\end{align*}
Then it suffices to show that as $n,J \to \infty$, 
\begin{align}\label{eq:W_nJ-conv}
W_{n,J}(\bm{z}_1,\bm{z}_2) &:= \left({\tilde{\psi}_J(\bm{z}_1) \over \|\tilde{\psi}_J(\bm{z}_1)\|}\right)'\tilde{G}_J\left({\tilde{\psi}_J(\bm{z}_2) \over \|\tilde{\psi}_J(\bm{z}_2)\|}\right) = \left({\tilde{\psi}_J(\bm{z}_1) \over \|\tilde{\psi}_J(\bm{z}_1)\|}\right)'G_J\left({\tilde{\psi}_J(\bm{z}_2) \over \|\tilde{\psi}_J(\bm{z}_2)\|}\right) + o_p(1). 
\end{align}
Now we restrict our attention to the case $\bm{z}_1 = \bm{z}_2 = \bm{z}$. The proofs for other cases are similar. Define $W_{n,J}(\bm{z}):= W_{n,J}(\bm{z},\bm{z})$. Applying Theorem \ref{thm:unif-rate1}, we have 
\begin{align*}
W_{n,J}(\bm{z}) 
&= {A_n \over n^2}\sum_{i,j=1}^{n}k_J\left({\bm{S}_i \over A_n}\right)k_J\left({\bm{S}_j \over A_n}\right)\bar{K}_b(\bm{S}_i - \bm{S}_j)(e_{n,i} + \varepsilon_{n,i})(e_{n,j} + \varepsilon_{n,j}) + o_p(1)\\
&= {A_n \over n^2}\sum_{i,j=1}^{n}k_J\left({\bm{S}_i \over A_n}\right)k_J\left({\bm{S}_j \over A_n}\right)\bar{K}_b(\bm{S}_i - \bm{S}_j)e_{n,i}e_{n,j}\\
&\quad + {2A_n \over n^2}\sum_{i,j=1}^{n}k_J\left({\bm{S}_i \over A_n}\right)k_J\left({\bm{S}_j \over A_n}\right)\bar{K}_b(\bm{S}_i - \bm{S}_j)e_{n,i}\varepsilon_{n,j}\\
&\quad + {A_n \over n^2}\sum_{i,j=1}^{n}k_J\!\left({\bm{S}_i \over A_n}\right)k_J\!\left({\bm{S}_j \over A_n}\right)\bar{K}_b(\bm{S}_i - \bm{S}_j)\varepsilon_{n,i}\varepsilon_{n,j} + o_p(1) =: W_{n,1} + W_{n,2} + W_{n,3} + o_p(1).
\end{align*}

For $W_{n,3}$, observe that 
\begin{align*}
W_{n,3} &= {A_n \over n}\sum_{i=1}^{n}k_J^2\left({\bm{S}_i \over A_n}\right)\varepsilon_{n,i}^2 + {A_n \over n^2}\sum_{i \neq j}^{n}k_J\left({\bm{S}_i \over A_n}\right)k_J\left({\bm{S}_j \over A_n}\right)\bar{K}_b(\bm{S}_i - \bm{S}_j)\varepsilon_{n,i}\varepsilon_{n,j} =: W_{n,31} + W_{n,32}. 
\end{align*}

For $W_{n,31}$, we have 
\begin{align*}
\E[W_{n,31}] &= {A_n \over n}\int k_J^2(\bm{x})\sigma_{\varepsilon}^2(\bm{x}/A_n)A_n^{-1}g(\bm{x}/A_n)d\bm{x} = {A_n \over n}\int k_J^2(\bm{v})\sigma_{\varepsilon}^2(\bm{v} )g(\bm{v})d\bm{v}\\
&= \kappa \left(\tilde{\psi}_J(\bm{z}) \over \|\tilde{\psi}_J(\bm{z})\|\right)'\left(\int \sigma_{\varepsilon}^2(\bm{v})\tilde{\psi}_J(\bm{v})\tilde{\psi}_J(\bm{v})'g(\bm{v})d\bm{v}\right)\left(\tilde{\psi}_J(\bm{z}) \over \|\tilde{\psi}_J(\bm{z})\|\right)'(1 + o(1)),\\ 
\Var(W_{n,31}) &= \left({A_n \over n^2}\right)^2n\Var(k_J^2(\bm{S}_1/A_n)\varepsilon_{n,1}^2) \leq \left({A_n \over n^2}\right)^2nE[k_J^4(\bm{S}_1/A_n)\varepsilon_{n,1}^4] \lesssim {A_n^2 \over n^2}{J^2 \over n} = o(1). 
\end{align*}
Then we have 
\begin{align*}
W_{n,31} = \kappa \left(\tilde{\psi}_J(\bm{z}) \over \|\tilde{\psi}_J(\bm{z})\|\right)'\left(\int \sigma_{\varepsilon}^2(\bm{v})\tilde{\psi}_J(\bm{v})\tilde{\psi}_J(\bm{v})'g(\bm{v})d\bm{v}\right)\left(\tilde{\psi}_J(\bm{z}) \over \|\tilde{\psi}_J(\bm{z})\|\right)' + o_p(1).
\end{align*}

For $W_{n,32}$, applying similar arguments in the proof of Theorem \ref{thm:asy-norm1}, 
we have $\E[W_{n,32}] = 0$ and 
\begin{align}
&\left({A_n \over n^2}\right)^{-2}\E[W_{n,32}^2] \nonumber \\
&= \sum_{i \neq j, k \neq \ell}\!\!\!\!\!\E\!\left[k_J(\bm{S}_i/A_n)k_J(\bm{S}_j/A_n)\bar{K}_b(\bm{S}_i - \bm{S}_j)k_J(\bm{S}_k/A_n)k_J(\bm{S}_\ell/A_n)\bar{K}_b(\bm{S}_k - \bm{S}_\ell)\varepsilon_{n,i}\varepsilon_{n,j}\varepsilon_{n,k}\varepsilon_{n,\ell}\right] \nonumber \\
&= \sum_{i \neq j}\E\left[k_J^2(\bm{S}_i/A_n)k_J^2(\bm{S}_j/A_n)\bar{K}_b^2(\bm{S}_i - \bm{S}_j)\varepsilon_{n,i}^2\varepsilon_{n,j}^2\right] \nonumber \\
&= n(n-1)\int_{R_0^2} \bar{K}_b(A_n(\bm{z}_1 - \bm{z}_2) )k_J^2(\bm{z}_1)k_J^2(\bm{z}_2)\sigma_{\varepsilon}^2(\bm{z}_1)\sigma_{\varepsilon}^2(\bm{z}_2)A_n^{-2}g(\bm{z}_1 )g(\bm{z}_2 )d\bm{z}_1 d\bm{z}_2 \nonumber \\
&= n(n-1)A_n^{-1}b_1 \dots b_d\int_{\bar{R}_n/\bm{b}}\bar{K}^2(\bm{v})\left(\int k_J^2\left(\bm{z}_2 + {\bm{v} \circ \bm{b} \over A_n}\right)k_J^2(\bm{z}_2) \right. \nonumber \\
&\left. \quad \times \sigma_{\varepsilon}^2\left(\bm{z}_2 + {\bm{v} \circ \bm{b} \over A_n}\right)\sigma_{\varepsilon}^2\left(\bm{z}_2 \right)g\left(\bm{z}_2  + {\bm{v} \circ \bm{b} \over A_n}\right)g\left(\bm{z}_2\right)d\bm{z}_2 \right)d\bm{v} \lesssim n(n-1)A_n^{-1}b_1 \dots b_d J^2. \label{Vn32-new}
\end{align}
Then we have $\E[W_{n,32}] = O\left({A_n \over n}{b_1 \dots b_d J^2 \over n}\right) = o(1)$ and this yields $W_{n,32} = o_p(1)$. The results on $W_{n,31}$ and $W_{n,32}$ yield
\begin{align}\label{Vn3-new}
W_{n,3} &\stackrel{p}{\to} \kappa \left(\tilde{\psi}_J(\bm{z}) \over \|\tilde{\psi}_J(\bm{z})\|\right)'\left(\int \sigma_{\varepsilon}^2(\bm{v})\tilde{\psi}_J(\bm{v})\tilde{\psi}_J(\bm{v})'g(\bm{v})d\bm{v}\right)\left(\tilde{\psi}_J(\bm{z}) \over \|\tilde{\psi}_J(\bm{z})\|\right)'.
\end{align}

For $W_{n,2}$, observe that
\begin{align*}
&\left({2A_n \over n^2}\right)^{-2}\E_{\cdot|\bm{S}}[W_{n,2}^2]\\ 
&= \sum_{i=1}^{n}k_J^4(\bm{S}_i/A_n)\eta^2(\bm{S}_i/A_n)\sigma_{\varepsilon}^2(\bm{S}_j/A_n)\\
&\quad + \sum_{i\neq j}^{n}k_J^2(\bm{S}_i/A_n)k_J^2(\bm{S}_j/A_n)\bar{K}_b(\bm{S}_i - \bm{S}_j)\eta(\bm{S}_i/A_n)\eta(\bm{S}_j/A_n)\sigma_{\varepsilon}^2(\bm{S}_j/A_n)\\
&\quad + \sum_{i \neq j}^{n}k_J(\bm{S}_i/A_n)k_J^3(\bm{S}_j/A_n)\bar{K}_b(\bm{S}_i - \bm{S}_j)\eta(\bm{S}_i/A_n)\eta(\bm{S}_j/A_n)\sigma_{\bm{e}}(\bm{S}_i - \bm{S}_j)\sigma_{\varepsilon}^2(\bm{S}_j/A_n)\\
&\quad + \sum_{i \neq j \neq  \ell}^{n}k_J(\bm{S}_i/A_n)k_J^2(\bm{S}_j/A_n)k_J(\bm{S}_\ell)\bar{K}_b(\bm{S}_i - \bm{S}_j)\bar{K}_b(\bm{S}_\ell - \bm{S}_j)\\
&\quad \quad \times \eta(\bm{S}_i/A_n)\eta(\bm{S}_j/A_n)\sigma_{\bm{e}}(\bm{S}_i - \bm{S}_\ell)\sigma_{\varepsilon}^2(\bm{S}_j/A_n) =: W_{n,21} + W_{n,22} + W_{n,23} + W_{n,24}
\end{align*}
where $\E_{\cdot|\bm{S}}[\cdot]$ denote the conditional expectation given the $\sigma$-field generated by $\{\bm{S}_i\}_{i \geq 1}$. For $W_{n,21}$, we have $\E[W_{n,21}] = O(nJ^2)$ and this yields $\left({2A_n \over n^2}\right)^2\!\E[W_{n,21}] = o(1)$. For $W_{n,22}$, applying similar arguments to show (\ref{Vn32-new}), we have $\E[W_{n,22}] \lesssim n^2A_n^{-1}b_1 \dots b_dJ^2$. This yields $\left({2A_n \over n^2}\right)^2\E[W_{n,22}] = o(1)$. For $W_{n,23}$, applying similar arguments in the proof of Theorem \ref{thm:asy-norm1}, we have
\begin{align*}
\E[W_{n,23}] &\lesssim \sum_{i \neq j}^{n}\E\left[k_J(\bm{S}_i/A_n)k_J^3(\bm{S}_j/A_n)\eta(\bm{X}_i/A_n)\eta(\bm{X}_j/A_n)\sigma_{\bm{e}}(\bm{X}_i - \bm{X}_j)\sigma_{\varepsilon}^2(\bm{X}_j/A_n)\right]\\
&= O\left(n^2A_n^{-1}J^2 \right).
\end{align*}
This yields $\left({2A_n \over n^2}\right)^2\E[W_{n,23}] = O\left({A_n \over  n}{J^2 \over n}\right) = o(1)$.

For $W_{n,24}$, applying similar arguments in the proof of Theorem \ref{thm:asy-norm1}, we have
\begin{align*}
\E[W_{n,24}] 
&= n(n-1)(n-2)\int_{R_0^3} \bar{K}_b(A_n(\bm{z}_1 - \bm{z}_2))\bar{K}_b(A_n(\bm{z}_3 - \bm{z}_2))\\
&\quad \times \sigma_{\bm{e}}(A_n(\bm{z}_1 - \bm{z}_3))k_J(\bm{z}_1)k_J^2(\bm{z}_2)k_J(\bm{z}_3)\eta(\bm{z}_1)\eta(\bm{z}_2)\sigma_{\varepsilon}^2\left(\bm{z}_2 \right) \\
&\quad \times A_n^{-3}g(\bm{z}_1 )g(\bm{z}_2 )g(\bm{z}_3 )d\bm{z}_1 d\bm{z}_2d\bm{z}_3\\
&= n(n-1)(n-2)A_n^{-1}\int_{R_0}k_J^2(\bm{z}_2)\sigma_{\varepsilon}^2(\bm{z}_2 )g(\bm{z}_2)\\
&\quad \left\{\int_{\bar{R}_n}\sigma_{\bm{e}}(\bm{v})\left(\int \bar{K}\left({\bm{v} + A_n(\bm{z}_3-\bm{z}_2) \over \bm{b}}\right)\bar{K}\left({A_n(\bm{z}_3-\bm{z}_2)  \over \bm{b}}\right) \right. \right.  \\
&\left. \left .\quad \times k_J\!\!\left(\!\bm{z}_3 + {\bm{v}  \over A_n}\!\right)\!k_J(\bm{z}_3)\eta \! \left(\!\bm{z}_3\!  + {\bm{v} \over A_n}\!\right)\!\eta\left(\bm{z}_3\! \right)\!g\!\left(\bm{z}_3\! + {\bm{v}  \over A_n}\right)\!g\left(\bm{z}_3\right)d\bm{z}_3\! \right)\!d\bm{v}\! \right\}\! d\bm{z}_2\\
&= n(n-1)(n-2)A_n^{-2}b_1 \dots b_d\int_{R_0}k_J^2(\bm{z}_2)\sigma_{\varepsilon}^2(\bm{z}_2)g(\bm{z}_2)\\
&\quad \left\{\int_{\bar{R}_n}\sigma_{\bm{e}}(\bm{v})\left(\int \bar{K}\left({\bm{v} \over \bm{b}} + \bm{w} - {A_n\bm{z}_2  \over \bm{b}}\right)\bar{K}\left(\bm{w} - {A_n\bm{z}_2  \over \bm{b}}\right) \right. \right.  \\
&\left. \left .\quad \times k_J\!\!\left(\!{\bm{w}\circ \bm{b} \over A_n} + {\bm{v}  \over A_n}\!\right)\!k_J\!\left({\bm{w}\! \circ\! \bm{b} \over A_n}\right)\!\eta \! \left(\!{\bm{w}\! \circ \! \bm{b} \over A_n} + {\bm{v} \over A_n}\!\right)\!\eta\!\left({\bm{w} \! \circ \! \bm{b} \over A_n}\right) \right. \right. \\
&\left. \left. \quad \times g\!\left({\bm{w} \! \circ \! \bm{b} \over A_n} + {\bm{v}  \over A_n}\right)\!g\! \left({\bm{w} \! \circ \! \bm{b} \over A_n}\right)\!d\bm{w}\! \right)\!d\bm{v}\! \right\}\! d\bm{z}_2 = O\left(n^3A_n^{-2}b_1 \dots b_d J^2\right).
\end{align*}
This yields $\left({2A_n \over n^2}\right)^2\E[W_{n,24}] = O\left({b_1 \dots b_d J^2 \over n}\right) = o(1)$. The results on $W_{n,21}$ $W_{n,22}$, $W_{n,23}$, and $W_{n,24}$ yield
\begin{align}\label{Vn2-new}
W_{n,2} &\stackrel{p}{\to} 0.
\end{align}

For $W_{n,1}$, applying similar arguments to show (\ref{Vn32-new}), we have
\begin{align*}
\left({A_n \over n^2}\right)^{-1}\E[W_{n,1}] 
&= n\int k_J^2(\bm{v})\eta^2(\bm{v})g(\bm{v})d\bm{v} + n(n-1)A_n^{-1}\int_{\bar{R}_n}\!\!\! \!\sigma_{\bm{e}}(\bm{v})\bar{K}_b(\bm{v}) \left(\int\!\!k_J\!\! \left(\!\bm{z}_2  \!+\! {\bm{v}\over A_n}\!\right)\!\! k_J(\bm{z}_2)  \right. \\
&\left. \quad  \times \eta\! \left(\!\bm{z}_2  \!+\! {\bm{v} \over A_n}\!\right)\!\eta(\bm{z}_2 )g\! \left(\! \bm{z}_2 \!+\! {\bm{v} \over A_n}\!\right)g(\bm{z}_2)d\bm{z}_2 \!\right)d\bm{v} =: W_{n,11} + W_{n,12}. 
\end{align*}

For $W_{n,12}$, we have
\begin{align*}
W_{n,12} &= {n(n-1) \over A_n}\int_{\bar{R}_n}\!\! \!\sigma_{\bm{e}}(\bm{v}) \left(\int  k_J\!\! \left(\!\bm{z}_2  \!+\! {\bm{v} \over A_n}\!\right)\!\! k_J(\bm{z}_2) \eta\! \left(\!\bm{z}_2  \!+\! {\bm{v} \over A_n}\!\right)\!\eta(\bm{z}_2 )g\! \left(\! \bm{z}_2  \!+\! {\bm{v} \over A_n}\!\right)g(\bm{z}_2 )d\bm{z}_2 \!\right)d\bm{v}\\
&\quad + {n(n-1) \over A_n}\int_{\bar{R}_n}\!\! \!\sigma_{\bm{e}}(\bm{v})(\bar{K}_b(\bm{v})-1) \left(\int k_J\!\! \left(\!\bm{z}_2  \!+\! {\bm{v} \over A_n}\!\right)\!\! k_J(\bm{z}_2)  \right. \\
&\left. \quad  \times \eta\! \left(\!\bm{z}_2 \!+\! {\bm{v} \over A_n}\!\right)\!\eta(\bm{z}_2)g\! \left(\! \bm{z}_2 \!+\! {\bm{v} \over A_n}\!\right)g(\bm{z}_2)d\bm{z}_2 \!\right)d\bm{v} =: W_{n,121} + W_{n,122}. 
\end{align*}

For $W_{n,121}$, from the proof of Theorem \ref{thm:asy-norm1}, we have
\begin{align*}
W_{n,121} &= {n^2 \over A_n}\left({\tilde{\psi}_J(\bm{z}) \over \|\tilde{\psi}_J(\bm{z})\|}\right)'\left(\int \eta^2(\bm{v})\tilde{\psi}_J(\bm{v})\tilde{\psi}_J(\bm{v})'g^2(\bm{v})d\bm{v}\right)\left(\int \sigma_{\bm{e}}(\bm{x})d\bm{x}\right)\left({\tilde{\psi}_J(\bm{z}) \over \|\tilde{\psi}_J(\bm{z})\|}\right)(1 + o(1)).
\end{align*}

For $W_{n,122}$, observe that  for any $M>0$, 
\begin{align*}
W_{n,122} 
&= n(n-1)A_n^{-1}\int_{\bar{R}_n \cap \{\|\bm{v}\| \leq M\}}\!\!\!\!\!\!\!\!\!\!\! \!\!\!\!\!\sigma_{\bm{e}}(\bm{v})(\bar{K}_b(\bm{v})-1) \left(\int \!\!k_J\!\! \left(\!\bm{z}_2  \!+\! {\bm{v}\over A_n}\!\right)\!\! k_J(\bm{z}_2) \eta\! \left(\!\bm{z}_2  \!+\! {\bm{v} \over A_n}\!\right)\!\eta(\bm{z}_2) \right. \\
&\left. \quad \times g\! \left(\! \bm{z}_2  \!+\! {\bm{v} \over A_n}\!\right)g(\bm{z}_2 )d\bm{z}_2 \!\right)d\bm{v}\\
& + n(n-1)A_n^{-1}
\int_{\bar{R}_n \cap \{\|\bm{v}\| > M\}}\!\!\!\!\!\!\!\!\!\!\! \!\!\!\!\!\sigma_{\bm{e}}(\bm{v})(\bar{K}_b(\bm{v})-1) \left(\int \!\!k_J\!\! \left(\!\bm{z}_2  \!+\! {\bm{v} \over A_n}\!\right)\!\! k_J(\bm{z}_2)\eta\! \left(\!\bm{z}_2 \!+\! {\bm{v} \over A_n}\!\right)\!\eta(\bm{z}_2 )  \right. \\
&\left. \quad  \times g\! \left(\! \bm{z}_2  \!+\! {\bm{v} \over A_n}\!\right)g(\bm{z}_2 )d\bm{z}_2 \!\right)d\bm{v} =: W_{n,1221} + W_{n,1222}. 
\end{align*} 
Observe that
\begin{align*}
|W_{n,1221}| &\lesssim n^2A_n^{-1}{M \over \min_{1 \leq j \leq d}b_j},\ |W_{n,1222}| \lesssim n^2A_n^{-1}\int_{\|\bm{v}\|>M}|\sigma_{\bm{e}}(\bm{v})|d\bm{v}. 
\end{align*}
Then by taking $M = \min_{1 \leq j \leq d}b_j^{1/2}$, we have 
\begin{align*}
W_{n,1221} &= O\left(n^2A_n^{-1}(\min_{1 \leq j \leq d}b_j)^{-1/2}\right),\ W_{n,1222} = o\left(n^2A_n^{-1}\right). 
\end{align*}
The results on $W_{n,121}$, $W_{n,1221}$, and $W_{n,1222}$ yield
\[
W_{n,12} = {n^2 \over A_n}\left({\tilde{\psi}_J(\bm{z}) \over \|\tilde{\psi}_J(\bm{z})\|}\right)'\left(\int \eta^2(\bm{v})\tilde{\psi}_J(\bm{v})\tilde{\psi}_J(\bm{v})'g^2(\bm{v})d\bm{v}\right)\left(\int \sigma_{\bm{e}}(\bm{x})d\bm{x}\right)\left({\tilde{\psi}_J(\bm{z}) \over \|\tilde{\psi}_J(\bm{z})\|}\right)(1 + o(1)).
\]
This yields
\begin{align}\label{Vn1-new}
\E[W_{n,1}] &= \left({\tilde{\psi}_J(\bm{z}) \over \|\tilde{\psi}_J(\bm{z})\|}\right)'\left\{\kappa \int \eta^2(\bm{v})\tilde{\psi}_J(\bm{v})\tilde{\psi}_J(\bm{v})'g(\bm{v})d\bm{v} \right. \nonumber \\
&\left. \quad + \left(\int \eta^2(\bm{v})\tilde{\psi}_J(\bm{v})\tilde{\psi}_J(\bm{v})'g^2(\bm{v})d\bm{v}\right)\left(\int \sigma_{\bm{e}}(\bm{x})d\bm{x}\right)\right\}\left({\tilde{\psi}_J(\bm{z}) \over \|\tilde{\psi}_J(\bm{z})\|}\right)(1 + o(1)).
\end{align}

Combining (\ref{Vn3-new}), (\ref{Vn2-new}), (\ref{Vn1-new}) and the results in the proof of Theorem \ref{thm:asy-norm1}, we obtain (\ref{eq:W_nJ-conv})
and this yields the desired result.

\section{Proof for Section \ref{Appendix:example}}\label{Appendix:ex-proof}

\subsection{Proof of Proposition \ref{prp:LevyMA-ex}}

Note that
\begin{align*}
{\tilde{\psi}_J(\bm{z})'\tilde{\Psi}_{J,n}'e_n \over n} = {1 \over n}\!\sum_{i=1}^{n}\tilde{\psi}_J(z)'\tilde{\psi}_J\!\left({\bm{S}_i \over A_n}\right)\eta\!\left({\bm{S}_i \over A_n}\right)\!e_{1,m_n}\!(\bm{S}_i) + {1 \over n}\!\sum_{i=1}^{n}\tilde{\psi}_J(z)'\tilde{\psi}_J\!\left({\bm{S}_i \over A_n}\right)\eta\!\left({\bm{S}_i \over A_n}\right)e_{2,m_n}\!(\bm{S}_i).
\end{align*}
To verify the asymptotic negligibility of the random field $\bm{e}_{2,m_n}$, it suffices to show that 
\begin{align}\label{eq:e2m-asy-neg}
\max_{1 \leq i \leq n}|e_{2,m_{n}}(\bm{S}_i)| = O_p\left(\exp\left(-{r_1 \over 8}n^{{\mathfrak{c}_0\mathfrak{c}_1\mathfrak{c}_2 \over 2d}}\right)\right),\ n \to \infty.
\end{align}
Indeed, (\ref{eq:e2m-asy-neg}) yields that
\begin{align*}
\sup_{\bm{z} \in R_0}\left|{1 \over n}\sum_{i=1}^{n}\tilde{\psi}_J(z)'\tilde{\psi}_J\left({\bm{S}_i \over A_n}\right)\eta\left({\bm{S}_i \over A_n}\right)e_{2,m_n}(\bm{S}_i)\right| &\leq \sup_{\bm{z} \in R_0}\|\tilde{\psi}_J(\bm{z})\|{1 \over n}\sum_{i=1}^{n}\left\|\tilde{\psi}_J\left({\bm{S}_i \over A_n}\right)\right\|\left|\eta\left({\bm{S}_i \over A_n}\right)\right||e_{2,m_n}(\bm{S}_i)|\\
&\lesssim (\zeta_J \lambda_J)^2\max_{1 \leq i \leq n}|e_{2,m_{n}}(\bm{S}_i)| \lesssim \exp\left(-{r_1 \over 16}n^{{\mathfrak{c}_0 \mathfrak{c}_1 \mathfrak{c}_2 \over 2d}}\right),
\end{align*}
which implies that $\bm{e}_{2,m_n}$ is asymptotically negligible. Hence we can replace $\bm{e}$ with $\bm{e}_{1,m_n}$ in the results in Section \ref{Sec:reg-trend}. 

Now we show (\ref{eq:e2m-asy-neg}). Note that under Condition (a), we have $E[|e(\bm{0})|^9]<\infty$ since $\bm{e}$ is Gaussian. Under Condition (b), we also have $E[|L([0,1]^d)|^9] < \infty$ since $\int_{|x|>1}|x|^9\nu_0(x)dx <\infty$ (cf. Theorem 25.3 in \cite{Sa99}). Define $\sigma_{\bm{e}_{1,m_n}}(\bm{x}) = E[e_{1,m_n}(\bm{0})e_{1,m_n}(\bm{x})]$. 
Then we have that 
\begin{align*}
E[|e_{1,m_n}(\bm{0})|^9] & \leq E[|e(\bm{0})|^9] \lesssim \int e^{-9r_1\|\bm{u}\|}d\bm{u}<\infty, \\ 
|\sigma_{\bm{e}_{1,m_n}}\!(\bm{x})| & \lesssim \!|E[e(\bm{0})e(\bm{x})]| \!\lesssim \!\int \!\!\!e^{-r_1\|\bm{u}\|}e^{-r_1\|\bm{x}-\bm{u}\|}d\bm{u} \leq \!\int \!\!\!e^{-r_1\|\bm{u}\|}e^{-{r_1 \over 2}(\|\bm{x}\|-\|\bm{u}\|)}d\bm{u} \lesssim  e^{-{r_1 \over 2}\|\bm{x}\|}. 
\end{align*}
The latter implies that $\int |\sigma_{\bm{e}_{1,m_n}}(\bm{v})|d\bm{v}<\infty$. Likewise, 
\begin{align*}
E[(e_{2,m_n}(\bm{0}))^4] &\lesssim \int_{\mathbb{R}^{d}}\!\!\! e^{-4r_1\|\bm{u}\|}\left(1- \mathfrak{t}\left(\|\bm{u}\| : m_{n}\right)\right)^{4}d\bm{u} \lesssim  \int_{\|\bm{u}\| \geq m_{n}/4}\!\!\!\!\!\!\!\!e^{-4r_1\|\bm{u}\|}\left|1 + {4 \over m_{n}}\left(\|\bm{u}\| - {m_{n} \over 2}\right)\right|^{4}d\bm{u}\\
& \lesssim \int_{\|\bm{u}\| \geq m_{n}/4}\!\!\!\!\!\!\!\!e^{-4r_1\|\bm{u}\|}\left|1 + {4\|\bm{u}\| \over m_{n}}\right|^{4}d\bm{u} \leq 2^{q-1}  \int_{\|\bm{u}\| \geq m_{n}/4}\!\!\!\!\!\!\!\!e^{-4r_1\|\bm{u}\|}\left(1 + {4^{4}\|\bm{u}\|^{4} \over m_{n}^{4}}\right)d\bm{u}\\
&\lesssim  \int_{m_{n}/4}^{\infty}e^{-4r_1t}\left(1 + {4^{4}t^{4} \over m_{n}^{4}}\right)t^{d-1}dt \lesssim  m_{n}^{d-1}e^{-r_1m_{n}}.
\end{align*}
By Markov's inequality and Lemma 2.2.2 in \cite{vaWe96}, we have 
\begin{align*}
\Prob\left(\max_{1 \leq i \leq n}\left|e_{2,m_{n}}(\bm{S}_i)\right| > \varrho \right) &\leq \varrho^{-1}\E\!\left[\max_{1 \leq i \leq n}\left|e_{2,m_{n}}(\bm{S}_i)\right|\right] \leq \varrho^{-1}n^{1/4}\max_{1 \leq i \leq n}\left(\!\E\left[\left|e_{2,m_n}(\bm{0})\right|^{4}\right]\right)^{1/4}\\
&\lesssim \varrho^{-1}n^{1/4}m_n^{(d-1)/4}e^{-r_1m_n/4}. 
\end{align*}
Letting $\varrho = \exp\left(-{r_1 \over 8}n^{{\mathfrak{c}_0 \mathfrak{c}_1 \mathfrak{c}_2 \over 2d}}\right)$, we have (\ref{eq:e2m-asy-neg}). 

Next, we verify that the random field $\bm{e}_{1,m_n}$ satisfies our regularity conditions. Let $\alpha_{\bm{e}_1}(a;b)$ be the $\alpha$-mixing coefficients of $\bm{e}_{1,m_n}$. Note that $\alpha_{\bm{e}_1}(a;b) \leq \alpha(a;b)$. Since $\bm{e}_{1,m_n}$ is $m_n$-dependent, we have $\alpha_2(\underline{A}_{n2}) = 0$, which yields
\begin{align}
& A_n (A_n^{(1)})^{-1}\alpha_1(\underline{A}_{n2})\alpha_2(A_n) = 0, \label{eq:Ass26-1}\\
&A_n^{(1)}(\zeta_J \lambda_J)^2 \left(\alpha_1^{1-2/q}(\underline{A}_{n2}) + \sum_{k=\underline{A}_{n1}}^{\infty}k^{d-1}\alpha_1^{1-2/q}(k)\right)\alpha_2^{1-2/q}(A_n^{(1)}) = 0. \label{eq:Ass26-4} 
\end{align}
Moreover, 
\begin{align}
A_n^{-1}(\bar{A}_{n1})^d (\zeta_J \lambda_J)^4\sum_{k=1}^{\overline{A}_{n1}}k^{2d-1}\alpha_1^{1-4/q}(k) 
& \lesssim A_n^{-1}(\bar{A}_{n1})^d J^2\sum_{k=1}^{m_n}k^{2d-1} \leq A_n^{-1}(\bar{A}_{n1})^d J^2m_n^{2d} \nonumber \\
&\lesssim n^{-\mathfrak{c}_0\left\{1-\mathfrak{c}_1(1+\mathfrak{c}_2)\right\}}J^2 = o(1). \label{eq:Ass26-2}
\end{align}
\begin{align}
\left({(\overline{A}_{n1})^{d-1}\overline{A}_{n2} \over A_n^{(1)}} + {A_n^{(1)} \over A_n}\left({\overline{A}_n \over \underbar{A}_{n1}}\right)^{d-1}\right)(\zeta_J\lambda_J)^2\sum_{k=1}^{\overline{A}_{n1}}k^{d-1}\alpha_1^{1-2/q}(k) &\lesssim \left\{ n^{{\mathfrak{c}_0 \mathfrak{c}_1 \mathfrak{c}_2 \over d} - {\mathfrak{c}_0 \mathfrak{c}_1 \over d}} + n^{{\mathfrak{c}_0 \mathfrak{c}_1 \over d} - {\mathfrak{c}_0 \over d}}\right\}Jm_n^{d} \nonumber \\ 
&\lesssim n^{-\mathfrak{c}_0 \mathfrak{c}_1\{{(1-\mathfrak{c}_2) \over d} - {\mathfrak{c}_2 \over 2}\}}J = o(1). \label{eq:Ass26-3}
\end{align}
Then Assumption \ref{Ass:asy-norm-trend} holds from (\ref{eq:Ass26-1}), (\ref{eq:Ass26-4}), (\ref{eq:Ass26-2}), and (\ref{eq:Ass26-3}). We can also verify Assumption \ref{Ass:ze-lam} and $A_n^{1/2}J^{-r/d} =o(1)$, which implies $A_n^{1/2}(\min_{1 \leq \ell \leq L}\|\tilde{\psi}_J(\bm{z}_\ell)\|)^{-1}\inf_{m \in \Psi_J}\|m_0 - m\|_\infty = o(1)$. Therefore, we obtain the desired result.

\section{Proofs for Section \ref{Sec:reg-covariate}}\label{Appendix:sec3}

\subsection{Proof of Proposition \ref{prp:variance2}}

Define $\check{B}_{J}^{(w)} = \E[b_J^{(w)}(\bm{S}_1/A_n, \bm{X}(\bm{S}_1))b_J^{(w)}(\bm{S}_1/A_n, \bm{X}(\bm{S}_1))']$ and 
\[
\bar{\mathfrak{m}}(\bm{z},\bm{x}) = \tilde{b}_J^{(w)}(\bm{z},\bm{x})'(\tilde{B}_{J,n}^{(w)'}\tilde{B}_{J,n}^{(w)}/n + \varsigma_{J,n} (\check{B}_{J}^{(w)})^{-1})^{-1} \tilde{B}_{J,n}^{(w)'}\mathfrak{M}_0/n.
\]
Note that 
\begin{align*}
\hat{\mathfrak{m}}(\bm{z},\bm{x}) - \tilde{\mathfrak{m}}(\bm{z},\bm{x}) &= \tilde{b}_J^{(w)}(\bm{z},\bm{x})'(\tilde{B}_{J,n}^{(w)'}\tilde{B}_{J,n}^{(w)}/n + \varsigma_{J,n} (\check{B}_{J}^{(w)})^{-1})^{-1} \tilde{B}_{J,n}^{(w)'}\nu_n/n\\
&\quad + \tilde{b}_J^{(w)}\!(\bm{z},\bm{x})'\!\left\{(\tilde{B}_{J,n}^{(w)'}\tilde{B}_{J,n}^{(w)}/n \!+\! \varsigma_{J,n} (\check{B}_{J}^{(w)})^{-1})^{-1} \!-  (\tilde{B}_{J,n}^{(w)'}\tilde{B}_{J,n}^{(w)}/n)^-\right\}\!\tilde{B}_{J,n}^{(w)'}\mathfrak{M}_0/n,
\end{align*}
where $\nu_n = (\nu_{n,1},\dots,\nu_{n,n})'$. Let $\mathcal{U}_n$ be the smallest subset of $R_0$ such that for each $\bm{z} \in R_0$ there exists a $\bm{z}_n \in R_0$ with $\|\bm{z} - \bm{z}_n\| \leq \eta_6 n^{-\eta_5}$ for some $\eta_5,\eta_6>0$. Applying Lemma \ref{lem:B} and almost the same argument in (Step 1) of the proof of Proposition \ref{prp:variance}, we have
\begin{align*}
&\Prob\left(\|\hat{\mathfrak{m}} - \bar{\mathfrak{m}}\|_\infty \geq 8C\zeta_{J,n}\lambda_{J,n}\sqrt{{\log n \over n}}\right)\\
&\leq \! \Prob\left(\!\max_{(\bm{z}_n,\bm{x}_n) \in \mathcal{U}_n \times D_n}\!\left|\tilde{b}_J^{(w)}\!(\bm{z}_n,\bm{x}_n)'\!\left\{\!\left(\!{\tilde{B}_{J,n}^{(w)'}\tilde{B}_{J,n}^{(w)} \over n} \!+ \varsigma_{J,n} (\check{B}_{J}^{(w)})^{-1}\!\!\right)^{-1} \!\!\!\!\!\!-\! I_J \!\right\} \!{\tilde{B}_{J,n}^{(w)}\nu_n \over n} \right| \! \geq \! 2C\zeta_{J,n}\lambda_{J,n}\!\sqrt{{\log n \over n}}\right)\\
&\quad + \Prob\left(\max_{(\bm{z}_n,\bm{x}_n) \in \mathcal{U}_n \times D_n}\left|\tilde{b}_J^{(w)}(\bm{z}_n,\bm{x}_n)'\tilde{B}_{J,n}^{(w)'}\nu_n/n \right| \geq 2C\zeta_{J,n}\lambda_{J,n}\sqrt{{\log n \over n}}\right) + o(1)\\
&=: \Upsilon_{n,1} + \Upsilon_{n,2} + o(1). 
\end{align*}
Applying almost the same argument to control $P_{n,1}$ and $P_{n,2}$ in the proof of Proposition \ref{prp:variance}, we have $\Upsilon_{n,1} = o(1)$ and $\Upsilon_{n,2} = o(1)$. Further, let $\mathcal{V}_n$ be the smallest subset of $R_0$ such that for each $\bm{z} \in R_0$ there exists a $\bm{z}_n \in R_0$ with $\|\bm{z} - \bm{z}_n\| \leq \eta_8 n^{-\eta_7}$ for some $\eta_7,\eta_8>0$. Applying Lemma \ref{lem:B} and almost the same argument in (Step 4) of the proof of Proposition \ref{prp:variance}, we have
\begin{align*}
\Prob\left(\|\bar{\mathfrak{m}} - \tilde{\mathfrak{m}}\|_\infty \geq 4C\varsigma_{J,n} \zeta_{J,n}^2\|\mathfrak{m}_0\|_{w,\infty} \right) &= \Prob\left(\max_{\bm{z}_n \in \mathcal{V}_n}|\bar{\mathfrak{m}}(\bm{z}_n) - \tilde{\mathfrak{m}}(\bm{z}_n)| \geq 2C\varsigma_{J,n} \zeta_{J,n}^2\|\mathfrak{m}_0\|_{w,\infty}  \right) + o(1)\\
&=: \Upsilon_{n,3} + o(1). 
\end{align*}
Applying almost the same argument to control $P_{n,3}$ in the proof of Proposition \ref{prp:variance}, we have $\Upsilon_{n,3} = o(1)$.

\subsection{Proof of Corollary \ref{cor:unif-rate2}}

Analogous to the proof of Corollary \ref{cor:unif-rate1}. 

\subsection{Proof of Theorem \ref{thm:unif-rate2}}

Analogous to the proof of Theorem \ref{thm:unif-rate1}.

\subsection{Proof of Proposition \ref{prp:L^2-rate2}}

Analogous to the proof of Proposition \ref{prp:L^2-rate1}. 

\subsection{Proof of Corollary \ref{cor:L^2-rate2}}

Analogous to the proof of Corollary \ref{cor:L^2-rate1}.

\subsection{Proof of Theorem \ref{thm:asy-norm2}}
Let $(\bm{z},\bm{x})=(\bm{z}_1,\bm{x}_1)$. We prove 
\begin{align*}
{\sqrt{n}(\hat{\mathfrak{m}}(\bm{z},\bm{x}) - \mathfrak{m}_0(\bm{z},\bm{x})) \over \sqrt{V_J(\bm{z},\bm{x})}} &\stackrel{d}{\to} N(0,1).
\end{align*}
For the multivariate case (i.e., $L \geq 2$), the desired result follows from the Cam\'er-Wold device. 

Note that 
\begin{align*}
&{\sqrt{n}(\hat{\mathfrak{m}}(\bm{z},\bm{x}) - \mathfrak{m}_0(\bm{z},\bm{x})) \over \sqrt{V_J(\bm{z},\bm{x})}}\\ 
&\quad= {\sqrt{n}(\hat{\mathfrak{m}}(\bm{z},\bm{x}) - \bar{\mathfrak{m}}(\bm{z},\bm{x})) \over \sqrt{V_J(\bm{z},\bm{x})}} + {\sqrt{n}(\bar{\mathfrak{m}}(\bm{z},\bm{x}) - \tilde{\mathfrak{m}}(\bm{z},\bm{x})) \over \sqrt{V_J(\bm{z},\bm{x})}} + {\sqrt{n}(\tilde{\mathfrak{m}}(\bm{z},\bm{x}) - \mathfrak{m}_0(\bm{z},\bm{x})) \over \sqrt{V_J(\bm{z},\bm{x})}}.
\end{align*}

Observe that 
\begin{align}\label{eq:V_J}
V_J(\bm{z},\bm{x}) &\sim \tilde{b}_J^{(w)}(\bm{z},\bm{x})'\E\left[\tilde{b}_J^{(w)}\left({\bm{S}_1 \over A_n},\bm{X}(\bm{S}_1)\right)\tilde{b}_J^{(w)}\left({\bm{S}_1 \over A_n},\bm{X}(\bm{S}_1)\right)'\right]\tilde{b}_J^{(w)}(\bm{z},\bm{x}) \nonumber \\
&= \tilde{b}_J^{(w)}(\bm{z},\bm{x})'\tilde{b}_J^{(w)}(\bm{z},\bm{x}) = \|\tilde{b}_J^{(w)}(\bm{z},\bm{x})\|^2.
\end{align}
Condition (b) yields that 
\[
\left|{\sqrt{n}(\tilde{\mathfrak{m}}(\bm{z},\bm{x}) - \mathfrak{m}_0(\bm{z},\bm{x})) \over \sqrt{V_J(\bm{z},\bm{x})}}\right| = O_p\left(\sqrt{n}\|\tilde{b}_J^{(w)}(\bm{z},\bm{x})\|^{-1}\inf_{\mathfrak{m} \in B_J^{(w)}}\|\mathfrak{m}_0 - \mathfrak{m}\|_\infty\right) = o_p(1). 
\]
Condition (c) yields that 
\[
\left|{\sqrt{n}(\bar{\mathfrak{m}}(\bm{z},\bm{x}) - \tilde{\mathfrak{m}}(\bm{z},\bm{x})) \over \sqrt{V_J(\bm{z},\bm{x})}}\right| = O_p\left(\|\tilde{b}_J^{(w)}(\bm{z},\bm{x})\|^{-1}\varsigma_{J,n} \zeta_{J,n}^2\|\mathfrak{m}_0\|_{w,\infty}\sqrt{n}\right) = o_p(1). 
\]
Then it is sufficient to show that 
\begin{align}\label{eq:asy-norm-db2}
{\sqrt{n}(\hat{\mathfrak{m}}(\bm{z},\bm{x}) - \bar{\mathfrak{m}}(\bm{z},\bm{x})) \over \sqrt{V_J(\bm{z},\bm{x})}} \stackrel{d}{\to} N(0,1). 
\end{align}

Now we show (\ref{eq:asy-norm-db2}). Recall that 
\[
\hat{\mathfrak{m}}(\bm{z},\bm{x}) - \bar{\mathfrak{m}}(\bm{z},\bm{x}) = \tilde{b}_J^{(w)}(\bm{z},\bm{x})'(\tilde{B}_{J,n}^{(w)'}\tilde{B}_{J,n}^{(w)}/n + \varsigma_{J,n} (\check{B}_J^{(w)})^{-1})^{-1} \tilde{B}_{J,n}^{(w)'}\nu_n/n.
\]
Decompose 
\begin{align*}
{\sqrt{n}(\hat{\mathfrak{m}}(\bm{z},\bm{x}) - \bar{\mathfrak{m}}(\bm{z},\bm{x})) \over \sqrt{V_J(\bm{z},\bm{x})}} &= {1 \over \sqrt{nV_J(\bm{z},\bm{x})}}\tilde{b}_J^{(w)}(\bm{z},\bm{x})'\tilde{B}_{J,n}^{(w)'}\nu_n \\
&\quad + {1 \over \sqrt{nV_J(\bm{z},\bm{x})}}\tilde{b}_J^{(w)}(\bm{z},\bm{x})'\!\left\{(\tilde{B}_{J,n}^{(w)'}\tilde{B}_{J,n}^{(w)}/n + \varsigma_{J,n} (\check{B}_J^{(w)})^{-1})^{-1} \!-I_J\right\}\!\tilde{B}_{J,n}^{(w)'}\nu_n\\
&=: T_{n,1} + T_{n,2}. 
\end{align*}

For $T_{n,2}$, observe that 
\begin{align}\label{ineq:T2-bound}
|T_{n,2}| &\leq \sqrt{{n \over V_J(\bm{z},\bm{x})}}\|\tilde{b}_J^{(w)}(\bm{z},\bm{x})\| \|(\tilde{B}_{J,n}^{(w)'}\tilde{B}_{J,n}^{(w)}/n + \varsigma_{J,n} (\check{B}_J^{(w)})^{-1})^{-1} -I_J\| \|\tilde{B}_{J,n}^{(w)'}\nu_n/n\|.
\end{align}
From Lemma \ref{lem:B} and a similar argument in Steps 1 and 2 of the proof of Proposition \ref{prp:variance}, we can show that 
\begin{align}
\|(\tilde{B}_{J,n}^{(w)'}\tilde{B}_{J,n}^{(w)}/n  + \varsigma_{J,n} (\check{B}_J^{(w)})^{-1})^{-1} -I_J\| &= O_p\left(\|(\tilde{B}_{J,n}^{(w)'}\tilde{B}_{J,n}^{(w)}/n + \varsigma_{J,n} (\check{B}_J^{(w)})^{-1}) -I_J\|\right) \nonumber \\ 
&= O_p\left(\zeta_{J,n}^2 \lambda_{J,n}^2 \sqrt{{\log J \over A_n}}\right) = o_p((\zeta_{J,n} \lambda_{J,n})^{-1}), \label{eq:BB'-approx-rate}\\
\|\tilde{B}_{J,n}^{(w)'}\nu_n/n\| &= O_p\left({\zeta_{J,n}\lambda_{J,n} \over \sqrt{n}}\right). \label{eq:Bv-rate}
\end{align}
Combining (\ref{eq:V_J}), (\ref{ineq:T2-bound}), (\ref{eq:BB'-approx-rate}), and (\ref{eq:Bv-rate}), we have
\begin{align}\label{eq:T2-asy-neg}
|T_{n,2}| &= o_p(1). 
\end{align}
For $T_{n,1}$, observe that 
\begin{align*}
T_{n,1} = {1 \over \sqrt{n}}\sum_{i=1}^{n}u_J\left({\bm{S}_i \over A_n},\bm{X}(\bm{S}_i)\right)\varepsilon_i,
\end{align*}
where $u_J(\bm{S}_i/A_n,\bm{X}(\bm{S}_i)) =  V_J(\bm{z},\bm{x})^{-1/2}\tilde{b}_J^{(w)}(\bm{z},\bm{x})'\tilde{b}_J^{(w)}(\bm{S}_i/A_n,\bm{X}(\bm{S}_i))\eta(\bm{S}_i/A_n,\bm{X}(\bm{S}_i))$. 
Note that $|u_J(\bm{S}_i/A_n,\bm{X}(\bm{S}_i))| \leq C_\mathfrak{h}\zeta_{J,n}\lambda_{J,n}$, where $C_\mathfrak{h} = \sup_{(\bm{z},\bm{x}) \in R_0 \times \mathbb{R}^p}\mathfrak{h}(\bm{z},\bm{x})$. Let $(\Omega,\mathcal{F},\Prob)$ denote the probability space that $\{\bm{S}_i\}_{i=1}^n$, $\{\varepsilon_i\}_{i=1}^n$, and $\bm{X} = \{\bm{X}(\bm{s}):\bm{s} \in \mathbb{R}^d\}$ are defined. For $t \geq 1$, let $\mathcal{F}_t = \sigma(\{\bm{S}_i,\varepsilon_i\}_{i=1}^{t} \cup \bm{X})$, the $\sigma$-field generated by $\{\bm{S}_i,\varepsilon_i\}_{i=1}^{t}$ and $\bm{X}$, and for $t=0$, let $\mathcal{F}_0 = \sigma(\bm{X})$. Then we can see that $\{u_J\left({\bm{S}_i \over A_n},\bm{X}(\bm{S}_i)\right)\varepsilon_i\}$ is a martingale difference sequence with respect to $\{\mathcal{F}_t\}$. Therefore, to show (\ref{eq:asy-norm-db2}), we use a martingale CLT (Corollary 2.8 in \cite{Mc74}). This requires to verify
\begin{itemize}
\item[(A)] $\max_{1 \leq i \leq n}|u_J(\bm{S}_i/A_n,\bm{X}(\bm{S}_i))\varepsilon_i|/\sqrt{n}=o_p(1)$,
\item[(B)] ${1 \over n}\sum_{i=1}^n u_J(\bm{S}_i/A_n,\bm{X}(\bm{S}_i))^2\varepsilon_i^2 \stackrel{p}{\to} 1$. 
\end{itemize}

Now we verify Condition (A). Let $\eta>0$ be arbitrary. Then we have
\begin{align*}
&\Prob\left(\max_{1 \leq i \leq n}\left|u_J\left({\bm{S}_i \over A_n},\bm{X}(\bm{S}_i)\right)\varepsilon_i\right|/\sqrt{n} > \eta\right)\\ 
&\leq n\Prob\left(\left|u_J\left({\bm{S}_1 \over A_n},\bm{X}(\bm{S}_1)\right)\varepsilon_1\right|/\sqrt{n} > \eta\right)\\
&\leq {n \over n\eta^2}\E\left[u_J\left({\bm{S}_1 \over A_n},\bm{X}(\bm{S}_1)\right)^2\varepsilon_1^21\left\{\left|u_J\left({\bm{S}_1 \over A_n},\bm{X}(\bm{S}_1)\right)\varepsilon_1\right|/\sqrt{n} > \eta\right\}\right]\\
&\leq {C_\mathfrak{h}^2\zeta_{J,n}^2 \lambda_{J,n}^2 \over \eta}\E\left[\varepsilon_1^2 1\left\{|\varepsilon_1| > \eta\sqrt{n}/(C_\mathfrak{h}\zeta_{J,n}\lambda_{J,n})\right\}\right]\\
&\leq {(C_\mathfrak{h}\zeta_{J,n} \lambda_{J,n})^q \over \eta^q n^{{q-2 \over 2}}}\E\left[|\varepsilon_1|^q 1\left\{|\varepsilon_1| > \eta\sqrt{n}/(C_\mathfrak{h}\zeta_{J,n}\lambda_{J,n})\right\}\right] = o(1). 
\end{align*}

Now we verify Condition (B). Applying Lemma \ref{lem:cov-matrix}, we have that 
\begin{align*}
\left|{1 \over n}\sum_{i=1}^n u_J\left({\bm{S}_i \over A_n},\bm{X}(\bm{S}_i)\right)^2\varepsilon_i^2-1\right| &= \left|\left({\tilde{b}_J^{(w)}(\bm{z},\bm{x}) \over \sqrt{V_J(\bm{z},\bm{x})}}\right)'(\hat{\Sigma} - \Sigma)\left({\tilde{b}_J^{(w)}(\bm{z},\bm{x}) \over \sqrt{V_J(\bm{z},\bm{x})}}\right)\right|\\
&\leq \left\|{\tilde{b}_J^{(w)}(\bm{z},\bm{x}) \over \sqrt{V_J(\bm{z},\bm{x})}}\right\|^2 \|\hat{\Sigma} - \Sigma\| \lesssim \|\hat{\Sigma} - \Sigma\| = o_p(1). 
\end{align*}
Therefore, we obtain (\ref{eq:asy-norm-db2}). 

\subsection{Proof of Proposition \ref{prp:var-est-reg}}

Note that $\hat{V}(\bm{z}_1,\bm{x}_1,\bm{z}_2,\bm{x}_2) = \tilde{b}_J^{(w)}(\bm{z}_1,\bm{x}_1)'\tilde{H}_J\tilde{b}_J^{(w)}(\bm{z}_2,\bm{x}_2)$ where 
\begin{align*}
\tilde{H}_J &\!=\! \left(\!{\tilde{B}_{J,n}^{(w)'}\tilde{B}_{J,n}^{(w)} \over n} + \varsigma_{J,n} (\check{B}_{J}^{(w)})^{-1}\!\!\right)^{-1}\left\{{1 \over n}\sum_{i=1}^n \tilde{b}_J^{(w)}\!\!\left({\bm{S}_i \over A_n},\bm{X}(\bm{S}_i)\!\right)\tilde{b}_J^{(w)}\!\!\left({\bm{S}_i \over A_n},\bm{X}(\bm{S}_i)\!\right)'\!\! \right.\\
&\left.\quad \quad \times \left(Y(\bm{S}_i) - \hat{\mathfrak{m}}\left({\bm{S}_i \over A_n},\bm{X}(\bm{S}_i)\right)\!\right)^2\right\}\left(\!{\tilde{B}_{J,n}^{(w)'}\tilde{B}_{J,n}^{(w)} \over n} + \varsigma_{J,n} (\check{B}_{J}^{(w)})^{-1}\!\!\right)^{-1}.
\end{align*}
Then it suffices to show that as $n,J \to \infty$, 
\begin{align*}
&\Pi_{n,J}(\bm{z}_1,\bm{x}_1,\bm{z}_2,\bm{x}_2)\\
&:= \left({\tilde{b}_J^{(w)}(\bm{z}_1,\bm{x}_1) \over \|\tilde{b}_J^{(w)}(\bm{z}_1,\bm{x}_1)\|}\right)'\tilde{H}_J\left({\tilde{b}_J^{(w)}(\bm{z}_2,\bm{x}_2) \over \|\tilde{b}_J^{(w)}(\bm{z}_2,\bm{x}_2)\|}\right) = \left({\tilde{b}_J^{(w)}(\bm{z}_1,\bm{x}_1) \over \|\tilde{b}_J^{(w)}(\bm{z}_1,\bm{x}_1)\|}\right)'H_J\left({\tilde{b}_J^{(w)}(\bm{z}_2,\bm{x}_2) \over \|\tilde{b}_J^{(w)}(\bm{z}_2,\bm{x}_2)\|}\right) + o_p(1). 
\end{align*}
Now we restrict our attention to the case $(\bm{z}_1,\bm{x}_1) = (\bm{z}_2,\bm{x}_2) = (\bm{z},\bm{x})$. The proofs for other cases are similar. Define $\Pi_{n,J}(\bm{z},\bm{x}):= \Pi_{n,J}(\bm{z},\bm{x},\bm{z},\bm{x})$. Applying Theorem \ref{thm:unif-rate2}, we have
\begin{align*}
\Pi_{n,J}(\bm{z},\bm{x}) &\!=\! \left({\tilde{b}_J^{(w)}(\bm{z},\bm{x}) \over \|\tilde{b}_J^{(w)}(\bm{z},\bm{x})\|}\right)'\!\!\!\left({1 \over n}\sum_{i=1}^{n}\tilde{b}_J^{(w)}\!\!\left({\bm{S}_i \over A_n},\bm{X}(\bm{S}_i)\!\right)\!\tilde{b}_J^{(w)}\!\!\left({\bm{S}_i \over A_n},\bm{X}(\bm{S}_i)\!\right)'\!\!\varepsilon_i^2\right)\!\!\left({\tilde{b}_J^{(w)}(\bm{z},\bm{x}) \over \|\tilde{b}_J^{(w)}(\bm{z},\bm{x})\|}\right)' \!\!+ o_p(1).
\end{align*}
Then Lemma \ref{lem:cov-matrix} yields the desired result. 

\section{Auxiliary lemmas}\label{Appendix:lemmas}

\begin{lemma}[\cite{Tr12}]\label{lem:tr12}
Let $\{\Theta_{n,i}\}_{i=1}^n$ be a finite sequence of independent random matrices with dimensions $d_1 \times d_2$. Assume $\E[\Theta_{n,i}]=0$ for each $i$ and $\max_{1\leq i \leq n}\|\Theta_{n,i}\|\leq M_n$, and define
\[
\sigma_n^2 = \max\left\{\left\|\sum_{i=1}^n\E[\Theta_{n,i} \Theta'_{n,i}]\right\|, \left\|\sum_{i=1}^n\E[\Theta'_{n,i} \Theta_{n,i}]\right\|\right\}.
\]
Then for all $t \geq 0$,
\[
\Prob\left(\left\|\sum_{i=1}^n \Theta_{n,i}\right\| \geq t\right) \leq (d_1 +d_2)\exp\left({-t^2/2 \over \sigma_n^2 + M_n t/3}\right).
\]
\end{lemma}

\begin{corollary}\label{cor:tr12}
Under the conditions of Lemma \ref{lem:tr12}, if $M_n \sqrt{\log(d_1+d_2)} = o(\sigma_n)$ then
\[
\left\|\sum_{i=1}^n \Theta_{n,i}\right\| = O_p\left(\sigma_n \sqrt{\log(d_1 + d_2)}\right).
\]
\end{corollary}
\begin{proof}
Letting $t=C\sigma_n \sqrt{\log(d_1+d_2)}$ for sufficiently large $C>1$, we have
\begin{align*}
\Prob\left(\left\|\sum_{i=1}^n \Theta_{n,i}\right\| \geq C\sigma_n \sqrt{\log(d_1+d_2)}\right) &\leq (d_1 + d_2)\exp\left({-C^2\sigma_n^2 \log(d_1 + d_2) \over \sigma_n^2 + CM_n\sigma_n \sqrt{(d_1 + d_2)}/3}\right)\\
&= (d_1 + d_2)\exp\left({-C^2 \log(d_1 + d_2) \over 1 + CM_n\sqrt{(d_1 + d_2)}/(3\sigma_n)}\right)\\
&\lesssim (d_1 + d_2)\exp\left(-\log(d_1 + d_2)^{C^2}\right) = (d_1 + d_2)^{1-C^2},  
\end{align*}
which yields the desired result. 
\end{proof}

\begin{lemma}\label{lem:psi}
Under Assumptions \ref{Ass:sample}(ii) and \ref{Ass:sieve1}(iii), we have
\[
\|\tilde{\Psi}'_{J,n}\tilde{\Psi}_{J,n}/n - I_J\| = O_p\left(\zeta_J\lambda_J\sqrt{\log J \over n}\right).
\]
\end{lemma}
\begin{proof}
Recall $\check{\Psi}_J = \E\left[\psi_J(\bm{S}_i/A_n)\psi_J(\bm{S}_i/A_n)'\right]$. Note that $\E[\tilde{\psi}_J(\bm{S}_i/A_n)\tilde{\psi}_J(\bm{S}_i/A_n)'] = I_J$ and 
\begin{align*}
\tilde{\Psi}'_{J,n}\tilde{\Psi}_{J,n} &= \sum_{i=1}^n \tilde{\psi}_J\left({\bm{S}_i \over A_n}\right)\tilde{\psi}_J\left({\bm{S}_i \over A_n}\right)' = \sum_{i=1}^n \check{\Psi}_J^{-1/2}\psi_J\left({\bm{S}_i \over A_n}\right)\psi_J\left({\bm{S}_i \over A_n}\right)'\check{\Psi}_J^{-1/2}.
\end{align*}
Moreover,  
\begin{align*}
\left\|\tilde{\psi}_J\left({\bm{S}_i \over A_n}\right)\tilde{\psi}_J\left({\bm{S}_i \over A_n}\right)' - I_J\right\| 
&\leq \left\|\tilde{\psi}_J\left({\bm{S}_i \over A_n}\right)\tilde{\psi}_J\left({\bm{S}_i \over A_n}\right)'\right\| + \| I_J\| \leq \left\|\check{\Psi}_J^{-1/2} \psi_J\left({\bm{S}_i \over A_n}\right)\psi_J\left({\bm{S}_i \over A_n}\right)'\check{\Psi}_J^{-1/2}\right\| + 1\\
&\leq \|\check{\Psi}_J^{-1/2}\|^2\left\|\psi_J\left({\bm{S}_i \over A_n}\right)\psi_J\left({\bm{S}_i \over A_n}\right)'\right\| + 1 \leq \zeta_J^2\lambda_J^2 + 1. 
\end{align*}
Further, 
\begin{align*}
&\E\left[\left(\tilde{\psi}_J\left({\bm{S}_i \over A_n}\right)\tilde{\psi}_J\left({\bm{S}_i \over A_n}\right)' - I_J\right)\left(\tilde{\psi}_J\left({\bm{S}_i \over A_n}\right)\tilde{\psi}_J\left({\bm{S}_i \over A_n}\right)' - I_J\right)'\right]\\
&= \E\left[\tilde{\psi}_J\left({\bm{S}_i \over A_n}\right)\tilde{\psi}_J\left({\bm{S}_i \over A_n}\right)'\tilde{\psi}_J\left({\bm{S}_i \over A_n}\right)\tilde{\psi}_J\left({\bm{S}_i \over A_n}\right)'\right] - 2\E\left[\tilde{\psi}_J\left({\bm{S}_i \over A_n}\right)\tilde{\psi}_J\left({\bm{S}_i \over A_n}\right)'\right] + I_J\\
&= \E\left[\left\|\tilde{\psi}_J\left({\bm{S}_i \over A_n}\right)\right\|^2\tilde{\psi}_J\left({\bm{S}_i \over A_n}\right)\tilde{\psi}_J\left({\bm{S}_i \over A_n}\right)'\right] - I_J. 
\end{align*}
Then we have
\begin{align*}
&\left\|\E\left[\left(\tilde{\psi}_J\left({\bm{S}_i \over A_n}\right)\tilde{\psi}_J\left({\bm{S}_i \over A_n}\right)' - I_J\right)\left(\tilde{\psi}_J\left({\bm{S}_i \over A_n}\right)\tilde{\psi}_J\left({\bm{S}_i \over A_n}\right)' - I_J\right)'\right]\right\|\\
&\leq \left\|\E\left[\left\|\tilde{\psi}_J\left({\bm{S}_i \over A_n}\right)\right\|^2\tilde{\psi}_J\left({\bm{S}_i \over A_n}\right)\tilde{\psi}_J\left({\bm{S}_i \over A_n}\right)'\right]\right\| + \|I_J\|\\
&\leq \sup_{\bm{z} \in R_0}\|\tilde{\psi}_J(\bm{z})\|^2 \left\|\E\left[\tilde{\psi}_J\left({\bm{S}_i \over A_n}\right)\tilde{\psi}_J\left({\bm{S}_i \over A_n}\right)'\right]\right\| + 1 \leq \zeta_J^2 \lambda_J^2 + 1. 
\end{align*}
Therefore, applying Corollary \ref{cor:tr12} with $M_n = (\zeta_J^2\lambda_J^2 + 1)/n$ and $\sigma_n^2 = (\zeta_J^2 \lambda_J^2 + 1)/n$, we obtain the desired result. 
\end{proof}

\begin{lemma}\label{lem:B}
Suppose that Assumptions \ref{Ass:sample}(ii), \ref{Ass:model2}(i), (ii), and \ref{Ass:sieve3}(iii) hold. Assume that $\bm{X}=\{\bm{X}(\bm{s}):\bm{s} \in \mathbb{R}^d\}$ and $\{\bm{S}_i\}_{i=1}^{n}$ are mutually independent. Then
\[
\|\tilde{B}_{J,n}^{(w)'}\tilde{B}_{J,n}^{(w)}/n - I_J\| = O_p\left(\zeta_{J,n}^2\lambda_{J,n}^2\sqrt{\log J \over A_n}\right).
\]
\end{lemma}

\begin{proof}
Define $\check{B}_J^{(w)}=\E[b_J^{(w)}(\bm{S}_i/A_n,\bm{X}(\bm{S}_1))b_J^{(w)}(\bm{S}_i/A_n,\bm{X}(\bm{S}_1))']$. Note that 
\begin{align*}
\tilde{B}_{J,n}^{(w)'}\tilde{B}_{J,n}^{(w)} &= \sum_{i=1}^n \tilde{b}_J^{(w)}\left({\bm{S}_i \over A_n},\bm{X}(\bm{S}_i)\right)\tilde{b}_J^{(w)}\left({\bm{S}_i \over A_n},\bm{X}(\bm{S}_i)\right)'\\
&= \sum_{i=1}^n (\check{B}_J^{(w)})^{-1/2}b_J^{(w)}\left({\bm{S}_i \over A_n},\bm{X}(\bm{S}_i)\right)b_J^{(w)}\left({\bm{S}_i \over A_n},\bm{X}(\bm{S}_i)\right)'(\check{B}_J^{(w)})^{-1/2}
\end{align*}
and $\E[\tilde{b}_J^{(w)}\left(\bm{S}_i/A_n,\bm{X}(\bm{S}_1)\right)\tilde{b}_J^{(w)}\left(\bm{S}_i/A_n,\bm{X}(\bm{S}_1)\right)'] = I_J$. Observe that
\begin{align*}
\left\|\tilde{b}_J^{(w)}\left({\bm{S}_i \over A_n},\bm{X}(\bm{S}_i)\right)\tilde{b}_J^{(w)}\left({\bm{S}_i \over A_n},\bm{X}(\bm{S}_i)\right)' - I_J\right\| \lesssim \zeta_{J,n}^2\lambda_{J,n}^2 + 1.
\end{align*}
Further, we have
\begin{align*}
\tilde{B}_{J,n}^{(w)'}\tilde{B}_{J,n}^{(w)}/n - I_J &= {1 \over n}\sum_{i=1}^n \tilde{b}_J^{(w)}\left({\bm{S}_i \over A_n},\bm{X}(\bm{S}_i)\right)\tilde{b}_J^{(w)}\left({\bm{S}_i \over A_n},\bm{X}(\bm{S}_i)\right)' - I_J\\
&= \sum_{\bm{\ell} \in L_{n1}}\!\!\!\!b_1^{(\bm{\ell};\bm{\Delta}_0)} \!+\! \sum_{\bm{\Delta} \neq \bm{\Delta}_0}\sum_{\bm{\ell} \in L_{n1}}\!\!\!\!b_1^{(\bm{\ell};\bm{\Delta})} \!+ \!\sum_{\bm{\Delta} \in \{1,2\}^d}\sum_{\bm{\ell} \in L_{n2}}\!\!\!\!b_1^{(\bm{\ell};\bm{\Delta})},
\end{align*}
where 
\[
b_1^{(\bm{\ell};\bm{\Delta})} = {1 \over n}\sum_{i=1}^{n}\left\{\tilde{b}_J^{(w)}\left({\bm{S}_i \over A_n},\bm{X}(\bm{S}_i)\right)\tilde{b}_J^{(w)}\left({\bm{S}_i \over A_n},\bm{X}(\bm{S}_i)\right)' - I_J\right\}1\{\bm{S}_i \in \Gamma_n(\bm{\ell};\bm{\Delta}) \cap R_n \}. 
\]
For $\bm{\Delta} \in \{1,2\}^d$, let $\{\tilde{b}_1^{(\bm{\ell};\bm{\Delta})}\}_{\bm{\ell} \in L_{n1} \cup L_{n2}}$ be independent random variables such that $b_1^{(\bm{\ell};\bm{\Delta})}$ and $\tilde{b}_1^{(\bm{\ell};\bm{\Delta})}$ have the same distribution. Applying Lemma \ref{lem:indep_block} below with $M_h = 1$, $m \sim A_n(A_n^{(1)})^{-1}$, and $\tau \sim \beta(\underline{A}_{n2}; A_n)$, we have that for $\bm{\Delta} \in \{1,2\}^d$, 
\begin{align}
&\sup_{t >0}\left|\Prob\left(\left|\sum_{\bm{\ell} \in L_{n1}}b_1^{(\bm{\ell};\bm{\Delta})}\right|>t\right) - \Prob\left(\left|\sum_{\bm{\ell} \in L_{n1}}\tilde{b}_1^{(\bm{\ell};\bm{\Delta})}\right|>t\right)\right| \lesssim \left({A_n \over A_n^{(1)}}\right)\beta(\underline{A}_{n2}; A_n) = o(1), \label{ineq:b1-beta-block1}\\
&\sup_{t >0}\left|\Prob\left(\left|\sum_{\bm{\ell} \in L_{n2}}b_1^{(\bm{\ell};\bm{\Delta})}\right|>t\right) - \Prob\left(\left|\sum_{\bm{\ell} \in L_{n2}}\tilde{b}_1^{(\bm{\ell};\bm{\Delta})}\right|>t\right)\right| \lesssim \left({A_n  \over A_n^{(1)}}\right)\beta(\underline{A}_{n2}; A_n ) = o(1) \label{ineq:b1-beta-block2}.
\end{align}
From (\ref{ineq:b1-beta-block1}) and (\ref{ineq:b1-beta-block2}), we have
\begin{align*}
&\Prob\left(\left\|\tilde{B}_{J,n}^{(w)'}\tilde{B}_{J,n}^{(w)}/n - I_J\right\|>C2^{d+2}\zeta_{J,n}^2\lambda_{J,n}^2\sqrt{{\log n \over A_n}}\right)\\ 
&\leq \sum_{\bm{\Delta} \in \{1,2\}^{d}}\hat{\mathcal{Q}}_{n1}(\bm{\Delta}) + \sum_{\bm{\Delta} \in \{1,2\}^{d}}\hat{\mathcal{Q}}_{n2}(\bm{\Delta}) + 2^{d+1}\underbrace{\left({A_n \over A_n^{(1)}}\right)\beta(\underline{A}_{n2};A_n)}_{=o(1)},
\end{align*}
where 
\begin{align*}
\hat{\mathcal{Q}}_{nj}(\bm{\Delta}) &= \Prob\left(\left\|\sum_{\bm{\ell} \in L_{nj}}\tilde{b}_1^{(\bm{\ell};\bm{\Delta})}\right\|>Cn\zeta_{J,n}^2\lambda_{J,n}^2\sqrt{{\log n \over A_n}}\right),\ j=1,2.
\end{align*}

Now we restrict our attention to show that $\hat{\mathcal{Q}}_{n1}(\bm{\Delta}) = o(1)$ for $\bm{\Delta} \neq \bm{\Delta}_{0}$. The proofs for other cases are similar. Note that
\begin{align*}
\left\|\tilde{b}_1^{(\bm{\ell};\bm{\Delta})}\right\| &\lesssim A_n^{-1} A_n^{(1)}(\zeta_{J,n}^2\lambda_{J,n}^2+1)\ a.s.\ (\text{from Lemma \ref{lem:number-summands}}).
\end{align*}
Further, observe that
\begin{align*}
\left\|\sum_{\bm{\ell} \in L_{n1}}\E\left[\tilde{b}_1^{(\bm{\ell};\bm{\Delta})}\tilde{b}_1^{(\bm{\ell};\bm{\Delta})'}\right]\right\| 
&\lesssim {n \over n^2}\E\left[\left\|\left\{\tilde{b}_J^{(w)}\left({\bm{S}_1 \over A_n},\bm{X}(\bm{S}_1)\right)\tilde{b}_J^{(w)}\left({\bm{S}_1 \over A_n},\bm{X}(\bm{S}_1)\right)' - I_J\right\} \right. \right. \\
&\left. \left. \quad \times \left\{\tilde{b}_J^{(w)}\left({\bm{S}_1 \over A_n},\bm{X}(\bm{S}_1)\right)\tilde{b}_J^{(w)}\left({\bm{S}_1 \over A_n},\bm{X}(\bm{S}_1)\right)' - I_J\right\}\right\|\right]\\
&\quad + {n(n-1) \over n^2}\E\left[\left\|\left\{\tilde{b}_J^{(w)}\left({\bm{S}_1 \over A_n},\bm{X}(\bm{S}_1)\right)\tilde{b}_J^{(w)}\left({\bm{S}_1 \over A_n},\bm{X}(\bm{S}_1)\right)' - I_J\right\} \right. \right. \\
&\left. \left. \quad \quad \times \left\{\tilde{b}_J^{(w)}\left({\bm{S}_2 \over A_n},\bm{X}(\bm{S}_2)\right)\tilde{b}_J^{(w)}\left({\bm{S}_2 \over A_n},\bm{X}(\bm{S}_2)\right)' - I_J\right\}'\right\|\right] =: W_{n,1} + W_{n,2}. 
\end{align*}
For $W_{n,1}$, we have $W_{n,1} \lesssim (\zeta_{J,n}^4\lambda_{J,n}^4+1)/n$. For $W_{n,2}$, letting $\bm{Z}_i = A_n\bm{S}_i$, $\bm{W} = \bm{S}_1 - \bm{S}_2$, we have
\begin{align*}
W_{n,2} &\lesssim \E\left[\left\|\left\{\tilde{b}_J^{(w)}\left({\bm{S}_1 \over A_n},\bm{X}(\bm{0})\right)\tilde{b}_J^{(w)}\left({\bm{S}_1 \over A_n},\bm{X}(\bm{0})\right)' - I_J\right\} \right. \right. \\
&\left. \left. \quad \quad \times \left\{\tilde{b}_J^{(w)}\left({\bm{S}_2 \over A_n},\bm{X}(\bm{S}_1 - \bm{S}_2)\right)\tilde{b}_J^{(w)}\left({\bm{S}_2 \over A_n},\bm{X}(\bm{S}_1 - \bm{S}_2)\right)' - I_J\right\}'\right\|\right]\\
&= {1 \over A_n}\E\left[\left\|\left\{\tilde{b}_J^{(w)}\left({\bm{W} \over A_n} + \bm{Z}_2,\bm{X}(\bm{0})\right)\tilde{b}_J^{(w)}\left({\bm{W} \over A_n} + \bm{Z}_2,\bm{X}(\bm{0})\right)' - I_J\right\} \right. \right. \\
&\left. \left. \quad \quad \times \left\{\tilde{b}_J^{(w)}\left(\bm{Z}_2,\bm{X}(\bm{W})\right)\tilde{b}_J^{(w)}\left(\bm{Z}_2,\bm{X}(\bm{W})\right)' - I_J\right\}'\right\|\right] \lesssim {\zeta_{J,n}^4 \lambda_{J,n}^4 \over A_n}. 
\end{align*}
Then applying Corollary \ref{cor:tr12} with $M_n = C_1A_n^{-1}A_n^{(1)}(\zeta_{J,n}^2\lambda_{J,n}^2 + 1)$ and $\sigma_n^2 = C_2A_n^{-1}\zeta_{J,n}^4 \lambda_{J,n}^4$ for some constants $C_1, C_2>0$ independent of $n$, we obtain the desired result. 
\end{proof}

Define 
\begin{align*}
\Sigma &= \E\left[\tilde{b}_J^{(w)}\left({\bm{S}_1 \over A_n}, \bm{X}(\bm{S}_1)\right)\tilde{b}_J^{(w)}\left({\bm{S}_1 \over A_n}, \bm{X}(\bm{S}_1)\right)'\mathfrak{h}\left({\bm{S}_1 \over A_n},\bm{X}(\bm{S}_1)\right)^2 \varepsilon_1^2\right],\\
\hat{\Sigma} &= {1 \over n}\sum_{i=1}^{n}\tilde{b}_J^{(w)}\left({\bm{S}_i \over A_n}, \bm{X}(\bm{S}_i)\right)\tilde{b}_J^{(w)}\left({\bm{S}_i \over A_n}, \bm{X}(\bm{S}_i)\right)'\mathfrak{h}\left({\bm{S}_i \over A_n},\bm{X}(\bm{S}_i)\right)^2 \varepsilon_i^2.
\end{align*}

\begin{lemma}\label{lem:cov-matrix}
Suppose that Assumptions \ref{Ass:sample}(i), (ii), \ref{Ass:model2}, \ref{Ass:sieve3}(iii) hold. 

Additionally, assume $(\zeta_{J,n}\lambda_{J,n})^{{2q \over q-2}} \lesssim \sqrt{{A_n \over (\log n)^2}}$. Then $\|\hat{\Sigma} - \Sigma\| = o_p(1)$. 
\end{lemma}

\begin{proof}
Define $\hat{\Sigma}_j= {1 \over n}\sum_{i=1}^{n}(\Xi_{j,i} - \E[\Xi_{j,i}])$, $j=1,2$,
where 
\begin{align*}
\Xi_{1,i} &= \tilde{b}_J^{(w)}\left({\bm{S}_i \over A_n}, \bm{X}(\bm{S}_i)\right)\tilde{b}_J^{(w)}\left({\bm{S}_i \over A_n}, \bm{X}(\bm{S}_i)\right)'\mathfrak{h}\left({\bm{S}_i \over A_n},\bm{X}(\bm{S}_i)\right)^2 \varepsilon_i^2\\
&\quad \times 1\left\{\left\| \tilde{b}_J^{(w)}\left({\bm{S}_i \over A_n}, \bm{X}(\bm{S}_i)\right)\tilde{b}_J^{(w)}\left({\bm{S}_i \over A_n}, \bm{X}(\bm{S}_i)\right)'\varepsilon_i^2\right\| \leq L_n^2\right\},\\
\Xi_{2,i} &= \tilde{b}_J^{(w)}\left({\bm{S}_i \over A_n}, \bm{X}(\bm{S}_i)\right)\tilde{b}_J^{(w)}\left({\bm{S}_i \over A_n}, \bm{X}(\bm{S}_i)\right)'\mathfrak{h}\left({\bm{S}_i \over A_n},\bm{X}(\bm{S}_i)\right)^2 \varepsilon_i^2\\
&\quad \times 1\left\{\left\| \tilde{b}_J^{(w)}\left({\bm{S}_i \over A_n}, \bm{X}(\bm{S}_i)\right)\tilde{b}_J^{(w)}\left({\bm{S}_i \over A_n}, \bm{X}(\bm{S}_i)\right)'\varepsilon_i^2\right\| > L_n^2\right\}
\end{align*}
and where $L_n = (\zeta_{J,n}\lambda_{J,n})^{{q \over q-2}}$. Note that $\hat{\Sigma} - \Sigma = \hat{\Sigma}_1 + \hat{\Sigma}_2$. Then it is sufficient to show $\|\hat{\Sigma}_j\| = o_p(1)$, $j=1,2$. 

First, we show $\|\hat{\Sigma}_1\| = o_p(1)$. Observe that $\|\Xi_{1,i}\| \leq C_\mathfrak{h}^2L_n^2$ where $C_\mathfrak{h} = \sup_{(\bm{z},\bm{x}) \in R_0 \times \mathbb{R}^p}\mathfrak{h}(\bm{z},\bm{x})$. Then we have $\|\Xi_{1,i} - \E[\Xi_{1,i}]\| \leq 2C_{\mathfrak{h}}^2 L_n^2$. Decompose 
\begin{align*}
\hat{\Sigma}_1 &= \sum_{\bm{\ell} \in L_{n1}}\!\!\Xi_1^{(\bm{\ell};\bm{\Delta}_0)} \!+\! \sum_{\bm{\Delta} \neq \bm{\Delta}_0}\sum_{\bm{\ell} \in L_{n1}}\!\!\Xi_1^{(\bm{\ell};\bm{\Delta})} \!+ \!\sum_{\bm{\Delta} \in \{1,2\}^d}\sum_{\bm{\ell} \in L_{n2}}\!\!\Xi_1^{(\bm{\ell};\bm{\Delta})},
\end{align*}
where $\Xi_1^{(\bm{\ell};\bm{\Delta})} = {1 \over n}\sum_{i=1}^{n}\left\{\Xi_{1,i} - \E[\Xi_{1,i}]\right\}1\{\bm{S}_i \in \Gamma_n(\bm{\ell};\bm{\Delta}) \cap R_n \}$. Observe that 
\begin{align}\label{ineq:Sig1-norm}
\|\Xi_1^{(\bm{\ell};\bm{\Delta})}\| &\leq {1 \over n}\sum_{i=1}^{n}\|\Xi_{1,i} - \E[\Xi_{1,i}]\|1\{\bm{S}_i \in \Gamma_n(\bm{\ell};\bm{\Delta})\} \nonumber \\
&\leq {2C_{\mathfrak{h}}^2 L_n^2 \over n} nA_n^{-1}A_n^{(1)} = {2C_{\mathfrak{h}}^2A_n^{(1)} L_n^2 \over A_n}\ a.s.\ (\text{from Lemma \ref{lem:number-summands}})
\end{align}
and 
\begin{align}\label{ineq:Sig1-var}
&\left\|\sum_{\bm{\ell} \in L_{n1}}\E\left[\Xi_1^{(\bm{\ell};\bm{\Delta})}\Xi_1^{(\bm{\ell};\bm{\Delta})'}\right]\right\| \nonumber \\
&\leq {n \over n^2}\left\|\E\left[\left(\Xi_{1,1} - \E[\Xi_{1,1}]\right)\left(\Xi_{1,1} - \E[\Xi_{1,1}]\right)'1\{\bm{S}_1 \in \cup_{\bm{\ell} \in L_{n1}} \Gamma_n(\bm{\ell};\bm{\Delta})\}\right]\right\| \nonumber \\
&\quad + {n(n-1) \over n^2}\left\|\E\left[\left(\Xi_{1,1} - \E[\Xi_{1,1}]\right)\left(\Xi_{1,2} - \E[\Xi_{1,2}]\right)'1\{\bm{S}_1,\bm{S}_2 \in \cup_{\bm{\ell} \in L_{n1}} \Gamma_n(\bm{\ell};\bm{\Delta})\}\right]\right\| \nonumber \\
&\leq {n \over n^2}\E\left[\left\|\left(\Xi_{1,1} - \E[\Xi_{1,1}]\right)\left(\Xi_{1,1} - \E[\Xi_{1,1}]\right)'\right\| \right] + {n(n-1) \over n^2}\E\left[\left\| \left(\Xi_{1,1} - \E[\Xi_{1,1}]\right)\left(\Xi_{1,2} - \E[\Xi_{1,2}]\right)'\right\| \right]\nonumber \\
&\leq {4C_{\mathfrak{h}}^4 L_n^4 \over n} + {4C_{\mathfrak{h}}^4 L_n^4 \over A_n} \lesssim {L_n^4 \over A_n}. 
\end{align}
Therefore, applying Lemma \ref{lem:tr12} and a similar argument to the proof of Lemma \ref{lem:B} with (\ref{ineq:Sig1-norm}) and (\ref{ineq:Sig1-var}), we have $\|\hat{\Sigma}_1\| = O_p\left({L_n^2 \sqrt{\log J \over A_n}}\right) = o_p(1)$. 

Next we show $\|\hat{\Sigma}_2\| = o_p(1)$. Observe that 
\begin{align*}
\|\Xi_{2,i}\| &\leq C_{\mathfrak{h}}^2\left\|\tilde{b}_J^{(w)}\left({\bm{S}_i \over A_n}, \bm{X}(\bm{S}_i)\right)\tilde{b}_J^{(w)}\left({\bm{S}_i \over A_n}, \bm{X}(\bm{S}_i)\right)'\varepsilon_i^2\right\|\\
&\quad \times 1\left\{\left\| \tilde{b}_J^{(w)}\left({\bm{S}_i \over A_n}, \bm{X}(\bm{S}_i)\right)\tilde{b}_J^{(w)}\left({\bm{S}_i \over A_n}, \bm{X}(\bm{S}_i)\right)'\varepsilon_i^2\right\| > L_n^2\right\}\\
&\leq C_{\mathfrak{h}}^2\zeta_{J,n}^2 \lambda_{J,n}^2 |\varepsilon_i|^2 1\{\zeta_{J,n}^2 \lambda_{J,n}^2 |\varepsilon_i|^2 > L_n^2\}.
\end{align*}
Then we have
\begin{align*}
\E[\|\hat{\Sigma}_2\|] &\leq 2\E[\|\Xi_{2,i}\|] \leq 2C_{\mathfrak{h}}^2\zeta_{J,n}^2 \lambda_{J,n}^2 \E[|\varepsilon_i|^2 1\{\zeta_{J,n}^2 \lambda_{J,n}^2 |\varepsilon_i|^2 > L_n^2\}]\\
&\leq {2C_{\mathfrak{h}}^2(\zeta_{J,n} \lambda_{J,n})^2 (\zeta_{J,n} \lambda_{J,n})^{q-2} \over L_n^{q-2}}\E[|\varepsilon_1|^q 1\{|\varepsilon_i|^2 > L_n^2/(\zeta_{J,n} \lambda_{J,n})^2\}] = o(1).
\end{align*}
Hence, applying Markov's inequality, we have $\|\hat{\Sigma}_2\| = o_p(1)$.
\end{proof}

\begin{lemma}\label{lem:matrix-inv-norm}
Let $B$ and $C$ be $J \times J$ matrices. Assume that $A=B+C$, $B$, and $C$ are non-singular, and $\lambda_{\text{min}}(C^{-1}) > \lambda_{\text{max}}(C^{-1})\{\|A - I_J\|(1 + \|B - I_J\|) + \|B - I_J\|\}$. Then
\begin{align*}
\|A^{-1} - B^{-1}\| &\leq {1 \over \lambda_{\text{min}}(C^{-1}) - \lambda_{\text{max}}(C^{-1})\{\|A - I_J\|(1 + \|B - I_J\|) + \|B - I_J\|\}}. 
\end{align*}
\end{lemma}
\begin{proof}
Observe that 
\begin{align}\label{eq:AB-inv-norm}
\|A^{-1} - B^{-1}\| &= \sup_{\|x\| \neq 0}{\|(A^{-1} - B^{-1})\bm{x}\| \over \|\bm{x}\|}= \sup_{\|A\bm{x}\| \neq 0}{\|(I_J - B^{-1}A)\bm{x}\| \over \|A\bm{x}\|} \nonumber \\
&= \sup_{\|A\bm{x}\| \neq 0}{\|B^{-1}C\bm{x}\| \over \|A\bm{x}\|} = \sup_{\|AC^{-1}B\bm{y}\| \neq 0}{\|\bm{y}\| \over \|AC^{-1}B\bm{y}\|} = {1 \over \inf_{\|x\|=1}\|AC^{-1}B\bm{x}\|}.
\end{align}
Note that 
\begin{align}\label{ineq:ACB-norm}
\|AC^{-1}B\bm{x}\| &\geq \|C^{-1}B\bm{x}\| - \|(A - I_J)C^{-1}B\bm{x}\| \nonumber \\
&\geq \|C^{-1}\bm{x}\| - \|C^{-1}(B-I_J)\bm{x}\|- \|(A - I_J)C^{-1}B\bm{x}\| \nonumber \\
&\geq \{\lambda_{\text{min}}(C^{-1}) - \|C^{-1}\|\left(\|A - I_J\|(1 + \|B - I_J\|) + \|B - I_J\|\right)\}\|\bm{x}\| \nonumber \\
&\geq \{\lambda_{\text{min}}(C^{-1}) - \lambda_{\text{max}}(C^{-1})\left(\|A - I_J\|(1 + \|B - I_J\|) + \|B - I_J\|\right)\}\|\bm{x}\|.
\end{align}
Combining (\ref{eq:AB-inv-norm}) and (\ref{ineq:ACB-norm}), we obtain the desired result.
\end{proof}

We refer to the following lemmas without those proofs. 
\begin{lemma}[(5.19) in \cite{La03b}]\label{lem:number-summands}
Under Assumption 2.2, 
we have
\[
\Prob\left(\sum_{i=1}^{n}1\{\bm{X}_{i} \in \Gamma_n(\bm{\ell}; \bm{\Delta})\} > C|\Gamma_n(\bm{\ell}; \bm{\Delta})|nA_{n}^{-1}\ \text{for some $\bm{\ell} \in L_{n1}$, i.o.}\right) = 0
\]
for any $\bm{\Delta} \in \{1,2\}^{d}$, where $C>0$ is a sufficiently large constant. 
\end{lemma}

\begin{lemma}[Corollary 2.7 in \cite{Yu94}]\label{lem:indep_block}
Let $m \in \mathbb{N}$ and let $Q$ be a probability measure on a product space $(\prod_{i=1}^{m}\Omega_{i}, \prod_{i=1}^{m}\Sigma_{i})$ with marginal measures $Q_{i}$ on $(\Omega_{i}, \Sigma_{i})$. Suppose that $h$ is a bounded measurable function on the product probability space such that $|h|\leq M_{h}<\infty$. For $1 \leq a\leq b \leq m$, let $Q_{a}^{b}$ be the marginal measure on $(\prod_{i=a}^{b}\Omega_{i}, \prod_{i=a}^{b}\Sigma_{i})$. For a given $\tau>0$, suppose that, for all $1 \leq k \leq m-1$,  
\begin{align}\label{prod_bound}
\|Q - Q_{1}^{k}\times Q_{k+1}^{m}\|_{TV} \leq 2\tau,
\end{align}
where $Q_{1}^{k}\times Q_{k+1}^{m}$ is a product measure and $\| \cdot \|_{TV}$ is the total variation. Then $|Qh -  Ph| \leq 2M_{h}(m-1)\tau$ where $P=\prod_{i=1}^{m}Q_{i}$, $Qh=\int hdQ$, and $Ph=\int hdP$. 
\end{lemma}

\begin{lemma}[Bernstein's inequality]\label{lem:Bernstein}
Let $X_{1},\hdots, X_{n}$ be independent zero-mean random variables. Suppose that  $\max_{1 \leq i \leq n}|X_{i}|\leq M<\infty$ a.s. Then, for all $t>0$,
\begin{align*}
\Prob\left(\sum _{i=1}^{n}X_{i}\geq t\right)\leq \exp \left(-{{t^{2} \over 2} \over \sum_{i=1}^{n}\E[X_{i}^{2}] + {Mt \over 3}}\right). 
\end{align*}
\end{lemma}

\end{appendix}

\bibliography{Spatial-sieve-Ref}
\bibliographystyle{apalike}

\end{document}